\documentclass[10pt]{amsart}
\usepackage{amscd}
\usepackage{amsmath}
\usepackage{amsfonts}
\usepackage{amssymb}
\usepackage{amsthm}
\usepackage{enumerate}
\usepackage[T1]{fontenc}
\usepackage{hyperref}
\usepackage{mathrsfs}
\usepackage{stackrel}
\usepackage[all]{xy}
\usepackage{yhmath}

\newtheorem{thm}{Theorem}[section]
\newtheorem{lem}[thm]{Lemma}
\newtheorem{defn}[thm]{Definition}
\newtheorem{prop}[thm]{Proposition}
\newtheorem{cor}[thm]{Corollary}
\newtheorem*{rmk}{Remark}

\DeclareMathOperator{\Sp}{Sp}

\newcommand{\A}{\mathcal{A}}
\newcommand{\B}{\mathcal{B}}
\newcommand{\C}{\mathcal{C}}
\newcommand{\D}{\mathcal{D}}

\newcommand{\F}{\mathcal{F}}
\newcommand{\G}{\mathcal{G}}
\newcommand{\I}{\mathcal{I}}

\renewcommand{\L}{\mathcal{L}}

\newcommand{\M}{\mathcal{M}}

\newcommand{\N}{\mathcal{N}}

\renewcommand{\O}{\mathcal{O}}

\newcommand{\Q}{\mathcal{Q}}
\newcommand{\R}{\mathcal{R}}

\newcommand{\T}{\mathcal{T}}

\newcommand{\h}[1]{\widehat{#1}}

\newcommand{\w}[1]{\wideparen{#1}}

\begin{document}
\title[Auslander regularity of $\D_n$]{Auslander regularity of completed rings of $p$-adic differential operators}
\author{Andreas Bode}
\maketitle
\begin{abstract}
We prove that any smooth rigid analytic variety $X$ admits an affinoid covering $\{U_i\}$ such that the Banach algebras involved in the Fr\'echet--Stein presentation of the completed ring of differential operators $\w{\D}_X(U_i)$ are Auslander regular for each $i$. We use this result to prove projection formulae and adjunction results for coadmissible $\w{\D}$-modules.
\end{abstract}

\tableofcontents
\section{Introduction}
The slogan of this paper can be taken to be ``Coadmissible $\w{\D}$-modules are microlocally perfect''. We will devote most of this introduction to explaining what we mean by this.

Throughout, let $K$ be a complete nonarchimedean field of mixed characteristic $(0, p)$, with valuation ring $R$ and residue field $k$. Let $\pi\in R$ with $0<|\pi|<1$. 

Let $X=\Sp A$ be a smooth affinoid rigid analytic $K$-variety, and let $L$ denote the global sections of the tangent sheaf of $X$, i.e.
\begin{equation*}
L=\T_X(X)=\mathrm{Der}_K(A).
\end{equation*} 
In \cite{DcapOne}, Ardakov--Wadsley defined the $K$-algebra $\w{\D}_X(X)$ of rapidly converging differential operators as a Fr\'echet completion of the enveloping algebra $U_A(L)$. We briefly recall the construction here. We choose an $R$-subalgebra $\A$ of $A$ which is topologically of finite presentation and spans $A$ as a $K$-vector space (an \emph{admissible affine formal model}), as well as an $\A$-submodule $\L\subseteq L$ which is closed under the Lie bracket, spans $L$ and preserves $\A$ under the natural action via derivations (an \emph{$(R, \A)$-Lie lattice}). Then we set
\begin{equation*}
\w{D}=\w{\D}_X(X)=\varprojlim_{n\geq 0} \left(\h{U_\A(\pi^n\L)}\otimes_R K\right),
\end{equation*} 
where $U_\A(\pi^n\L)$ denotes the universal enveloping algebra of the $(R, \A)$-Lie--Rinehart algebra $\pi^n\L$ as in \cite{Rinehart}, and $\ \h{}\ $ denotes $\pi$-adic completion. For instance, if $X=\mathbb{D}^1=\Sp K\langle x\rangle$ is the closed unit disc, then
\begin{equation*}
\w{\D}_X(X)=\left\{\sum_{i=0}^\infty f_i \partial^i: \ f_i\in K\langle x\rangle,  \ |\pi^{-in}f_i|\to 0 \ \text{as} \ i\to \infty\ \forall n\geq 0\right\},
\end{equation*}
where $\partial=\frac{\mathrm{d}}{\mathrm{d}x}$.

It was shown in \cite[subsection 6.1]{DcapOne} that this definition does not depend on the choice of $\A$ and $\L$. The Banach algebras $D_n:=\h{U_\A(\pi^n\L)}\otimes K$ of course do depend on these choices, even though we suppress this in the notation.

The theory of $\w{\D}$-modules on rigid analytic spaces then provides a geometric way of studying $p$-adic locally analytic representations and their associated Lie algebra representations, see \cite{Ardakov}.

The algebra $\w{D}$ is a Fr\'echet--Stein algebra in the sense of \cite{ST}, so we can consider the abelian category of coadmissible $\w{D}$-modules. Here, we say that a left $\w{D}$-module $M$ is coadmissible if $M_n:=D_n\otimes_{\w{D}}M$ is a finitely generated $D_n$-module, and $M\cong \varprojlim M_n$ via the natural morphism.

In understanding the ring $\w{D}$ and coadmissible modules over it, it is then common to reduce to $D_n$ via base change and exploit the fact that in many ways, the algebras $D_n$ have better ring-theoretic properties than $\w{D}$: for instance, they are Noetherian (whereas $\w{D}$ is usually not), and it was shown in \cite[Theorem A]{DcapThree} that if $K$ is discretely valued and $L$ is a free $A$-module, then we can choose $\L$ in such a way that $D_n$ is Auslander--Gorenstein for each $n\geq 0$. Comparing with the ring of algebraic (finite-order) differential operators, it is natural to ask if $D_n$ is in fact Auslander regular, i.e. whether $D_n$ has also finite global dimension\footnote{Auslander regularity is a cohomological regularity condition for (usually non-commutative) Noetherian rings. We refer to section 4 for a more detailed discussion.}. We would also like to remove the discreteness assumption in \cite{DcapThree}.

In this paper, we show that this can be achieved when passing to a suitable covering.

\begin{thm}
\label{MainThm}
Let $X$ be a smooth rigid analytic $K$-variety. Then there exists an admisssible affinoid covering $(X_i)$, $X_i=\mathrm{Sp}A_i$, such that for each $i$, there exists an admissible affine formal model $\A_i\subseteq A_i$ and an $(R, \A_i)$-Lie lattice $\L_i\subseteq \T_X(X_i)$ such that
\begin{equation*}
D_{n,i}=\h{U_{\A_i}(\pi^n\L_i)}\otimes_R K
\end{equation*}
is Auslander regular for all $n\geq 0$, with
\begin{equation*}
\mathrm{gl.dim.} D_{n,i}\leq 2\mathrm{dim}X_i+3.
\end{equation*}
\end{thm}

Recently, Fernando Pe\~na gave the improved bound of $2\mathrm{dim}X_i$ for the global dimension for a very general class of spaces, see \cite[Theorem D]{Pena}. His methods are completely different from ours. It is not clear to us if his methods can also be used to establish Auslander regularity.

In some cases, we are able to show that $\mathrm{gl. dim.}D_{n,i}=\mathrm{dim}X_i$, in analogy with the algebraic case. In general, using Bernstein's inequality for coadmissible modules (see \cite[Theorem A]{DcapThree}, \cite{Bodegl}), this is at least true `in the limit as $n\to \infty$':

\begin{cor}[{see Corollary \ref{diminthelimit}}]
	\label{limitdim}
Let $X$ be a smooth rigid analytic $K$-variety of dimension $d$, and let $\M$, $\N$ be coadmissible $\w{\D}_X$-modules. Then $\mathcal{E}xt^j_{\w{\D}_X}(\M, \N)=0$ for all $j>d$.
\end{cor}
Here, $\w{\D}_X$ denotes the sheaf associated to the assignment $U\mapsto \w{\D}_U(U)$ for any affinoid subdomain $U$, and we call a $\w{\D}_X$-module coadmissible if it is locally given as the localization (in the sense of \cite[subsection 8.2]{DcapOne}) of a coadmissible $\w{\D}_U(U)$-module for $U$ affinoid. The sheaf $\mathcal{E}xt^j$ is the cohomology of $\mathrm{R}\mathcal{H}om$ calculated in the derived category 
\begin{equation*}
	\mathrm{D}(\w{\D}_X):=\mathrm{D}(\mathrm{Mod}_{\mathrm{Shv}(X, LH(\h{\B}c_K))}(\w{\D}_X))
\end{equation*}
of complete bornological $\w{\D}_X$-modules introduced in \cite{SixOp}.

We can now return to our slogan: ``Coadmissible $\w{\D}$-modules are microlocally perfect''. We remark that we can heuristically regard $\w{\D}_X$ as a non-commutative quantization of the rigid analytic cotangent bundle $T^*X$, with coadmissible $\w{\D}_X$-modules being the corresponding analogue of coherent $\O_{T^*X}$-modules. Note that $T^*X$ is not affinoid, even if $X$ is, but is rather given locally as $U\times \mathbb{A}^{d, \mathrm{an}}_K$, i.e. an increasing union of products $U\times \Sp K\langle \pi^nx_1,\hdots, \pi^nx_d\rangle$ for $U\subseteq X$ and $\Sp K\langle \pi^nx_1, \hdots, \pi^nx_d\rangle$ a polydisc. On the $\D$-module theoretic side, this is reflected by the fact that $\w{\D}(U)$ is written as an inverse limit of Noetherian Banach algebras, which quantize suitable `boxes' $U\times \Sp K\langle \pi^nx_1, \hdots,\pi^nx_d\rangle$. Theorem \ref{MainThm} can then be interpreted as saying that coadmissible $\w{\D}_X$-modules admit locally a finite resolution by finite projective modules, where `locally' refers both to the base space $X$ \emph{and} to the cotangent direction -- the ring $D_{n, i}$ can be regarded as a non-commutative quantization of a `box', an affinoid subspace of the cotangent space $T^*X_i$.  

Since coadmissible modules are not necessarily finitely generated, we point out that it is evidently not true that coadmissible $\w{\D}_X$-modules admit finite resolutions by finite projectives locally on $X$ -- even for finitely presented modules, the non-Noetherianity of $\w{\D}(U)$ makes such resolutions extremely unlikely. Theorem \ref{MainThm} allows us to overcome this by working with suitable resolutions on the Noetherian level, and then taking the limit. 

While we omit the details, we remark that a standard argument (\cite[Corollary 7.3]{SchmidtAuslander}) implies the Auslander regularity of the skew-group rings
\begin{equation*}
	D_n\rtimes_N G
\end{equation*}
appearing in \cite{Ardakov}, whenever $D_n$ is Auslander regular. Here, $G$ is a compact $p$-adic Lie group acting continuously on $X$, and the skew-group rings alluded to above are a key instrument in understanding $G$-equivariant $\D$-modules on rigid analytic spaces, compare \cite{Ardakov}. 

We also use the Theorem above to deduce a number of standard identities for the $\w{\D}$-module operations defined in \cite{SixOp}. In particular, we establish a projection formula and an adjunction relation for the direct and inverse image operations (precise definitions and notations will be recalled in section 6).

\begin{thm}
\label{projandadj}
Let $f: X\to Y$ be a separated morphism of smooth rigid analytic $K$-varieties, satisfying some mild finiteness conditions (see Theorem \ref{adjunctiontheorem}).
\begin{enumerate}[(i)]
\item (Projection formula) Let $\M^\bullet \in \mathrm{D}^b_\C(\w{\D}_X^\mathrm{op})$ such that $f_+\M^\bullet\in \mathrm{D}^b_\C(\w{\D}_Y^\mathrm{op})$, and let $\N^\bullet\in \mathrm{D}^-_\C(\w{\D}_Y)$ such that $f^!\N^\bullet\in \mathrm{D}^-_\C(\w{\D}_X)$. Then there is a natural isomorphism
\begin{equation*}
f_+\M^\bullet\widetilde{\otimes}^\mathbb{L}_{\w{\D}_Y} \N^\bullet \cong \mathrm{R}f_*(\M^\bullet\widetilde{\otimes}^\mathbb{L}_{\w{\D}_X} f^!\N^\bullet)[\mathrm{dim}Y-\mathrm{dim}X]
\end{equation*}
in $\mathrm{D}(\mathrm{Shv}(Y, LH(\h{\B}c_K)))$.
\item Let $\M^\bullet \in \mathrm{D}_\C(\w{\D}_X)$ such that $f_!\M^\bullet \in \mathrm{D}_\C(\w{\D}_Y)$, and let $\N^\bullet \in \mathrm{D}_\C(\w{\D}_Y)$ such that $f^!\N^\bullet \in \mathrm{D}_\C(\w{\D}_X)$. Then there is a natural isomorphism
\begin{equation*}
\mathrm{R}\mathcal{H}om_{\w{\D}_Y}(f_!\M^\bullet, \N^\bullet)\cong \mathrm{R}f_*\mathrm{R}\mathcal{H}om_{\w{\D}_X}(\M^\bullet, f^!\N^\bullet).
\end{equation*}
\item Let $\M^\bullet \in \mathrm{D}_\C(\w{\D}_X)$ such that $f_+\M^\bullet \in \mathrm{D}_\C(\w{\D}_Y)$, and let $\N^\bullet \in \mathrm{D}_\C(\w{\D}_Y)$ such that $f^+\N^\bullet \in \mathrm{D}_\C(\w{\D}_X)$. Then there is a natural isomorphism
\begin{equation*}
\mathrm{R}f_*\mathrm{R}\mathcal{H}om_{\w{\D}_X}(f^+\N^\bullet, \M^\bullet)\cong \mathrm{R}\mathcal{H}om_{\w{\D}_X}(\N^\bullet, f_+\M^\bullet).
\end{equation*}
\end{enumerate}
\end{thm}

\begin{proof}
(i) is Theorem \ref{projectionformula}. (ii) and (iii) is given in Theorem \ref{adjunctiontheorem}.
\end{proof}

Here, $\mathrm{D}^b_{\C}(\w{\D}_X)$ denotes the full subcategory of the derived category $\mathrm{D}(\w{\D}_X)$ which consists of \emph{bounded $\C$-complexes}, i.e. complexes in $\mathrm{D}^b(\w{\D}_X)$ with coadmissible cohomology groups (by \cite[Corollary 8.7]{SixOp}).

We direct the reader to the notational conventions at the end of this introduction for a discussion of the dimensional shift $\mathrm{dim}Y-\mathrm{dim}X$ in part (i).

The conditions on $\M^\bullet$ and $\N^\bullet$ are automatically satisfied if $f$ is smooth and projective, but it is useful to formulate our results in greater generality: even for morphisms $f$ such that $f_+$ and $f^!$ do not preserve $\C$-complexes in general, the Theorem above will often apply when $\M^\bullet$ and $\N^\bullet$ satisfy suitable holonomicity conditions. A notion of holonomic $\C$-complexes will be introduced in a subsequent paper. 

We now explain the key steps in the proof of Theorem \ref{MainThm}: the case where $X=\mathbb{D}^m$ is a polydisc was done in \cite{Smith} in the discretely valued case and in \cite{Bodegl} in general. We then consider the case where $X$ can be written as a closed subvariety of a polydisc via an embedding $\iota: X\to \mathbb{D}^m$, where the corresponding vanishing ideal is generated by some suitably `regular' sequence $x_1, \dots, x_r$. It then remains to use calculations similar to the Kashiwara-type equivalences in \cite{DcapTwo} and \cite{SixOp} to reduce to the polydisc case. This last step is slightly subtle, as we end up varying the lattices involved, so that the Noetherian algebra $D_n$ in question does not get compared to an analogous algebra on the polydisc, but rather to a certain intermediate ring, which we call $D'_{ \infty, n}$ -- this is a Fr\'echet algebra obtained by varying the convergence condition only for those derivations corresponding to local coordinates which are in the vanishing ideal (i.e., $x_1, \hdots, x_r$), while keeping the others fixed. 

The fact that we can reduce from an arbitrary $X$ to this situation follows from results due to Kiehl \cite{Kiehl} regarding the local structure of smooth rigid analytic $K$-varieties and tubular neighbourhoods.

The remaining results are then obtained by arguing locally, proving the analogous statements over $\D_n$ using Auslander regularity, and taking the limit.

\subsection*{Structure of the paper}
In section 2, we recall some results from \cite{Kiehl} which show that locally, a smooth rigid analytic $K$-variety may be written as a closed subvariety of a polydisc in a rather special way which we call `Kiehl standard form', and that such a closed subvariety admits a tubular neighbourhood (without the need to pass to a smaller admissible covering).

In section 3, we give a brief overview of some of the key results from \cite{SixOp}, which show how Fr\'echet--Stein algebras and coadmissible modules can be studied within the framework of complete bornological vector spaces.

In section 4, we begin our study of Auslander regularity. We show that for a Noetherian Banach $K$-algebra $A$, Auslander regularity can be established from the vanishing of Ext groups within the category of complete bornological $A$-modules just as well as from abstract Ext groups. We then recall the result from \cite{Smith}, \cite{DcapThree}, \cite{Bodegl} that in the case of the polydisc, the ring $D_n$ is indeed Auslander regular.

In section 5, we build on \cite{DcapTwo} to introduce the notion of a weakly standard closed embedding. If $\iota: X\to Y$ is a weakly standard embedding of smooth affinoids, we use the Kashiwara operations from \cite{SixOp} to compare $D_n$-modules over $X$ with modules over a suitable intermediate ring $D'_{ \infty, n}$ defined over $Y$. These calculations ensure that if the differential operators on $Y$ give rise to Auslander regular rings, then we also have Auslander regular rings on $X$. We also carry out similar arguments for a pullback along a smooth morphism and then use the results from section 2 to finish the proof of Theorem \ref{MainThm}.

In section 6, we discuss applications to coadmissible $\w{\D}$-modules. We first recall the definition of the sheaves $\w{\D}$ and $\D_n$ and the six operations introduced in \cite{SixOp}. We then prove Theorem \ref{projandadj} and Corollary \ref{limitdim}, in each case arguing (micro)locally, using Auslander regularity of $\D_n$ and then carefully taking the limit $n\to \infty$. That this limiting procedure is indeed legitimate rests on a couple of technical lemmata, mainly regarding the interplay of derived completed tensor products and direct products, which we prove in full detail in the appendix. 

\subsection*{Notational conventions}
Throughout, $K$ denotes a complete nonarchimedean field of mixed characteristic $(0, p)$, with ring of integers $R$ and residue field $k$. We write $\mathfrak{m}$ for the maximal ideal of $R$ and pick $\pi\in \mathfrak{m}$ with $\pi\neq 0$.

All rigid analytic varieties considered here are assumed to be quasi-separated and quasi-paracompact. When referring to the dimension of a smooth rigid analytic variety $X$, we will tacitly assume that $X$ is equidimensional. 

Given an $R$-module $M$, we set $M_K=M\otimes_R K$.

If $V$ is normed $K$-vector space, we let $V^\circ$ denote its unit ball.

We write $R\langle x_1, \dots, x_m\rangle$ for the $\pi$-adic completion of the polynomial ring $R[x_1, \dots, x_m]$, and set
\begin{equation*}
T_m=K\langle x_1, \dots, x_m\rangle=R\langle x_1, \dots, x_m\rangle\otimes _R K,
\end{equation*}
the Tate algebra in $m$ variables. When the number of variables is understood, we might shorten this to $K\langle x\rangle$ and use multi-index notation to describe elements, i.e.
\begin{equation*}
K\langle x\rangle =\left\{\sum_{\alpha\in \mathbb{N}^m} a_{\alpha} x^\alpha: \ a_\alpha\in K, \ a_\alpha\to 0 \ \text{as}\ |\alpha|\to \infty\right\}
\end{equation*}
where $\alpha=(\alpha_1, \dots, \alpha_m)\in \mathbb{N}^m$, $|\alpha|=\alpha_1+\dots+\alpha_m$.\\
As it will feature repeatedly in this work, we also introduce the notation
\begin{align*}
K\{x\}&=\varprojlim_{n\geq 0} K\langle \pi^nx\rangle\\
&=\left\{\sum_{\alpha} a_\alpha x^\alpha: \ \pi^{-n|\alpha|}a_\alpha\to 0 \ \text{as}\ |\alpha|\to \infty\ \forall n \right\}.
\end{align*}

\section{Local structure of smooth rigid analytic spaces}
In this brief section, we recall some results on smooth rigid analytic $K$-varieties, due to Kiehl \cite{Kiehl}.

\begin{prop}[{\cite[Folgerung 1.14]{Kiehl}}]
\label{localform}
Let $X$ be a smooth rigid analytic $K$-variety. Then there exists an admissible covering of $X$ by affinoid subspaces $X_i=\Sp A_i$ such that each $X_i$ is of the following form:

There exist $m_i\in \mathbb{N}$, $t_i\in T_{m_i}$ and a monic polynomial $P_i(Y)\in T_{m_i}[Y]$ such that $\frac{\mathrm{d}P_i}{\mathrm{d}Y}$ becomes a unit in $T_{m_i}\langle t_i^{-1}\rangle[Y]/(P_i(Y))$, such that
\begin{equation*}
A_i\cong T_{m_i}\langle t_i^{-1}\rangle[Y]/(P_i(Y)).
\end{equation*}
Moreover, $A_i$ is finite \'etale over $T_{m_i}\langle t_i^{-1}\rangle$.
\end{prop}

We say that an affinoid $K$-variety has \textbf{Kiehl standard form} if it is of the form given in the previous proposition, i.e. $X=\Sp A$ with
\begin{equation*}
A\cong T_m\langle t^{-1}\rangle [Y]/(P(Y))
\end{equation*}
as above.

As in \cite{Kiehl}, let $d$ denote the image of $Y$ in $A$. Choose $s$ such that $|\pi^s d|\leq 1$, then it is immediate that we can write 
\begin{equation*}
A\cong T_m\langle t^{-1}\rangle \langle \pi^s Y\rangle/(P(Y)).
\end{equation*}
In particular, an affinoid of Kiehl standard form can be regarded as a closed subvariety of $\Sp T_m\langle t^{-1}\rangle \langle \pi^sY\rangle$.
The following is one of the key results from \cite{Kiehl}.

\begin{prop}[{cf. \cite[Theorem 1.18]{Kiehl}}]
\label{tubular}
Let $X=\Sp A$, where
\begin{equation*}
A=T_m\langle t^{-1}\rangle\langle \pi^sY\rangle/(P(Y))
\end{equation*}
as above. Let $X'=\Sp T_m\langle t^{-1}\rangle \langle \pi^s Y\rangle$. Then there exists some $\epsilon\in K^\times$ and an isomorphism
\begin{equation*}
\Sp A\langle Z\rangle\to X'(\epsilon^{-1} P(Y)), 
\end{equation*} 
sending $Z$ to $\epsilon^{-1}P(Y)$.
\end{prop}

\begin{proof}
We spell out the proof from \cite[Theorem 1.18]{Kiehl}, mainly to convince the reader that it is not necessary to pass to a smaller affinoid covering. Let $B=T_m\langle t^{-1}\rangle \langle \pi^sY\rangle$.

By \cite[Satz 1.16]{Kiehl}, there exists $\epsilon\in K^\times$ such that for $B'=B\langle \epsilon^{-1}P(Y)\rangle$, the natural quotient map
\begin{equation*}
B' \to B'/P(Y)\cong A
\end{equation*}
admits a section in the category of $K$-algebras. Explicitly, the proof of \cite[Satz 1.16]{Kiehl} shows that for sufficiently small $\epsilon$, $B'$ already contains an element $d'$ mapping to $d\in A$ with the property that $P(d')=0$ in $B'$. This produces the desired section.

Sending $Z$ to $\epsilon^{-1}P(Y)$, we thus obtain a commutative square
\begin{equation*}
\begin{xy}
\xymatrix{
A\langle Z\rangle\ar[r] \ar[d] & B'=B\langle \epsilon^{-1}P(Y)\rangle\ar[d]\\
A\ar[r]& B'/(P(Y))
}
\end{xy} 
\end{equation*}
where the bottom arrow is an isomorphism, the arrow on the left is given by quotienting out by $Z$, and the arrow on the right is given by quotienting out by $P(Y)$. Shrinking $\epsilon$ and rescaling $Z$, we can invoke \cite[Folgerung 1.9]{Kiehl} so that the top horizontal arrow is surjective.

By assumption, $A\langle Z \rangle$ and $B'$ are regular affinoid algebras, both of dimension $m+1$. The connected components of $\Sp A\langle Z\rangle$ are of the form $\Sp A_i\langle Z\rangle$, where $\Sp A_i$ is a connected component of $\Sp A$. Since $B'/(P(Y))\cong A$, the closed embedding $\Sp B'\to \Sp A\langle Z\rangle$ exhibits $\Sp B'$ as a closed subvariety which intersects each connected component non-trivially.

Let $\Sp A_i$ be a connected component of $\Sp A$. By the preceding paragraph, the completed tensor product $A_i\h{\otimes}_A B'$ is non-zero, and we have a commutative diagram
\begin{equation*}
\begin{xy}
\xymatrix{
A\langle Z\rangle \ar[r]\ar[d] & B'\ar[d]\\
A_i\langle Z\rangle \ar[r]^\theta\ar[d] & A_i\h{\otimes}_A B'\ar[d]\\
A_i\ar[r] & A_i\h{\otimes}_A (B'/(P(Y)),
}
\end{xy}
\end{equation*}  
where all horizontal arrows are surjective and the bottom horizontal arrow is even an isomorphism. Now $A_i\langle Z\rangle$ is Noetherian, regular and connected, hence an integral domain, and of dimension $m+1$, and $B'_i:=A_i\h{\otimes}_A B'$ is likewise regular of dimension $m+1$, being a non-empty affinoid subdomain of $\Sp B'$. But then the surjective morphism $\theta: A_i\langle Z\rangle \to B'_i$ must be an isomorphism: if $\mathfrak{p}_0\subsetneq \mathfrak{p}_1\subsetneq \hdots\subsetneq \mathfrak{p}_{m+1}$ is a chain of prime ideals in $B'_i$, it can by surjectivity be lifted to a chain of prime ideals $\mathfrak{q}_\bullet$ in $A_i\langle Z\rangle$ of length $m+1$, all containing $\mathrm{ker}\theta$, and $\{0\}\subseteq\mathfrak{q}_0\subsetneq \hdots \subsetneq \mathfrak{q}_{m+1}$ produces a chain of length $m+2$ unless $\{0\}=\mathrm{ker}\theta=\mathfrak{q}_0$. Therefore $A_i\langle Z\rangle \to B'_i$ is an isomorphism for each $i$, and thus $A\langle Z\rangle \to B'$ is an isomorphism.  
\end{proof}

\section{Fr\'echet--Stein algebras and coadmissible modules}
We quickly recall the theory of Fr\'echet--Stein algebras and coadmissible modules as developed by Schneider--Teitelbaum \cite{ST}, translating the results into the language of complete bornological modules as in \cite{SixOp}.

Let $A$ be a Noetherian Banach $K$-algebra. Then any finitely generated $A$-module can be equipped with a canonical Banach structure such that any $A$-module morphism between finitely generated $A$-modules is continuous (i.e. bounded) and strict (see \cite[section 3.7.3]{BGR}). The theory of Fr\'echet--Stein algebras and coadmissible modules generalizes these results, taking the theory of coherent $\O$-modules on rigid analytic quasi-Stein spaces as inspiration.
\begin{defn}
A $K$-Fr\'echet algebra $A$ is called a (two-sided) \textbf{Fr\'echet--Stein algebra} if $A\cong \varprojlim A_n$, where $A_n$ is a (two-sided) Noetherian Banach $K$-algebra for each $n$, the transition maps $A_{n+1}\to A_n$ are flat on both sides and have dense image for each $n$.

A left module $M$ over a Fr\'echet--Stein algebra $A$ is called \textbf{coadmissible} if $M\cong \varprojlim M_n$, where $M_n$ is a finitely generated $A_n$-module and the natural morphism $A_n\otimes_{A_{n+1}}M_{n+1}\to M_n$ is an isomorphism for each $n$.
\end{defn}

Given a Fr\'echet--Stein algebra $A$, we denote its category of coadmissible left modules by $\C_A$.
\begin{thm}[{\cite[Corollary 3.5 and following remarks]{ST}}]
Let $A$ be a Fr\'echet--Stein algebra. The category $\C_A$ is an abelian category. Each coadmissible module can be equipped with a canonical Fr\'echet structure such that each $A$-module morphism between coadmissible $A$-modules is continuous and strict.
\end{thm}
We also remark that any finitely presented $A$-module is coadmissible. 
\begin{prop}[{\cite[Corollary 3.1, Theorem B]{ST}}]
Let $A$ be a Fr\'echet--Stein algebra, and let $M\cong \varprojlim M_n$ be a coadmissible $A$-module. 
\begin{enumerate}[(i)]
\item The natural morphism $A_n\otimes_A M\to M_n$ is an isomorphism for each $n$.
\item The inverse system $(M_n)_{n\in \mathbb{N}}$ satisfies the topological Mittag-Leffler property, so that 
\begin{equation*}
	\mathrm{R}^1\varprojlim M_n=0.
\end{equation*}
\end{enumerate}
\end{prop}

Examples:
\begin{enumerate}[(i)]
\item Let $X$ be a rigid analytic Stein space. Then $\O_X(X)$ is a Fr\'echet--Stein algebra, and the global sections functor yields an equivalence of categories between coherent $\O_X$-modules and coadmissible $\O_X(X)$-modules.
\item Let $G$ be a compact $p$-adic Lie group. Then the locally analytic distribution algebra $D(G, K)$ is a Fr\'echet--Stein algebra, and the duals of coadmissible $D(G, K)$-modules are by definition the admissible locally analytic $G$-representations (\cite[Theorem 5.1]{ST}).
\item Let $X$ be a smooth affinoid $K$-variety. Then the completed algebra of differential operators $\w{\D}_X(X)$ is a Fr\'echet--Stein algebra (\cite[Theorem 1.1.(ii)]{Bode1}, \cite[Theorem 2.8]{SixOp}). 
\end{enumerate}

This paper is concerned with example (iii). To explain the construction of $\w{\D}_X(X)$, we first need to introduce a number of preliminary concepts.

\begin{defn}
Let $A$ be a commutative $K$-algebra. 
A \textbf{smooth $(K, A)$-Lie--Rinehart} is a finitely generated projective $A$-module $L$, equipped with a $K$-bilinear Lie bracket and a map of Lie algebras $\rho: L\to \mathrm{Der}_K(A)$ satisfying
\begin{equation*}
[\partial, a\cdot \partial']=\rho(\partial)(a)\cdot \partial'+a\cdot [\partial, \partial']
\end{equation*}
for all $a\in A$, $\partial, \partial'\in L$.
\end{defn}
The typical and (for us) most important example is $L=\mathrm{Der}_K(A)$, $\rho=\mathrm{id}$ when $A$ is a smooth affinoid $K$-algebra.

We also require the corresponding integral notions.
\begin{defn}
Let $A$ be an affinoid $K$-algebra. An \textbf{admissible affine formal model} of $A$ is an $R$-subalgebra $\A\subseteq A$ which is topologically of finite presentation over $R$ such that $\A_K\cong A$ via the natural morphism.

If $(L, \rho)$ is a smooth $(K, A)$-Lie--Rinehart algebra, an $(R, \A)$-\textbf{Lie lattice} is a finitely generated $\A$-submodule $\L\subseteq L$ such that $\L_K\cong L$, $\L$ is closed under the Lie bracket, and $\rho(\L)(\A)\subseteq \A$.\\
We say that $\L$ is a \textbf{smooth} Lie lattice if it is moreover a projective $\A$-module. 
\end{defn}

Let $\L$ be an $(R, \A)$-Lie lattice of a smooth $(K, A)$-Lie--Rinehart algebra $L$. Then we can form the universal enveloping algebra
\begin{equation*}
U_\A(\L)
\end{equation*}
following \cite{Rinehart} -- this is a ring generated by $\A$ and $\L$ subject to the relations
\begin{enumerate}[(i)]
\item $\partial \partial'-\partial'\partial=[\partial, \partial']$ for all $\partial, \partial'\in \L$,
\item if $a\in \A$, $\partial\in \L$, then $a\partial$ (multiplication in $U_\A(\L)$) is the same as $a\cdot \partial$ ($\A$-module structure on $\L$), and
\begin{equation*}
\partial a-a\partial=\rho(\partial)(a).
\end{equation*}
\end{enumerate}

If $\L$ is an $(R, \A)$-Lie lattice, so is $\pi^n\L$ for any $n\geq 0$, and we define the completed enveloping algebra
\begin{equation*}
\w{U_A(L)}:=\varprojlim \h{U_\A(\pi^n\L)}_K.
\end{equation*}
It was shown in \cite[subsection 6.1]{DcapOne} that this does not depend on the choice of $\A$ and $\L$.

Examples:
\begin{enumerate}[(i)]
\item If $A=K$ and $L=\mathfrak{g}$ is a finite-dimensional $K$-Lie algebra, we can regard it as a Lie--Rinehart algebra with $\rho=0$. Then $\w{U_K(\mathfrak{g})}$ is the Arens--Michael envelope of the usual universal enveloping algebra of $U(\mathfrak{g})$, as studied in \cite{Schmidt}.
\item As a special case of the above, if $\mathfrak{g}=Kx_1\oplus \dots \oplus Kx_m$ is abelian, choose the Lie lattice $\mathfrak{g}_R=\oplus Rx_i$ to obtain
\begin{equation*}
\w{U_K(\mathfrak{g})}=K\{x\}\cong \O(\mathbb{A}^{m, \mathrm{an}}).
\end{equation*} 
\item Let $X=\Sp A$ be a smooth affinoid $K$-variety, and take $L=\mathrm{Der}_K(A)$, $\rho=\mathrm{id}$. Then we define
\begin{equation*}
\w{\D}_X(X)=\w{U_A(L)}.
\end{equation*}
We will give a more explicit description of $\w{\D}_X(X)$ in the case of $X=\mathbb{D}^m$ in the next section.
\end{enumerate}

\begin{prop}[{\cite[Theorem 2.8]{SixOp}}]
Let $A$ be an affinoid $K$-algebra and $L$ a smooth $(K, A)$-Lie--Rinehart algebra. Then $\w{U_A(L)}$ is a Fr\'echet--Stein algebra.
\end{prop}

We can thus talk about coadmissible $\w{U_A(L)}$-modules. In fact, coadmissible $\w{U_A(L)}$-modules share one additional feature with coadmissible modules over distribution algebras: if $D(G, K)$ is the distribution algebra of some compact $p$-adic Lie group $G$, then any coadmissible $D(G, K)$-module is a nuclear Fr\'echet space \cite[Lemma 6.1]{ST}. In the case of coadmissible $\w{U_A(L)}$-modules, we have the following relative version:

\begin{defn}[{\cite[Definition 5.3]{SixOp}}]
	Let $A$ be a (not necessarily commutative) Noetherian Banach $K$-algebra. A left Fr\'echet $A$-module $M$ is called \textbf{nuclear over $A$} (or $A$-nuclear) if $M\cong \varprojlim M_n$, where $M_n$ is a left Banach $A$-module with $A^\circ\cdot M_n^\circ\subseteq M_n^\circ$ such that for each $n$,
	\begin{enumerate}[(i)]
		\item the image of $M$ in $M_n$ is dense.
		\item there exists a left Banach $A$-module $F_n$ with $A^\circ\cdot F_n^\circ\subseteq F_n^\circ$ and a continuous $A$-module surjection $F_n\to M_n$ such that the composition $F_n\to M_n\to M_{n-1}$ is strictly completely continuous over $A$ in the sense of \cite[Definition 1.1]{Kiehlproper}.
	\end{enumerate}
\end{defn}

\begin{prop}[{\cite[Proposition 5.5]{SixOp}}]
Let $A$ be an affinoid $K$-algebra and let $L$ be a smooth $(K, A)$-Lie--Rinehart algebra admitting a smooth Lie lattice. Then any coadmissible $\w{U_A(L)}$-module is nuclear over $A$.
\end{prop}

Note that \cite[Lemma 9.3]{DcapOne} ensures the existence of a smooth Lie lattice at least locally. 

More generally, if $A$ is Noetherian Banach and $U=\varprojlim U_n$ is a Fr\'echet--Stein $K$-algebra with a continuous algebra map $A\to U$ such that the $U_n$ exhibit $U$ as a nuclear $A$-module, then any coadmissible $U$-module is $A$-nuclear, by the same argument as in \cite[Proposition 5.5]{SixOp}. We summarize this situation by saying that $U$ is a Fr\'echet--Stein algebra nuclear over $A$.

The notion of nuclearity relative to some algebra $A$ will only be relevant to us insofar as it features repeatedly as a condition to translate the results from the beginning of this section into bornological statements (see below). 

To do homological algebra with coadmissible modules, it is desirable to embed them into a larger category which is well-behaved both from a homological and a topological standpoint. The paper \cite{SixOp} achieves this by considering complete bornological modules.

We now summarize some results regarding coadmissible modules in a complete bornological context. We refer to \cite{ProsmansSchneiders} and \cite{SixOp} for a detailed introduction to the category $\h{\B}c_K$ of complete bornological $K$-vector spaces. Recall that a bornology on a $K$-vector space $V$ is a collection $\B$ of `bounded' subsets which satisfy:
\begin{enumerate}[(i)]
\item $\{v\}\in\B$ for each $v\in V$.
\item $\B$ is closed under taking finite unions.
\item if $B\subseteq B'$ and $B'\in \B$, then $B\in \B$.
\item if $B\in \B$, then $R\cdot B\in \B$.
\item if $B\in \B$ and $\lambda\in K$, then $\lambda\cdot B\in \B$.´
\end{enumerate}  
The category $\B c_K$ of bornological vector spaces then has as objects $K$-vector spaces equipped with a bornology, and as morphisms those $K$-linear maps which send bounded sets to bounded sets. A bornological vector space is called complete if its bornology is generated by $\pi$-adically complete $R$-submodules. We denote the full subcategory of complete bornological $K$-vector spaces by $\h{\B}c_K$.

For example, if $V$ is a locally convex topological $K$-vector space whose topology is defined by a family of semi-norms $q_i$, we can regard $V$ as a bornological vector space by declaring a subset $B$ to be bounded if and only if it is bounded with respect to each $q_i$. This defines a `bornologification' functor
\begin{align*}
LCVS_K\to \B c_K\\ 
V\mapsto V^b.
\end{align*}
By \cite[p. 102, p. 109]{Houzel}, $(-)^b$ is fully faithful on metrisable spaces. In particular, Fr\'echet spaces can naturally be regarded as complete bornological vector spaces (bornological completeness follows e.g. from \cite[Corollary 1.18]{BambozziCGT}).

The category $\h{\B}c_K$ carries a closed symmetric monoidal structure, where the tensor product $\h{\otimes}_K$ is an extension of the usual completed (projective) tensor product on Banach spaces. If $A\in \h{\B}c_K$ is a monoid (ring object), we denote by $\mathrm{Mod}_{\h{\B}c_K}(A)$ the category of complete bornological left $A$-modules.

\begin{prop}[{\cite[Lemma 4.28, Corollary 5.13, Corollary 5.14]{SixOp}}]
Let $A$ be a Noetherian Banach $K$-algebra and assume that $A^\circ$ is Noetherian (if $K$ is discretely valued) or almost Noetherian (if $K$ is densely valued). Let $U=\varprojlim U_n$ be a Fr\'echet--Stein $K$-algebra nuclear over $A$.		
	
The functor $(-)^b$ induces an exact, fully faithful embedding
\begin{equation*}
(-)^b: \C_U\to \mathrm{Mod}_{\h{\B}c_K}(U)
\end{equation*}
of the category of coadmissible $U$-modules into the category of complete bornological $U$-modules. All morphisms between coadmissible $U$-modules are strict in $\mathrm{Mod}_{\h{\B}c_K}(U)$.
\end{prop}
In the reference, the proofs are given for completed enveloping algebras, but in the light of \cite[Proposition 5.12]{SixOp} it is immediate that the same results hold in this greater generality. 

We can now formulate the complete bornological analogues of the results from the beginning of this section.
\begin{thm}[{\cite[Corollary 5.38, Corollary 5.27]{SixOp}}]
\label{Bornproperties}
Let $U$ be a Fr\'echet--Stein $K$-algebra, nuclear over some Noetherian Banach algebra $A$ with (almost) Noetherian unit ball. Let $M\cong \varprojlim M_n$ be a coadmissible $U$-module.
\begin{enumerate}[(i)]
\item We have $M_n^b\cong U_n^b\h{\otimes}_{U^b} M^b\cong U_n^b\h{\otimes}_{U^b}^\mathbb{L}M^b$ in $\mathrm{D}(\mathrm{Mod}_{\h{\B}c_K}(U_n))$ for each $n$ via the natural morphism.
\item The inverse system $(M_n^b)_{n\in \mathbb{N}}$ satisfies a bornological Mittag-Leffler property (it is a pre-nuclear system, \cite[Definition 5.24]{SixOp}), so that $M^b\cong \mathrm{R}\varprojlim M_n^b$.
\end{enumerate}
\end{thm}

\begin{prop}[{\cite[Corollary 5.39]{SixOp}}]
Let $U$ be as above. A complete bornological $U$-module $M\in \mathrm{Mod}_{\h{\B}c_K}(U)$ is coadmissible if and only if the following is satisfied: $M_n:=U_n^b\h{\otimes}_{U^b}M$ is a finitely generated $U_n$-module, equipped with its canonical Banach structure, and the natural morphism $M\to \varprojlim M_n$ is an isomorphism.
\end{prop}

We will usually regard coadmissible modules as complete bornological modules and suppress $(-)^b$ from the notation.

\section{Auslander regularity and the polydisc case}
\subsection{Auslander regularity using complete bornological modules}
\begin{defn}
Let $A$ be a two-sided Noetherian ring.
\begin{enumerate}[(i)]
\item The \textbf{grade} of a left $A$-module $M$ is 
\begin{equation*}
j(M)=\mathrm{min}\{i: \mathrm{Ext}^i_A(M, A)\neq 0\}
\end{equation*}
and $\infty$ if no such $i$ exists.
\item We say that $A$ satisfies the \textbf{Auslander condition} if for every finitely generated left $A$-module $M$ and any $i\geq 0$, we have $j(N)\geq i$ whenever $N$ is a (right) submodule of $\mathrm{Ext}^i_A(M, A)$.
\item We say that $A$ is \textbf{Auslander--Gorenstein} if it satisfies the Auslander condition and has finite left and right self-injective dimension.
\item We say that $A$ is \textbf{Auslander regular} if it satisfies the Auslander condition and has finite global dimension.
\end{enumerate}
\end{defn}

Recall that $A$ has finite left self-injective dimension $\leq j$ if and only if the following equivalent conditions are satisfied:
\begin{enumerate}[(i)]
\item the left $A$-module $_AA$ admits an injective resolution of length $\leq j$;
\item $\mathrm{Ext}^k_A(M, A)=0$ for any $k>j$ and any left $A$-module $M$;
\item $\mathrm{Ext}^k_A(M, A)=0$ for any $k>j$ and any finitely generated left $A$-module $M$.
\end{enumerate}
The equivalence of (i) and (iii) follows e.g. from Baer's criterion and dimension shifting.

If the left and right self-injective dimensions are both finite, they are equal.

Similarly, $A$ has finite (left) global dimension $\leq j$ if and only if the following equivalent conditions are satisfied:
\begin{enumerate}[(i)]
\item any left $A$-module admits a projective resolution of length $\leq j$;
\item any finitely generated left $A$-module admits a projective resolution of length $\leq j$;
\item $\mathrm{Ext}^k_A(M, N)=0$ for any $k>j$ and any left $A$-modules $M, N$;
\item $\mathrm{Ext}^k_A(M, N)=0$ for any $k>j$ and any finitely generated left $A$-modules $M, N$.
\end{enumerate}
The fact that it suffices to consider finitely generated modules follows from \cite[tags 065T, 0G8T]{stacksproj}. 

As $A$ is assumed to be Noetherian, its left global dimension agrees with its right global dimension.

\begin{lem}
\label{injdimgldim}
Let $A$ be a Noetherian ring of finite global dimension. Then $\mathrm{gl. dim.}(A)= \mathrm{inj. dim.}(A)$, where $\mathrm{inj. dim.}$ denotes the self-injective dimension of $A$.
\end{lem}
\begin{proof}
The fact that $\mathrm{gl. dim}(A)\geq \mathrm{inj. dim.}(A)$ is immediate from the definition.

Let $M$ be a finitely generated left $A$-module. Now for any finitely generated left $A$-module $N$, we can consider a short exact sequence
\begin{equation*}
0\to N'\to A^r\to N\to 0
\end{equation*}
of finitely generated $A$-modules. Applying $\mathrm{RHom}_A(M, -)$ and considering the associated long exact sequence of Ext groups, we have
\begin{equation*}
\mathrm{Ext}^k_A(M, N)\cong \mathrm{Ext}^{k+1}_A(M, N')
\end{equation*}
whenever $k>\mathrm{inj.dim.}(A)$.

It follows that if $k>\mathrm{inj.dim}(A)$ has the property that
\begin{equation*}
\mathrm{Ext}^{k+1}_A(M, N)=0
\end{equation*}
for all finitely generated left $A$-modules $N$, then
\begin{equation*}
\mathrm{Ext}^k_A(M, N)=0
\end{equation*}
for all finitely generated left $A$-modules $N$. Since such a $k$ exists by the assumption of finite global dimension, we can induct downwards to deduce that
\begin{equation*}
\mathrm{Ext}^k_A(M, N)=0
\end{equation*}
whenever $k>\mathrm{inj.dim}(A)$.
\end{proof}

So far we have worked in the category $A\mathrm{-mod}$ of all abstract (left) $A$-modules. As our methods will rely on some results from non-archimedean functional analysis, we prefer to work in the derived category of complete bornological modules $\mathrm{Mod}_{\h{\B}c_K}(A)$ and use the results below.

Let $A$ be a two-sided Noetherian Banach algebra.

Given $M, N\in \mathrm{Mod}_{\h{\B}c_K(A)}$, we let $\underline{\mathrm{Hom}}_A(M, N)\in \h{\B}c_K$ denote the inner hom (see \cite[Discussions after Proposition 3.10, Theorem 4.10]{SixOp}). Note that if $M$, $N$ are finitely generated $A$-modules, equipped with their canonical Banach structure, then $\underline{\mathrm{Hom}}_A(M, N)=\mathrm{Hom}_A(M, N)$, equipped with the sup norm.

\begin{lem}
\label{projinBc}
Let $A$ be a two-sided Noetherian Banach $K$-algebra, and let $P$ be a finitely generated $A$-module, equipped with its canonical Banach structure. The following are equivalent:
\begin{enumerate}[(i)]
\item $P$ is a projective $A$-module.
\item $P$ is a projective object in $\mathrm{Mod}_{\h{\B}c_K}(A)$.
\item For any finitely generated $A$-module $N$, equipped with the canonical Banach structure,
\begin{equation*}
\mathrm{H}^1(\mathrm{R}\underline{\mathrm{Hom}}_A(P, N))=0.
\end{equation*}
\end{enumerate}
\end{lem}
\begin{proof}
(i) is equivalent to $P$ being a direct summand of some free module $A^r$. As all $A$-module morphisms between finitely generated $A$-modules are bounded, this is equivalent to $P$ being a direct summand of $A^r$ in $\mathrm{Mod}_{\h{\B}c_K}(A)$, so (i) and (ii) are equivalent.

(ii) implies (iii) immediately, so it remains to show that (iii) implies (i). Suppose (iii) is satisfied. Let $A^r\to P$ be a surjective $A$-module morphism, which is then a strict epimorphism in $\mathrm{Mod}_{\h{\B}c_K}(A)$. By assumption, the induced morphism
\begin{equation*}
\underline{\mathrm{Hom}}_A(P, A^r)\to \underline{\mathrm{Hom}}_A(P, P)
\end{equation*}
is a strict epimorphism, and hence a surjection. Regarding the underlying vector spaces, this says that
\begin{equation*}
\mathrm{Hom}_A(P, A^r)\to \mathrm{Hom}_A(P, P)
\end{equation*}
is surjective, so that the surjection $A^r\to P$ admits a section.
\end{proof}

\begin{cor}
	Let $A$ be a two-sided Noetherian Banach $K$-algebra, and let $M$ be a finitely generated $A$-module (equipped with its canonical Banach structure).
	Then
	\begin{equation*}
		\mathrm{H}^j(\mathrm{R}\underline{\mathrm{Hom}}_A(M, A))\cong \mathrm{Ext}^j_A(M, A)
	\end{equation*}
	for each $j$, where the finitely generated right $A$-module $\mathrm{Ext}^j_A(M, A)$ is endowed with its canonical Banach structure.
\end{cor}
\begin{proof}
	Choose a projective resolution $P^\bullet$ of $M$ by finitely generated $A$-modules. By the above, $P^i$ is a projective object in $\mathrm{Mod}_{\h{\B}c_K}(A)$, and the complex $\hdots \to P^{-1}\to P^0\to M\to 0$ is strictly exact in $\mathrm{Mod}_{\h{\B}c_K}(A)$. 
	
	Thus $\mathrm{R}\underline{\mathrm{Hom}}_A(M, A)$ is represented by the complex $\underline{\mathrm{Hom}}_A(P^\bullet, A)$, which is simply $\mathrm{Hom}_A(P^\bullet, A)$, equipping each term with its canonical Banach structure. As $\mathrm{Hom}_A(P^j, A)$ is a finitely generated right $A$-module for each $j$, this complex is strict. Taking $j$th cohomology yields the result.
\end{proof}

In the situation of the Corollary above, we will write $\underline{\mathrm{Ext}}^j_A(M, A)$ to denote $\mathrm{Ext}^j_A(M, A)$, equipped with its canonical Banach structure.

\begin{lem}
\label{gldimviaborn}
Let $A$ be a two-sided Noetherian Banach $K$-algebra. The following are equivalent:
\begin{enumerate}[(i)]
\item $A$ has finite global dimension $\leq d$.
\item For any $j\geq d+1$, and any finitely generated left $A$-modules $M$, $N$ (equipped with their canonical Banach structures), we have 
\begin{equation*}
\mathrm{H}^j(\mathrm{R}\underline{\mathrm{Hom}}_A(M, N))=0.
\end{equation*}
\end{enumerate} 
\end{lem}
\begin{proof}
If (i) holds, then any finitely generated $A$-module $M$ admits a resolution of length $\leq d$ by finitely generated projective $A$-modules. By the above, applying $(-)^b$ then yields a projective resolution of $M^b$ in $\mathrm{Mod}_{\h{\B}c_K}(A)$, since $(-)^b$ is exact on finitely generated $A$-modules with their canonical Banach structure (\cite[Lemma 4.4]{SixOp}). Hence (ii) holds. 

Conversely, if (ii) is satisfied, let $M$ be a finitely generated $A$-module, and let $F^\bullet$ be a (possibly infinite) resolution of $M$ by free $A$-modules of finite rank. It suffices to show that the finitely generated $A$-module $P=\mathrm{ker}(F^{-d+1}\to F^{-d+2})$ is projective. A straightforward dimension shifting argument applied to (ii) shows that 
\begin{equation*}
\mathrm{H}^1(\mathrm{R}\underline{\mathrm{Hom}}_A(P, N))=0
\end{equation*}
for any finitely generated $A$-module $N$. Hence we are done by Lemma \ref{projinBc}.
\end{proof}

We also recall the following result concerning Auslander regularity in a Fr\'echet--Stein setting.

\begin{lem}
	\label{coadext}
	Let $U=\varprojlim U_n$ be a two-sided Fr\'echet--Stein $K$-algebra. Let $M$ be a coadmissible left $U$-module.
	\begin{enumerate}[(i)]
		\item The natural morphism 
		\begin{equation*}
			\mathrm{Ext}^j_U(M, U)\otimes_UU_n\to \mathrm{Ext}^j_{U_n}(U_n\otimes_UM, U_n)
		\end{equation*}
		is an isomorphism for each $j$.
		\item For each $j$, $\mathrm{Ext}^j_U(M, U)\cong \varprojlim \mathrm{Ext}^j_{U_n}(U_n\otimes_UM, U_n)$ is a coadmissible right $U$-module.
		\item For each $j$, $\mathrm{H}^j(\mathrm{R}\underline{\mathrm{Hom}}_U(M,U))$ is naturally isomorphic to $\mathrm{Ext}^j_U(M, U)$, endowed with its natural Fr\'echet structure.
	\end{enumerate}
\end{lem}

\begin{proof}
	(i) can be found in \cite[Lemma 8.4]{ST}, which immediately implies (ii). The proof of (iii) follows along the same lines as in \cite[Theorem 9.17, Corollary 9.18]{SixOp}.
\end{proof}

As before, we write $\underline{\mathrm{Ext}}^j_U(M, U)$ to denote the coadmissible right $A$-module $\mathrm{Ext}^j_U(M, U)$ when viewed as an object of $\mathrm{Mod}_{\h{\B}c_K}(U^{\mathrm{op}})$. Note that the above together with Theorem \ref{Bornproperties} also yields that
\begin{equation*}
	\underline{\mathrm{Ext}}^j_U(M, U)\h{\otimes}_U U_n\cong \underline{\mathrm{Ext}}^j_{U_n}(U_n\h{\otimes}_U M, U_n)
\end{equation*}
in $\mathrm{Mod}_{\h{\B}c_K}(U_n^{\mathrm{op}})$.

\subsection{The case $X=\mathbb{D}^m$}

One main goal of this paper is to establish the Auslander regularity of the rings $\h{U_\A(\pi^n\L)}_K$ appearing in the definition of $\w{\D_X}(X)$ in section 3.

We begin by discussing the case of a completed ring of differential operators on a polydisc, which is already covered in \cite{Smith} and \cite{Bodegl}. We mainly spell this out to familiarize the reader with various standard constructions and notations.
Let $A=T_m=K\langle x_1, \dots, x_m\rangle$ denote the Tate algebra in $m$ variables, so that $X=\Sp A$ is the $m$-dimensional closed unit polydisc. Let $\A=R\langle x_1, \dots, x_m\rangle$, an admissible affine formal model of $A$. 

Note that $L=\mathrm{Der}_K(A)$ is a free $A$-module of rank $m$, generated by $\frac{\mathrm{d}}{\mathrm{d}x_i}$, $i=1, \dots, m$, equipped with the usual $K$-bilinear Lie bracket. Let $\L\subseteq L$ be an $(R, \A)$-Lie lattice which is free as an $\A$-module. For example, one can take $\L=\oplus_{i=1}^m \A \frac{\mathrm{d}}{\mathrm{d}x_i}$, but we do not need to restrict to this case.

For any $n\geq 0$, $\pi^n\L$ is a smooth $(R, \A)$-Lie lattice, so that we can form the universal enveloping algebra 
\begin{equation*}
U_\A(\pi^n\L).
\end{equation*}
As in the case of universal enveloping algebras over fields, $U_\A(\pi^n\L)$ is naturally filtered (corresponding to the order filtration of differential operators), and
\begin{equation*}
\mathrm{gr} U_\A(\pi^n\L)\cong \mathrm{Sym}_\A (\pi^n\L)
\end{equation*}
due to \cite[Theorem 3.1]{Rinehart}.

Thus $U_\A(\pi^n\L)$ can be identified with the subring of $\mathrm{End}_K(A)$ consisting of those endomorphisms which can be written as finite sums
\begin{equation*}
\sum_{\alpha\in \mathbb{N}^m} f_\alpha d_1^{\alpha_1}\dots d_m^{\alpha_m}, \ f_\alpha\in \pi^{n|\alpha|}\A, 
\end{equation*}
where $|\alpha|=\alpha_1+\dots +\alpha_m$ for any $\alpha=(\alpha_1, \dots, \alpha_m)\in \mathbb{N}^m$ and $d_1, \dots, d_m\in \L$ is an $\A$-basis of $\L$.

We can now set
\begin{equation*}
D_n=\h{U_\A(\pi^n\L)}_K.
\end{equation*}
It follows that as a $K$-Banach space (and even as a left Banach $A$-module), we can identify $D_n$ with
\begin{equation*}
K\langle x_1, \dots, x_m, \pi^nd_1, \dots, \pi^nd_m\rangle.
\end{equation*}

\begin{thm}[{Auslander regularity in the polydisc case}]
\label{regforpolydisc}
With notation as in the preceding paragraph, $D_n$ is Auslander regular for any $n\geq 1$, with
\begin{equation*}
\mathrm{gl.dim.}(D_n)\leq 2m.
\end{equation*}
If $\L=\oplus \A \frac{\mathrm{d}}{\mathrm{d}x_i}$, then $\mathrm{gl.dim.}(D_n)=m$ for all $n\geq 1$.

If $K$ is discretely valued, then the same holds also for $n=0$.
\end{thm} 
\begin{proof}
We first assume that $K$ is discretely valued, so that we can assume without loss of generality that $\pi$ is a uniformizer. In this case, the proof is given in \cite{Smith}. We summarize it here for the convenience of the reader.
	
Fix $n\geq 0$ and write $U=U_\A(\pi^n\L)$. As noted above, $U$ is filtered, inducing a Zariskian filtration on $U/\pi U$, and 
\begin{equation*}
\mathrm{gr} (U/\pi U)\cong k[x_1, \dots, x_m, Y_1, \dots, Y_m], 
\end{equation*}
where $Y_i$ is the principal symbol corresponding to the generator $\pi^nd_i$. It thus follows from \cite[Theorem 3.9]{Bjork} (see also \cite[Proposition 2.5.(b)]{Smith}) that 
\begin{equation*}
\h{U}/\pi \h{U}\cong U/\pi U
\end{equation*}
is Auslander regular. 

Now $D_n$ admits itself a Zariskian filtration $\{\pi^i \h{U}\}_{i\in \mathbb{Z}}$, whose associated graded ring is
\begin{equation*}
\mathrm{gr}D_n\cong \h{U}/\pi \h{U}[s, s^{-1}],
\end{equation*}
where $s$ is the principal symbol of $\pi$. Applying \cite[Theorem 3.9]{Bjork} again, we see that $D_n$ is Auslander regular.

It remains to determine the global dimension of $D_n$, which by Lemma \ref{injdimgldim} is the same as its self-injective dimension. Thus 
\begin{equation*}
\mathrm{gl.dim.}(D_n)\leq 2m
\end{equation*}
by \cite[Lemma 4.3]{DcapThree}, and if $\L=\oplus \A\frac{\mathrm{d}}{\mathrm{d}x_i}$, we have 
\begin{equation*}
\mathrm{gl.dim.}(D_n)=m
\end{equation*}
by \cite[Corollary 7.4, Theorem 3.3]{Smith}. 

If $K$ is not discretely valued and $n\geq 1$, an analogue of the proof above using almost mathematics was given in \cite[Theorem 1.2, Theorem 1.1.(ii)]{Bodegl}.
\end{proof}

We can also allow for other affine formal models than $R\langle x\rangle$, but this does not really produce any new rings, as the following result illustrates. 

\begin{prop}
\label{varymodel}
Let $\A$ be an admissible affine formal model for $A=T_m$, and let $\L$ be an $(R, \A)$-Lie lattice in $\mathrm{Der}_K(A)$ which is free as an $\A$-module. Then there exists an $(R, R\langle x\rangle)$-Lie lattice $\L'$ which is free over $R\langle x\rangle$ and $s\geq 0$ such that
\begin{equation*}
\h{U_\A(\pi^n\L)}_K\cong \h{U_{R\langle x\rangle}(\pi^{n-s}\L')}_K 
\end{equation*}
for all $n\geq s$.

In particular, $\h{U_\A(\pi^n\L)}_K$ is Auslander regular of global dimension $\leq 2m$ for sufficiently large $n$.
\end{prop}
\begin{proof}
Let $d_1, \dots, d_m$ be an $\A$-basis of $\L$. By \cite[Propositions 3.1/8 and 9]{Boschlectures}, $\A$ is contained in $R\langle x\rangle$. Then there exists an integer $s\geq 0$ such that 
\begin{equation*}
\pi^sd_i(R\langle x\rangle)\subseteq R\langle x\rangle
\end{equation*}
for all $i$. In particular, $\L'=\sum R\langle x\rangle \pi^sd_i$ is an $(R, R\langle x\rangle)$-Lie lattice. Moreover, since the $d_i$ form a basis of $\mathrm{Der}_K(A)$ over $A$, this sum is actually a direct sum, and
\begin{equation*}
\L'=R\langle x\rangle \otimes_\A \pi^s\L.
\end{equation*}
Now apply \cite[Proposition 3.3.(b)]{DcapOne} with $\sigma=\mathrm{id}$.

\end{proof}

\section{Auslander regularity in the general case}
\subsection{Weakly standard embeddings}
In what follows, it will be crucial to understand the passage along a closed embedding of smooth rigid analytic $K$-varieties, and how it relates various classes of $\D$-modules. In \cite{DcapTwo}, the notion of a standard embedding was introduced to make explicit $\D$-module theoretic calculations possible.
\begin{defn}[{\cite[Definition 5.2]{DcapTwo}}]
Let $I$ be an ideal of an affinoid $K$-algebra $A$, and let $\iota: \Sp A/I\to \Sp A$ be the natural closed embedding. We call a subset $\{d_1, \dots, d_m\}$ of $L=\mathrm{Der}_K(A)$ an $I$-\textbf{standard basis} if 
\begin{enumerate}[(i)]
\item $\{d_1, \dots, d_m\}$ is an $A$-module basis for $L$,
\item there is a generating set $\{f_1, \dots, f_r\}$ for $I$ with $r\leq m$ such that 
\begin{equation*}
d_i(f_j)=\delta_{ij}
\end{equation*}
for all $1\leq i\leq m$, $1\leq j\leq r$.
\end{enumerate}
We say that $\iota$ admits a standard basis (or that $\iota$ is a \textbf{standard embedding}) if there exists an $I$-standard basis of $L$.
\end{defn}

We remark that if $X\to Y$ is any closed embedding of smooth rigid analytic $K$-varieties, then $Y$ admits an affinoid covering $Y_i$ such that $X\cap Y_i\to Y_i$ is standard for all $i$. This was proved in \cite[Theorem 6.2]{DcapTwo} in greater generality (replacing $L$ by an arbitrary Lie--Rinehart algebra), but in our setting it also follows from Kiehl's theory of tubular neighbourhoods.\\
\\
We make the following generalization.
\begin{defn}
\label{defwse}
Let $I$ be an ideal of an affinoid $K$-algebra $A$, and let $\iota: \Sp A/I\to \Sp A$ be the natural closed embedding. We call a subset $\{d_1, \dots, d_m\}$ of $L=\mathrm{Der}_K(A)$ a \textbf{weakly $I$-standard basis} if 
\begin{enumerate}[(i)]
\item $\{d_1, \dots, d_m\}$ is an $A$-module basis for $L$,
\item there is a generating set $\{f_1, \dots, f_r\}$ for $I$ with $r\leq m$ such that 
\begin{enumerate}
\item for all $1\leq i\leq r$, the image of $d_i(f_i)$ in $A/(f_1, \dots, f_i)$ is a unit, and $d_i(d_i(f_i))=0$ in $A$.
\item for all $1\leq i\leq m$, $1\leq j\leq r$ with $i>j$, we have $d_i(f_j)\in (f_1, \dots, f_j)$.
\item for all $1\leq i\leq j\leq m$, we have 
\begin{equation*}
[d_i, d_j]\in \begin{cases}\oplus_{s=i}^m Ad_s \ \text{if} \ i\leq r\\
\oplus_{s=r+1}^m Ad_s \ \text{if} \ i>r.
\end{cases} 
\end{equation*}
(In other words: for all $1\leq i\leq r+1$, the $A$-submodule $\oplus_{s=i}^m Ad_s$ is closed under the Lie bracket.)
\end{enumerate}
\end{enumerate}
We say that $\iota$ admits a weakly standard basis (or that $\iota$ is a \textbf{weakly standard embedding}) if there exists a weakly $I$-standard basis for $L$.
\end{defn}
It is straightforward to verify that standard embeddings are weakly standard -- the only non-trivial step concerns (ii).(c), where we observe that for standard bases,
\begin{equation*}
\oplus_{s=i}^m Ad_s=\{d\in \mathrm{Der}_K(A): d(f_k)=0 \ \forall \ k< i\}
\end{equation*}
whenever $i\leq r+1$.\\
\\
Let $\iota: \Sp A/I\to \Sp A$ be a weakly standard embedding, using the same notation as in the definition above. We make the following observations:
\begin{enumerate}[(i)]
\item $\Sp A$ is smooth, as $\mathrm{Der}_K(A)$ is a free $A$-module.
\item Let $d\in \mathrm{Der}_K(A)$ such that $d(f_1)\in (f_1)$. Writing $d=\sum a_i d_i$, it follows directly from (ii).(b) in the definition that $a_1d_1(f_1)\in (f_1)$. But since $d_1(f_1)$ is a unit modulo $f_1$, we conclude that $a_1\in (f_1)$, i.e.
\begin{equation*}
\{d\in \mathrm{Der}_K(A): d(f_1)\in (f_1)\}=(f_1)d_1\oplus (\oplus_{i=2}^m Ad_i).
\end{equation*}
In particular, $B:=A/(f_1)$ is smooth, with
\begin{equation*}
\mathrm{Der}_K(B)=\frac{\{d\in \mathrm{Der}_K(A): d(f_1)\in (f_1)\}}{(f_1)\mathrm{Der}_K(A)}\cong B\otimes_A (\oplus_{i=2}^m Ad_i).
\end{equation*}
\item Note that in this case, the natural morphisms $\Sp A/I\to \Sp B$ and $\Sp B\to \Sp A$ are also weakly standard, with the obvious choice of generators. In fact, more generally $\Sp A/I\to \Sp A/(f_1, \dots, f_i)$ and $\Sp A/(f_1, \dots, f_i)\to \Sp A$ are weakly standard for any $1\leq i\leq r$. We deduce by induction that any weakly standard embedding is the composition of weakly standard embeddings of codimension 1 (meaning $r=1$). In general, however, it is not true that the composition of weakly standard embeddings is weakly standard.
\end{enumerate}

Given a weakly standard embedding $\iota: \Sp B\to \Sp A$ with vanishing ideal $I$ and weakly $I$-standard basis $d_1, \dots, d_m$, let $\A$ be an admissible affine formal model and let $\L'=\oplus_{i=1}^m \A d_i$. Note that we can assume that $\L'$ is a Lie lattice by rescaling the basis if necessary. Then $\B=\A/(\A\cap I)$ is an admissible affine formal model for $B$, and 
\begin{equation*}
\L:=\B\otimes_\A (\oplus_{i=r+1}^m \A d_i)
\end{equation*}  
is an $(R, \B)$-Lie lattice inside $\mathrm{Der}_K(B)=\oplus_{i=r+1}^m Bd_i$.
\begin{lem}
\label{rescaledlattice}
Let $\iota: \Sp B\to \Sp A$ be a weakly standard embedding as above, with $d_1, \dots, d_m$ a weakly $I$-standard basis spanning an $(R, \A)$-Lie lattice. Then for any non-negative integers $n_1\geq n_2\geq \dots n_r\geq n\geq 0$, the set $\{\pi^{n_1}d_1, \dots, \pi^{n_r}d_r, \pi^nd_{r+1}, \dots, \pi^nd_m\}$ is also a weakly $I$-standard basis spanning an $(R, \A)$-Lie lattice.  
\end{lem}
\begin{proof}
It follows directly from the definition that the given set is still a weakly $I$-standard basis. We only need to verify that the $\A$-module 
\begin{equation*}
(\oplus_{i=1}^r \pi^{n_i}\A d_i)\oplus (\oplus_{i=r+1}^m \pi^n\A d_i)
\end{equation*}
is still closed under the Lie bracket. But this is a straightforward consequence of Definition \ref{defwse}.(ii).(c).
\end{proof}

The following lemma provides the link between section 2 and the current section: the tubular neighbourhood of an affinoid of Kiehl standard form admits a weakly standard embedding into a polydisc.

\begin{lem}
\label{tubeiswse}
Let $t\in T_m$, $P\in T_m[Y]$, $s\in \mathbb{Z}$, and $\epsilon \in K^\times$. Then the natural surjection
\begin{align*}
T_{m+3}\cong T_{m+2}\langle \pi^sY\rangle &\to T_m\langle t^{-1}\rangle \langle \pi^sY\rangle \langle \epsilon^{-1}P\rangle \\
x_{m+1}&\mapsto t^{-1}\\
x_{m+2}&\mapsto \epsilon^{-1}P
\end{align*}
yields a weakly standard embedding $\Sp T_m\langle t^{-1}\rangle \langle \pi^sY\rangle \langle \epsilon^{-1}P\rangle \to \Sp T_{m+3}$.
\end{lem}
\begin{proof}
The natural quotient map in the statement of the lemma has kernel $I=(tx_{m+1}-1, x_{m+2}-\epsilon^{-1}P)$. Write $\partial_i=\frac{\mathrm{d}}{\mathrm{d}x_{i}}$ for each $1\leq i\leq m+2$, $\partial_Y=\frac{\mathrm{d}}{\mathrm{d}Y}$. Consider the following basis of $\mathrm{Der}_K(T_{m+3})$:
\begin{itemize}
\item $\mathrm{d}_1=\partial_{m+1}$
\item $\mathrm{d}_2=\partial_{m+2}$
\item $\mathrm{d}_3=\partial_Y+\epsilon^{-1}\partial_Y(P)\partial_{m+2}$
\item $\mathrm{d}_4=\partial_1-\partial_1(t)x^2_{m+1}\partial_{m+1}+\epsilon^{-1}\partial_1(P)\partial_{m+2}$
\item $\mathrm{d}_5=\partial_2-\partial_2(t)x^2_{m+1}\partial_{m+1}+\epsilon^{-1}\partial_2(P)\partial_{m+2}$\\
 $\vdots$
\item $\mathrm{d}_{m+3}=\partial_m-\partial_m(t)x^2_{m+1}\partial_{m+1}+\epsilon^{-1}\partial_m(P)\partial_{m+2}$.
\end{itemize}
We verify that this is a weakly $I$-standard basis:\\
Firstly, $\partial_{m+1}(tx_{m+1}-1)=t$ is a unit modulo $tx_{m+1}-1$, and $\partial_{m+1}(t)=0$, as $t\in T_m$. Also, $\partial_{m+2}(x_{m+2}-\epsilon^{-1}P)=1$, as $P\in T_m[Y]$. Thus (ii).(a) is satisfied.\\
\\
For (ii).(b), $\partial_{m+2}(tx_{m+1}-1)=0$ and $(\partial_Y+\epsilon^{-1}\partial_Y(P)\partial_{m+2})(tx_{m+1}-1)=0$. Moreover, for $1\leq i\leq m$, we have
\begin{align*}
(\partial_i-\partial_i(t)x_{m+1}^2\partial_{m+1}+\epsilon^{-1}\partial_i(P)\partial_{m+2})(tx_{m+1}-1)&=\partial_i(t)x_{m+1}-\partial_i(t)x_{m+1}^2t\\&=\partial_i(t)x_{m+1}(1-tx_{m+1}).
\end{align*}
We also obtain $(\partial_Y+\epsilon^{-1}\partial_Y(P)\partial_{m+2})(x_{m+2}-\epsilon^{-1}P)=0$ and
\begin{equation*}
(\partial_i-\partial_i(t)x_{m+1}^2\partial_{m+1}+\epsilon^{-1}\partial_i(P)\partial_{m+2})(x_{m+2}-\epsilon^{-1}P)=0
\end{equation*}
for $1\leq i\leq m$ immediately. Thus (ii).(b) is satisfied.\\
\\
It remains to check (ii).(c). There is nothing to show for bracketing with $\partial_{m+1}$. $\partial_{m+2}$ commutes with the remaining elements of the basis, as all $\partial_i$ commute with each other and none of the other terms involve an $x_{m+2}$.\\
Now 
\begin{align*}
[\partial_Y+\epsilon^{-1}\partial_Y(P)\partial_{m+2}, \partial_i-\partial_i(t)x_{m+1}^2\partial_{m+1}+\epsilon^{-1}\partial_i(P)\partial_{m+2}]=\\
\epsilon^{-1}\partial_Y\partial_i(P)\partial_{m+2}-\partial_i(\epsilon^{-1}\partial_Y(P))\partial_{m+2}=0
\end{align*}
for all $1\leq i\leq m$, and 
\begin{align*}
[ \partial_i-\partial_i(t)x_{m+1}^2\partial_{m+1}+\epsilon^{-1}\partial_i(P)\partial_{m+2},  \partial_j-\partial_j(t)x_{m+1}^2\partial_{m+1}+\epsilon^{-1}\partial_j(P)\partial_{m+2}]=\\
-\partial_i\partial_j(t)x_{m+1}^2\partial_{m+1}+\epsilon^{-1}\partial_i\partial_j(P)\partial_{m+2}+2\partial_i(t)\partial_j(t)x_{m+1}^3\partial_{m+1}\\+\partial_j\partial_i(t)x_{m+1}^2\partial_{m+1}-2\partial_j(t)\partial_i(t)x_{m+1}^3\partial_{m+1}-\epsilon^{-1}\partial_j\partial_i(P)\partial_{m+2}=0
\end{align*}
for $1\leq i\leq j\leq m$.\\
Thus (ii).(c) is satisfied and we have produced a weakly standard basis.
\end{proof}

\subsection{Kashiwara's equivalence for intermediate rings}

We now study in detail weakly standard embeddings of codimension $1$.

Let $\iota: \Sp B=\Sp A/I\to \Sp A$ be a weakly standard embedding, and suppose that $r=1$, so that $I=(f)$ is principal. Let $\A$ be an affine formal model of $A$. Let $d_1, \dots, d_m$ be a weakly $I$-standard basis spanning an $(R, \A)$-Lie lattice, 
\begin{equation*}
\L':=\oplus_{i=1}^m \A d_i.
\end{equation*}
Let $\B=\A/(\A\cap I)$, and $\L=\B\otimes_\A (\oplus_{i=2}^m \A d_i)$, an $(R, \B)$-Lie lattice inside $\mathrm{Der}_K(B)$. We can thus form the rings
\begin{equation*}
D'_n=\h{U_\A(\pi^n\L')}_K
\end{equation*}
and 
\begin{equation*}
D_n=\h{U_\B(\pi^n\L)}_K.
\end{equation*}
Note that $\C:=\oplus_{i=2}^m \A d_i$ is also an $(R, \A)-$Lie lattice in the $(K, A)$-Lie--Rinehart algebra $\oplus_{i=2}^r Ad_i$ thanks to the property (ii).(c), and we have
\begin{equation*}
D_n\cong A/I\otimes_A\h{U_\A(\pi^n\C)}_K.
\end{equation*} 
As $K$-Banach spaces, we have isomorphisms
\begin{equation*}
D'_n\cong A\langle \pi^nd_1, \dots, \pi^nd_m\rangle, \ D_n\cong \frac{A}{I}\langle \pi^nd_2, \dots \pi^nd_m\rangle.
\end{equation*}
Now fix $n\geq 1$. For any $j\geq n$, let $\L_{j, n}=\pi^j \A d_1 \oplus (\oplus_{i=2}^m  \pi^n \A d_i)$. Note that by Lemma \ref{rescaledlattice}, the lattice $\L_{j, n}$ is also an $(R, \A)$-Lie lattice in $\mathrm{Der}_K(A)$. We now define the intermediate rings
\begin{equation*}
D'_{j, n}=\h{U_\A(\L_{j, n})}_K, \ D'_{\infty, n}=\varprojlim_j D'_{j, n}.
\end{equation*} 
As a $K$-Banach space, we have
\begin{equation*}
D'_{j, n}\cong A\langle \pi^jd_1, \pi^nd_2 \dots, \pi^nd_m\rangle.
\end{equation*}
Note that $\h{U_\A(\pi^n\C)}_K\cong A\langle \pi^nd_2, \dots, \pi^nd_m\rangle$ is naturally a $K$-subalgebra of $D'_{j, n}$ for each $j$, and hence also of $D'_{\infty, n}$.
\begin{prop}
\label{DninfFS}
The $K$-algebra $D'_{\infty,n}$ is a Fr\'echet--Stein algebra nuclear over $\h{U_{\A}(\pi^n\C)}_K$. Any coadmissible $D'_{\infty, n}$-module is nuclear relative to $\h{U_\A(\pi^n\C)}_K$ and of countable type.
\end{prop}

\begin{proof}
By \cite[Lemma 2.5]{SixOp}, $D'_{j,n}$ is a Noetherian Banach $K$-algebra for all $j\geq n\geq 1$.

By the proof of \cite[Theorem 4.1.11]{Ardakov}, the transition maps $D'_{j+1, n}\to D'_{j, n}$ are flat on both sides for each $j$. Each of the transition maps has dense image, as the subring of finite-order differential operators is dense in $D'_{j, n}$ for each $j$. Thus $D'_{\infty, n}=\varprojlim D'_{j, n}$ is a Fr\'echet--Stein algebra.

As a Banach module over $\h{U_\A(\pi^n\C)}_K$, $D'_{j, n}$ can be written as 
\begin{equation*}
D'_{j, n}\cong \h{U_\A(\pi^n\C)}_K\h{\otimes}_K K\langle \pi^jd_1\rangle
\end{equation*}
such that the transition map $D'_{ j+1, n}\to D'_{ j, n}$ is induced by the strictly completely continuous morphism $K\langle \pi^{j+1}d_1\rangle \to K\langle \pi^jd_1\rangle$. Thus $D'_{\infty, n}$ is nuclear over $\h{U_{\A}(\pi^n\C)}_K$, and as remarked in section 3, this implies that coadmissible $D'_{\infty, n}$-modules are nuclear over $\h{U_{\A}(\pi^n\C)}_K$. Since $D'_{j,n}$ is a Banach space of countable type for each $j$, coadmissible $D'_{\infty, n}$-modules are also of countable type.
\end{proof}

\begin{lem}
	\label{injectivity}
	Multiplication by $f$ is an injective map $A\to A$. In particular, left multiplication by $f$ yields an injective map
	\begin{equation*}
		D'_{\infty, n}\to D'_{\infty, n}.
	\end{equation*}
\end{lem}
\begin{proof}
	As $A$ is Noetherian and smooth, it decomposes into a finite product of smooth integral domains $A\cong \prod_{i=1}^l A_i$. We write accordingly $f=(f_1, \hdots, f_l)$, and it suffices to show that $f_j\neq 0$ for each $j$.
	
	Fix $j$, let $e_j\in A_j$ denote the corresponding idempotent, and suppose that $f_j=f\cdot e_j=0$. Note that $\partial(e_j)=0$ for any $\partial\in \mathrm{Der}_K(A)$. Thus 
	\begin{equation*}
		0=d_1(f_j)=fd_1(e_j)+d_1(f)e_j=d_1(f)e_j.
	\end{equation*}
	But by definition of weakly standard embeddings, $d_1(f)$ is a unit modulo $f$, so that $e_j$ maps to zero in $A/f$. But this is a contradiction, as
	\begin{equation*}
		\Sp A/f\cong \coprod \Sp A_i/f_i,
	\end{equation*} 
	and $A_j/f_j=A_j$, since we assumed $f_j=0$.
	
	To obtain the analogous result for $D'_{\infty, n}$, simply note that as an $A$-module,
	\begin{equation*}
		D'_{\infty, n}\cong A\h{\otimes}_K (\varprojlim_j K\langle \pi^jd_1, \pi^nd_2, \hdots, \pi^nd_m\rangle)\cong A\h{\otimes}_K K\langle \pi^nd_2, \hdots, \pi^nd_m\rangle \h{\otimes}_K K\{d_1\}
	\end{equation*}
	by applying \cite[Corollary 5.22]{SixOp} twice. Now $K\langle \pi^nd_2, \hdots, \pi^nd_m\rangle$ is a strongly flat object in $\h{\B}c_K$ by \cite[Lemma 4.16]{SixOp}, and $K\{d_1\}$ is a strongly flat object by \cite[Corollary 5.36]{SixOp}.
\end{proof}

Naturally, the same argument can be used to show that right multiplication with $f$ in $D'_{\infty, n}$ is also injective, so that we have short strictly exact sequences
\begin{equation*}
	\begin{xy}
		\xymatrix{0\ar[r]&D'_{\infty, n}\ar[r]^{f\cdot}&D'_{ \infty, n}\ar[r]& \frac{D'_{ \infty, n}}{ID'_{ \infty, n}}\ar[r]&0}
	\end{xy}
\end{equation*}
and
\begin{equation*}
	\begin{xy}
		\xymatrix{0\ar[r]&D'_{ \infty, n}\ar[r]^{\cdot f}&D'_{ \infty, n}\ar[r]& \frac{D'_{ \infty, n}}{D'_{ \infty, n}I}\ar[r]&0.}
	\end{xy}
\end{equation*}

We now define functors between $D_n$-modules and $D'_{\infty, n}$-modules, which should be thought of as mixed-level versions of the pushforward and pullback in Kashiwara's equivalence (see \cite{DcapTwo}, \cite[subsection 9.1]{SixOp}). Let $\mathrm{D}(D_n)$ resp. $\mathrm{D}(D'_{ \infty, n})$ denote the unbounded derived categories of complete bornological modules over $D_n$ resp. $D'_{n, \infty}$. 

Note that 
\begin{equation*}
\frac{D'_{ \infty, n}}{D'_{ \infty, n}I}
\end{equation*}
is a complete bornological $(D'_{ \infty, n}, D_n)$-bimodule, and there is an isomorphism
\begin{equation*}
\frac{D'_{ \infty, n}}{D'_{\infty, n}I}\cong \varprojlim (K\langle \pi^jd_1\rangle \h{\otimes}_K D_n)\cong K\{\mathrm{d}_1\}\h{\otimes}_K D_n
\end{equation*}
of right $D_n$-modules, where the last isomorphism follows from \cite[Corollary 5.22]{SixOp}. We also note that as a left $D'_{ \infty, n}$-module, $\frac{D'_{ \infty, n}}{D'_{ \infty, n}I}$ is coadmissible, because it is finitely presented.

In the same vein,
\begin{equation*}
\frac{D'_{ \infty, n}}{ID'_{ \infty, n}}
\end{equation*}
is a complete bornological $(D_n, D'_{ \infty, n})$-bimodule. 

We consider the pushforward functor
\begin{align*}
\iota_+:& \mathrm{D}(D_n)\to \mathrm{D}(D'_{ \infty, n})\\
&M^\bullet\mapsto \frac{D'_{ \infty,n}}{D'_{ \infty, n}I}\h{\otimes}^\mathbb{L}_{D_n}M^\bullet
\end{align*}
and the pullback functor
\begin{align*}
\iota^!: &\mathrm{D}(D'_{ \infty, n})\to \mathrm{D}(D_n)\\
&N^\bullet\mapsto \frac{D'_{ \infty, n}}{ID'_{ \infty, n}}\h{\otimes}^\mathbb{L}_{D'_{ \infty, n}}N^\bullet[-1].
\end{align*}

\begin{lem}
\label{Kashiwaraequiv}
Let $M$ be a finitely generated $D_n$-module, equipped with its canonical Banach structure. Then the following holds:
\begin{enumerate}[(i)]
\item $\iota_+M$ is a coadmissible $D'_{ \infty, n}$-module.
\item There is a natural isomorphism
\begin{equation*}
M\cong \iota^!\iota_+M.
\end{equation*} 
\end{enumerate}
\end{lem}
\begin{proof}
Since $K\{d_1\}\h{\otimes}_K-$ is exact by \cite[Corollary 5.36]{SixOp}, it follows from the description above that $\iota_+$ is exact, so $\iota_+M$ is concentrated in degree zero, where it is simply
\begin{equation*}
\frac{D'_{ \infty, n}}{D'_{ \infty, n}I}\h{\otimes}_{D_n}M.
\end{equation*}
This is a coadmissible $D'_{ \infty,n}$-module by \cite[Proposition 5.33]{SixOp} (cf. also \cite[Lemma 7.3]{DcapOne}).

For (ii), note that each element in $\iota_+M$ can be written uniquely as 
\begin{equation*}
\sum_{i\in \mathbb{N}} d_1^im_i
\end{equation*}
with $m_i\in M$ satisfying $\pi^{-ji}m_i\to 0$ for all $j$. Recall that $f$ is a generator of $I$ such that $u:=d_1(f)$ is a unit modulo $f$. Moreover, $d_1^2(f)=0$ (i.e. $u$ commutes with $d_1$), and $fm=0$ for all $m\in M$, so multiplication by $f$ on $\iota_+M$ is determined by the formula
\begin{equation*}
fd_1^im=-i\cdot d_1^{i-1}um
\end{equation*}
for any $i\in \mathbb{N}$, $m\in M$.

The short exact sequence
\begin{equation*}
\begin{xy}
\xymatrix{
0\ar[r]& D'_{ \infty, n} \ar[r]^{f\cdot}& D'_{ \infty, n}\ar[r]& \frac{D'_{ \infty, n}}{ID'_{ \infty, n}}\ar[r]& 0
}
\end{xy}
\end{equation*}
yields a free resolution of $\frac{D'_{ \infty,n}}{ID'_{ \infty, n}}$ as a right $D'_{ \infty, n}$-module (where injectivitiy follows from Lemma \ref{injectivity}), so that $\iota^!\iota_+M$ is represented by the complex
\begin{equation*}
\begin{xy}
\xymatrix{
\iota_+M\ar[r]^{f\cdot}& \iota_+M,
}
\end{xy}
\end{equation*}
with the first term appearing in degree zero. It follows from the description of this map above that $\mathrm{H}^0(\iota^!\iota_+M)\cong M$ naturally. Multiplication by $f$ is a strict surjection on $\iota_+M$ by the same argument as in \cite[Lemma 9.3.(ii)]{SixOp}, so $\mathrm{H}^1(\iota^!\iota_+M)=0$.
\end{proof}
\begin{lem}
For any complete bornological $D'_{ \infty, n}$-module $N$, there is a natural isomorphism 
\begin{equation*}
\iota^!N=\frac{D'_{ \infty, n}}{ID'_{ \infty, n}}\h{\otimes}^\mathbb{L}_{D'_{ \infty, n}}N[-1]\cong \mathrm{R}\underline{\mathrm{Hom}}_{D'_{ \infty, n}}\left(\frac{D'_{ \infty, n}}{D'_{ \infty, n}I}, N\right)
\end{equation*}
\end{lem}
\begin{proof}
Use the same resolution as in the proof of the previous lemma.
\end{proof}

\begin{cor}
\label{Kashiwaraadjunction}
There is a natural isomorphism
\begin{equation*}
\mathrm{R}\underline{\mathrm{Hom}}_{D'_{ \infty, n}}(\iota_+M, N)\cong \mathrm{R}\underline{\mathrm{Hom}}_{D_n}(M, \iota^!N)
\end{equation*}
in $\mathrm{D}(\h{\B}c_K)$ for any complete bornological $D_n$-module $M$ and any complete bornological $D'_{ \infty, n}$-module $N$.

In particular, there is a natural isomorphism
\begin{equation*}
\mathrm{R}\underline{\mathrm{Hom}}_{D_n}(M, N)\cong \mathrm{R}\underline{\mathrm{Hom}}_{D'_{ \infty, n}}(i_+M, i_+N)
\end{equation*}
in $\mathrm{D}(\h{\B}c_K)$ for any finitely generated $D_n$-modules $M$ and $N$, equipped with their canonical Banach structures.
\end{cor}
\begin{proof}
This follows from the above and tensor-hom adjunction \cite[Theorem 4.10, Lemma 3.11 and the following remark]{SixOp}. 

The case of finitely generated $D_n$-modules now is a direct consequence of Lemma \ref{Kashiwaraequiv}.
\end{proof}

\begin{lem}
	\label{ipluscommuteswithD}
	Let $M$ be a finitely generated $D_n$-module. Then there is a natural isomorphism
	\begin{equation*}
		\mathrm{R}\underline{\mathrm{Hom}}_{D_n}(M, D_n)\h{\otimes}^{\mathbb{L}}_{D_n}\frac{D'_{\infty, n}}{ID'_{ \infty, n}}\cong \mathrm{R}\underline{\mathrm{Hom}}_{D'_{ \infty, n}}(\iota_+M, D'_{ \infty, n})[1]
	\end{equation*}
	in $\mathrm{D}({D'_{ \infty, n}}^{\mathrm{op}})$, inducing isomorphisms
	\begin{equation*}
		\underline{\mathrm{Ext}}^j_{D_n}(M, D_n)\h{\otimes}_{D_n}\frac{D'_{ \infty, n}}{ID'_{ \infty, n}}\cong \underline{\mathrm{Ext}}^{j+1}_{D'_{ \infty, n}}(\iota_+M, D'_{ \infty, n}).
	\end{equation*}
\end{lem}
\begin{proof}
	Via tensor-hom adjunction, we have the natural isomorphisms
	\begin{align*}
		\mathrm{R}\underline{\mathrm{Hom}}_{D'_{ \infty, n}}(\iota_+M, D'_{ \infty, n})&\cong \mathrm{R}\underline{\mathrm{Hom}}_{D_n}(M, \mathrm{R}\underline{\mathrm{Hom}}_{D'_{ \infty,n}}(\frac{D'_{ \infty, n}}{D'_{ \infty, n}I}, D'_{ \infty, n}))\\
		&\cong \mathrm{R}\underline{\mathrm{Hom}}_{D_n}(M, \frac{D'_{ \infty, n}}{ID'_{ \infty, n}})[-1]
	\end{align*}
	thanks to Lemma \ref{injectivity}.
	
	Since $K\{d_1\}$ is strongly flat in $\h{\B}c_K$ by \cite[Corollary 5.36]{SixOp}, $\frac{D'_{ \infty, n}}{ID'_{ \infty, n}}$ is a strongly flat $D_n$-module. Hence, if $P^\bullet\to M$ is a projective resolution of $M$ by finitely generated projective $D_n$-modules, then
	\begin{equation*}
		\mathrm{R}\underline{\mathrm{Hom}}_{D_n}(M, D_n)\h{\otimes}^{\mathbb{L}}_{D_n}\frac{D'_{ \infty, n}}{ID'_{ \infty, n}}
	\end{equation*}
	is given by $\mathrm{Hom}_{D_n}(P^\bullet, D_n)\h{\otimes}_{D_n}\frac{D'_{ \infty, n}}{ID'_{ \infty, n}}$.
	
	It thus remains to show that the natural morphism
	\begin{equation*}
		\underline{\mathrm{Hom}}_{D_n}(P, D_n)\h{\otimes}_{D_n}\frac{D'_{ \infty, n}}{ID'_{ \infty, n}}\to \underline{\mathrm{Hom}}_{D_n}(P, \frac{D'_{ \infty, n}}{ID'_{ \infty, n}})
	\end{equation*}
	is an isomorphism for $P$ finitely generated projective, but this is immediate from the case where $P$ is free.
\end{proof}
\begin{cor}
	Suppose that $D'_{j, n}$ satisfies the Auslander condition for all $j$. Then $D_n$ satisfies the Auslander condition.
\end{cor}
\begin{proof}
	Let $M$ be a finitely generated $D_n$-module, and let $N\subseteq \mathrm{Ext}^r_{D_n}(M, D_n)$. We wish to show that $\mathrm{Ext}^i_{D_n}(N, D_n)=0$ for all $i<r$.
	
	By exactness, $N\h{\otimes}_{D_n}\frac{D'_{ \infty, n}}{ID'_{ \infty, n}}$ is a right coadmissible submodule of
	\begin{equation*}
		\underline{\mathrm{Ext}}^r_{D_n}(M, D_n)\h{\otimes}_{D_n} \frac{D'_{ \infty, n}}{ID'_{ \infty, n}}\cong \underline{\mathrm{Ext}}^{r+1}_{D'_{ \infty, n}}(i_+M, D'_{ \infty, n}).
	\end{equation*}
	Hence,
	\begin{equation*}
		N\h{\otimes}_{D_n}\frac{D'_{ \infty, n}}{ID'_{ \infty, n}}\h{\otimes}_{D'_{ \infty, n}}D'_{ j,n}
	\end{equation*}
	is a submodule of
	\begin{equation*}
		\underline{\mathrm{Ext}}^{r+1}_{D'_{ \infty, n}}(\iota_+M, D'_{ \infty, n})\h{\otimes}_{D'_{ \infty, n}}D'_{ j, n}\cong \underline{\mathrm{Ext}}^{r+1}_{D'_{ j, n}}(D'_{ j, n}\h{\otimes}_{D'_{ \infty, n}} i_+M, D'_{ j, n})
	\end{equation*}
	by Lemma \ref{coadext}.
	
	As $D'_{ j, n}\h{\otimes}_{D'_{\infty, n}} \iota_+M$ is a finitely generated $D'_{ j, n}$-module and $D'_{ j, n}$ satisfies the Auslander condition, it follows that
	\begin{equation*}
		\underline{\mathrm{Ext}}^i_{D'_{ j, n}}(N\h{\otimes}_{D_n}\frac{D'_{ \infty}}{ID'_{n, \infty, n}}\h{\otimes}_{D'_{ \infty, n}} D'_{ j, n}, D'_{ j, n})=0
	\end{equation*}
	for $i<r+1$.
	
	But it follows again from Lemma \ref{ipluscommuteswithD} and Lemma \ref{coadext} that
	\begin{equation*}
		D'_{ j, n}\h{\otimes}_{D'_{ \infty, n}}\frac{D'_{ \infty,n}}{D'_{ \infty, n}I}\h{\otimes}_{D_n}\underline{\mathrm{Ext}}^i_{D_n}(N, D_n)\cong \underline{\mathrm{Ext}}^{i+1}_{D'_{ j, n}}(N\h{\otimes}_{D_n} \frac{D'_{ \infty, n}}{ID'_{ \infty, n}}\h{\otimes}_{D'_{ \infty, n}}D'_{ j, n}, D'_{ j, n}).
	\end{equation*}
	Hence
	\begin{equation*}
		\frac{D'_{ j, n}}{D'_{ j, n}I}\h{\otimes}_{D_n}\underline{\mathrm{Ext}}^i_{D_n}(N, D_n)=0
	\end{equation*}
	for all $i< r$.
	
	Since $\frac{D'_{ j, n}}{D'_{ j, n}I}\cong K\langle\pi^jd_1\rangle\h{\otimes}_K D_n$ as a right $D_n$-module, we can conclude that $\underline{\mathrm{Ext}}^i_{D_n}(N, D_n)=0$, and $D_n$ satisfies the Auslander condition.
\end{proof}

\begin{cor}
	\label{AGKashiwara}
	If $\mathrm{inj.dim.}D'_{ j, n}\leq d$ for all $j$, then $\mathrm{inj.dim.}D_n\leq d-1$.
\end{cor}
\begin{proof}
	If $M$ is a finitely generated $D_n$-module, then $\iota_+M$ is a coadmissible $D'_{ \infty, n}$-module with
	\begin{equation*}
		\underline{\mathrm{Ext}}^i_{D'_{ \infty, n}}(\iota_+M, D'_{ \infty, n})\cong \varprojlim \underline{\mathrm{Ext}}^i_{D'_{ j, n}}(D'_{ j, n}\h{\otimes}_{D'_{ \infty, n}}\iota_+M, D'_{ j, n})
	\end{equation*}
	by Lemma \ref{coadext}.
	
	Hence if $\underline{\mathrm{Ext}}^i_{D'_{ j, n}}(D'_{ j, n}\h{\otimes} \iota_+M, D'_{ j, n})=0$ for all $j$ and for all $i>d$, then 
	\begin{equation*}
		\underline{\mathrm{Ext}}^i_{D'_{ \infty, n}}(\iota_+M, D'_{ \infty, n})=0
	\end{equation*}
	for all $i>d$.
	
	Thus Lemma \ref{ipluscommuteswithD} implies that
	\begin{equation*}
		\underline{\mathrm{Ext}}^i_{D_n}(M, D_n)\h{\otimes}_{D_n}\frac{D'_{ \infty, n}}{ID'_{ \infty, n}}=0
	\end{equation*}
	for all $i>d-1$.
	
	Since $\frac{D'_{ \infty, n}}{ID'_{ \infty, n}}\cong D_n\h{\otimes}_K K\{d_1\}$ as a left $D_n$-module, it follows that $\underline{\mathrm{Ext}}^i_{D_n}(M, D_n)=0$ for all $i>d-1$, as required.
\end{proof}

\begin{cor}
	\label{clAuslanderpullback}
	If there exists a natural number $d$ such that $D'_{ j, n}$ is Auslander regular of global dimension $\leq d$, then $D_n$ is Auslander regular of global dimension $\leq d-1$.
\end{cor}
\begin{proof}
	It follows from the above that $D_n$ is Auslander--Gorenstein. It thus suffices to show that $D_n$ has global dimension $\leq d-1$. By Lemma \ref{injdimgldim}, it suffices to show that $D_n$ has finite global dimension, as the bound is then determined by Corollary \ref{AGKashiwara}.
	
	Let $M, N$ be finitely generated left $D_n$-modules.
	
	By Corollary \ref{Kashiwaraadjunction}, there is a natural isomorphism
	\begin{equation*}
		\mathrm{R}\underline{\mathrm{Hom}}_{D_n}(M,N)\cong \mathrm{R}\underline{\mathrm{Hom}}_{D'_{ \infty, n}}(\iota_+M, \iota_+N)
	\end{equation*}
	in $\mathrm{D}(\h{\B}c_K)$.
	
	Writing $(\iota_+N)_j$ for the finitely generated left $D'_{ j, n}$-module
	\begin{equation*}
		D'_{ j, n}\h{\otimes}_{D'_{ \infty, n}} \iota_+N,
	\end{equation*}
	Theorem \ref{Bornproperties}.(ii) yields a short strictly exact sequence
	\begin{equation*}
		0\to \iota_+N\to \prod_j (\iota_+N)_j\to \prod_j (\iota_+N)_j\to 0.
	\end{equation*}
	
	Hence we have
	\begin{equation*}
		\mathrm{R}\underline{\mathrm{Hom}}_{D_n}(M, N)\cong \mathrm{holim}_j\  \mathrm{R}\underline{\mathrm{Hom}}_{D'_{j, n}}((\iota_+M)_j,(\iota_+N)_j),
	\end{equation*}
	and \cite[Lemma 3.3]{SixOp} yields a short strictly exact sequence
	\begin{equation*}
		0\to \mathrm{H}^1(\mathrm{R}\varprojlim_j \underline{\mathrm{Ext}}^{i-1}_{D'_{ j, n}}((\iota_+M)_j, (\iota_+N)_j))\to \underline{\mathrm{Ext}}^i_{D_n}(M, N)\to \varprojlim \mathrm{Ext}^i_{D'_{ j, n}}((\iota_+M)_j, (\iota_+N)_j)\to 0.
	\end{equation*}
	
	Since $D'_{ j, n}$ has global dimension $\leq d$, this implies that
	\begin{equation*}
		\mathrm{H}^i(\mathrm{R}\underline{\mathrm{Hom}}_{D_n}(M, N))=0
	\end{equation*}
	for all $i>d+1$. Hence $D_n$ has finite global dimension by Lemma \ref{gldimviaborn}, as required.
\end{proof}

\subsection{Variation: Smooth pullback}
In the previous subsection, we saw how Kashiwara-like calculations can be used to compare Ext groups before and after applying a pushforward functor along a closed embedding. We can see in these arguments the shadow of two basic $\D$-module theoretic principles, well-known in the classical case of algebraic $\D$-modules: the pushforward along a closed embedding is fully faithful (Kashiwara equivalence) and commutes with the duality functor (as closed embeddings are proper).

Classically, duality also commutes with pullback along a smooth morphism. In this subsection, we carry out the analogous steps to the previous argument to show that we can compare Ext groups before and after applying such a pullback.

This turns out to be a rather involved calculation, leading finally to Proposition \ref{smoothAuslander} as an analogue of Corollary \ref{clAuslanderpullback} for smooth pullback.

We consider the following situation: let $Y=\Sp A$ be a smooth affinoid $K$-variety, let $B=A\langle Z\rangle$ and $f: X=\Sp B\to \Sp A$ be the natural projection map. For the purposes of this paper, it will suffice to consider this particular situation, we plan to discuss more general smooth morphisms in a subsequent paper.

Suppose that $\T_Y(Y)=\mathrm{Der}_K(A)$ is a free $A$-module of rank $m$, with basis $d_1, \hdots, d_m$. Let $\A\subseteq A$ be an admissible affine formal model and set $\B=\A\langle Z\rangle$. We can assume without loss of generality that $\L:=\oplus \A\cdot d_i$ is an $(R, \A)$-Lie lattice in $\T_Y(Y)$.

Note that 
\begin{equation*}
	\T_X(X)\cong B\otimes_A \T_Y(Y)\oplus B\cdot \frac{\mathrm{d}}{\mathrm{d}Z}
\end{equation*}
as $(K, B)$-Lie algebras. Writing $\sigma: B\otimes_A\T_Y(Y)\to \T_X(X)$ to denote a section to the natural morphism
\begin{equation*}
	\theta: \T_X(X)\to B\otimes_A \T_Y(Y),
\end{equation*}
it follows that $\sigma(d_1), \hdots, \sigma(d_m), \frac{\mathrm{d}}{\mathrm{d}Z}$ form a $B$-module basis of $\T_X(X)$, with each $\sigma(d_i)$ commuting with $\frac{\mathrm{d}}{\mathrm{d}Z}$.

Let $\partial_1, \hdots, \partial_{m+1}$ be a $B$-module basis of $\T_X(X)=\mathrm{Der}_K(B)$ with the property that $\L':=\oplus \B\partial_i$ is an $(R, \B)$-Lie lattice in $\T_X(X)$ and $\theta$ sends $\partial_i$ to $d_i$ for $i\leq m$ and $\partial_{m+1}$ to $0$. 

In particular (as the kernel of $\theta$ is generated by $\frac{\mathrm{d}}{\mathrm{d}Z}$), there exist elements $g_1, \hdots, g_m\in B$ such that $\partial_i=\sigma(d_i)+g_i\frac{\mathrm{d}}{\mathrm{d}Z}$ for $i\leq m$, and a unit $g\in B$ such that $\partial_{m+1}=g\cdot \frac{\mathrm{d}}{\mathrm{d}Z}$ (since both elements generate the kernel of $\theta$).

We assumme additionally that $g\in \B$ and $g_i\in \B$ for each $1\leq i\leq m$.

Let $\L''=\oplus \B\cdot \sigma(d_i)\oplus \B\cdot \frac{\mathrm{d}}{\mathrm{d}Z}$. By our assumptions above, this is an $(R, \B)$-Lie lattice in $\T_X(X)$ containing $\L'$.

We now write
\begin{equation*}
	D_n=\h{U_\A(\pi^n\L)}_K, \ D'_n=\h{U_\B(\pi^n\L')}_K.
\end{equation*}

We also set $D''_n=\h{U_{\B}(\pi^n\L'')}_ K$ and note that the inclusion $\pi^n\L'\to \pi^n\L''$ induces a natural $K$-algebra morphism $D'_n\to D''_n$.

We briefly explain the role which these various algebras will play in our argument: Suppose we know that $D'_n$ is Auslander regular, and we wish to show that $D_n$ is so, too. Now $D_n$-modules can be lifted back to $D'_n$-modules, corresponding to a $\D$-module inverse image along $X\to Y$, but it is not surprising that there is a much closer link between $D_n$-modules and $D''_n$-modules -- for example, note that the $\sigma(d_i)$ together with $\frac{\mathrm{d}}{\mathrm{d}Z}$ form a $(Z)$-standard basis, corresponding to the closed embedding $Y\to X$.

We will therefore pull back from $D_n$-module to $D'_n$-modules, then base change to $D''_n$-modules, and finally pull back along the closed embedding $Y\to X$, back to $D_n$-modules. Keeping track of how these operations affect Ext groups will allow us to deduce the Auslander regularity of $D_n$.

Note that $B\h{\otimes}_A D_n$ is a complete bornological $(D'_n, D_n)$-bimoude, as the morphism $\theta$ restricts to a morphism $\pi^n\L'\to \B\otimes_{\A}\pi^n\L$, giving $B\h{\otimes}_AD_n$ the structure of a left $D'_n$-module. In fact, as a left $D'_n$-module, the above yields an isomorphism
\begin{equation*}
	B\h{\otimes}_A D_n\cong \frac{D'_n}{D'_n\cdot \partial_{m+1}}.
\end{equation*}

We thus have a functor
\begin{align*}
	f^*: \mathrm{D}(D_n)\to \mathrm{D}(D'_n)\\
	M^\bullet\mapsto (B\h{\otimes}_A D_n)\h{\otimes}^{\mathbb{L}}_{D_n}M^\bullet.
\end{align*}

This is the usual extraordinary inverse image functor for $\D$-modules, see e.g. \cite[subsection 7.3]{SixOp}, except that we ignore a degree shift to simplify our calculations.

\begin{lem}
	\label{pullback}
	\leavevmode
	\begin{enumerate}[(i)]
		\item The functor $f^*$ is exact in the sense that for any complete bornological $D_n$-module $M$, $f^*M$ is concentrated in degree $0$.
		\item The functor $f^*$ sends finitely generated $D_n$-modules to finitely generated $D'_n$-modules.
	\end{enumerate}
\end{lem}
\begin{proof}
	\begin{enumerate}[(i)]
		\item As a right $D_n$-module, $B\h{\otimes}_A D_n\cong K\langle Z\rangle\h{\otimes}_K D_n$, and $K\langle Z\rangle$ is flat by \cite[Lemma 4.16]{SixOp}.
		\item Since $B\h{\otimes}_AD_n\cong \frac{D'_n}{D'_n\partial_{m+1}}$ is finitely generated, it follows from \cite[Lemma 5.32]{SixOp} that
		\begin{equation*}
			(B\h{\otimes}_A D_n)\h{\otimes}_{D_n}M\cong (B\h{\otimes}_A D_n)\otimes_{D_n} M
		\end{equation*}
		is a finitely generated left $D'_n$-module for any finitely generated $D_n$-module $M$.
	\end{enumerate}
\end{proof}

The next two lemmas play a similar role as Lemma \ref{injectivity} in the previous subsection.

\begin{lem}
	\label{transferpitor}
	\leavevmode
	\begin{enumerate}[(i)]
		\item Let $P=\frac{\mathrm{d}}{\mathrm{d}Z}$ or $P=\partial_{m+1}$. Multiplication by $P$ (from the left or the right) is an injective map on $U_B(\T_X(X))$.
		\item Let $r\geq 0$ be such that $\pi^r\frac{\mathrm{d}}{\mathrm{d}Z}\in \pi^n\L'$. The following $R$-modules have bounded $\pi$-torsion:
		\begin{equation*}
			\frac{U_{\B}(\pi^n\L')}{U_{\B}(\pi^n\L')\pi^n\partial_{m+1}}, \frac{U_{\B}(\pi^n\L')}{\pi^n\partial_{m+1}U_{\B}(\pi^n\L')}, \frac{U_{\B}(\pi^n\L')}{U_{\B}(\pi^n\L')\cdot \pi^r\frac{\mathrm{d}}{\mathrm{d}Z}}, \frac{U_{\B}(\pi^n\L')}{\pi^r\frac{\mathrm{d}}{\mathrm{d}Z}U_{\B}(\pi^n\L')},
		\end{equation*}
		as well as
		\begin{equation*}
			\frac{U_{\B}(\pi^n\L'')}{U_{\B}(\pi^n\L'')\pi^n\partial_{m+1}}, \frac{U_{\B}(\pi^n\L'')}{\pi^n\partial_{m+1}U_{\B}(\pi^n\L'')}, \frac{U_{\B}(\pi^n\L'')}{U_{\B}(\pi^n\L'')\cdot\pi^n \frac{\mathrm{d}}{\mathrm{d}Z}}, \frac{U_{\B}(\pi^n\L'')}{\pi^n\frac{\mathrm{d}}{\mathrm{d}Z}U_{\B}(\pi^n\L'')}.
		\end{equation*}
	\end{enumerate}
\end{lem}
\begin{proof}
	\begin{enumerate}[(i)]
		\item This follows immediately from the fact that $\mathrm{gr}U_B(\T_X(X))\cong \mathrm{Sym}_B \T_X(X)$ is a polynomial ring over $B$ by \cite[Theorem 3.1]{Rinehart}.
		\item In the discretely valued case, this is immediate by Noetherianity. In general, we will obtain $\pi$-torsionfreeness whenever we quotient out by a basis element of the lattice, and then obtain bounded $\pi$-torsion by comparing the different lattices.
		
		Explicitly, the filtration on $U_{\B}(\pi^n\L')$ induces a natural filtration on $\frac{U_{\B}(\pi^n\L')}{U_{\B}(\pi^n\L')\pi^n\partial_{m+1}}$, with
		\begin{equation*}
			\mathrm{gr}\frac{U_{\B}(\pi^n\L')}{U_{\B}(\pi^n\L')\pi^n\partial_{m+1}}\cong \frac{\mathrm{Sym}_{\B}\pi^n\L'}{(\mathrm{Sym}_{\B}\pi^n\L')\cdot \pi^n\partial_{m+1}}\cong \B[X_1, \hdots, X_m],
		\end{equation*}
		with $X_i$ the symbol of $\pi^n\partial_i$. Since $\B[X_1, \hdots, X_m]$ is $\pi$-torsionfree, so is $\frac{U_{\B}(\pi^n\L)}{U_{\B}(\pi^n\L')\pi^n\partial_{m+1}}$.
		
		The same argument verifies that $\frac{U_{\B}(\pi^n\L')}{\pi^n\partial_{m+1}U_{\B(\pi^n\L')}}$, $\frac{U_{\B}(\pi^n\L'')}{U_{\B}(\pi^n\L'')\pi^n\frac{\mathrm{d}}{\mathrm{d}Z}}$ and $\frac{U_{\B}(\pi^n\L'')}{\pi^n\frac{\mathrm{d}}{\mathrm{d}Z}U_{\B}(\pi^n\L'')}$ are $\pi$-torsionfree.
		
		For the remaining four quotients, note that we can regard $U_{\B}(\pi^n\L')$ and $U_{\B}(\pi^n\L'')$ as subsets of $U_B(\T_X(X))$, since $\L'$ and $\L''$ are free $\B$-modules. If now $Q\in U_{\B}(\pi^n\L')$ such that
		\begin{equation*}
			\pi^sQ=Q'\cdot \pi^r\frac{\mathrm{d}}{\mathrm{d}Z}
		\end{equation*}
		for some $s\geq r$ and some $Q'\in U_{\B}(\pi^n\L')$, we can use the equation $\partial_{m+1}=g\frac{\mathrm{d}}{\mathrm{d}Z}$ (with $g$ a unit in $B$) to obtain
		\begin{equation*}
			\pi^{s-r}Q=Q'g^{-1}\partial_{m+1}\in U_B(\T_X(X)).
		\end{equation*}
		Letting $l\geq 0$ with $\pi^lg^{-1}\in \B$, we have
		\begin{equation*}
			\pi^{l+n+s-r}Q=Q'\pi^lg^{-1}\pi^n\partial_{m+1}\in U_{\B}(\pi^n\L')\cdot \pi^n\partial_{m+1}.
		\end{equation*}
		But by the previous argument, $\frac{U_{\B}(\pi^n\L')}{U_{\B}(\pi^n\L')\pi^n\partial_{m+1}}$ is $\pi$-torsionfree, so 
		\begin{equation*}
			Q\in U_{\B}(\pi^n\L')\pi^n\partial_{m+1},
		\end{equation*}
		i.e. $Q=Q''\pi^n\partial_{m+1}$ for $Q''\in U_{\B}(\pi^n\L')$. Thus $\pi^{r-n}Q=Q''g\pi^r\frac{\mathrm{d}}{\mathrm{d}Z}$, and the $\pi$-torsion of $\frac{U_{\B}(\pi^n\L')}{U_{\B}(\pi^n\L')\pi^r\frac{\mathrm{d}}{\mathrm{d}Z}}$ is annihilated by $\pi^{r-n}$.
		
		If $Q\in U_{\B}(\pi^n\L)$ with $\pi^sQ=\pi^r\frac{\mathrm{d}}{\mathrm{d}Z}Q'$ for some $Q'\in U_{\B}(\pi^n\L')$ and some $s\geq r$, multiply both sides with $\pi^{n-r}g$ to obtain
		\begin{equation*}
			\pi^{s+n-r}gQ=\pi^n\partial_{m+1}Q'.
		\end{equation*}
		Since $\frac{U_{\B}(\pi^n\L')}{\pi^n\partial_{m+1}(\pi^n\L')}$ is $\pi$-torsionfree, $gQ\in \pi^n\partial_{m+1}U_{\B}(\pi^n\L')$, i.e. $gQ=\pi^n\partial_{m+1}Q''$ for $Q''\in U_{\B}(\pi^n\L')$. Thus $\pi^{r}Q=\pi^{r}g^{-1}\partial_{m+1}Q''= \pi^r\frac{\mathrm{d}}{\mathrm{d}Z}Q''$, and the $\pi$-torsion of $\frac{U_{\B}(\pi^n\L')}{\pi^n\partial_{m+1}\cdot U_{\B}(\pi^n\L')}$ is annihilated by $\pi^r$.
		
		If $Q\in U_{\B}(\pi^n\L'')$ such that 
		\begin{equation*}
			\pi^sQ=Q'\pi^n\partial_{m+1}
		\end{equation*}
		for $s\geq 0$, $Q'\in U_{\B}(\pi^n\L'')$, then $\pi^sQ=Q'\pi^ng\frac{\mathrm{d}}{\mathrm{d}Z}\in U_{\B}(\pi^n \L'')\cdot \pi^n\frac{\mathrm{d}}{\mathrm{d}Z}$, and hence $Q\in U_{\B}(\pi^n\L'')\pi^n\frac{\mathrm{d}}{\mathrm{d}Z}$ by the above, i.e. $Q=Q''\pi^n\frac{\mathrm{d}}{\mathrm{d}Z}$. Thus, as before $\pi^lQ=(Q''\pi^lg^{-1})\pi^n\partial_{m+1}$, so the $\pi$-torsion of $\frac{U_{\B}(\pi^n\L'')}{U_{\B}(\pi^n\L'')\pi^n\partial_{m+1}}$ is annihilated by $\pi^l$, where $l$ is a positive integer chosen such that $\pi^lg^{-1}\in \B$.
		
		Lastly, if $Q\in U_{\B}(\pi^n\L'')$ such that
		\begin{equation*}
			\pi^sQ=\pi^n\partial_{m+1}Q'
		\end{equation*}
		for $s\geq 0$ and $Q'\in U_{\B}(\pi^n\L'')$, then
		\begin{equation*}
			\pi^{s+l}g^{-1}Q=\pi^{l+n}\frac{\mathrm{d}}{\mathrm{d}Z}Q',
		\end{equation*}
		so there exists some $Q''\in U_{\B}(\pi^n\L'')$ such that
		\begin{equation*}
			\pi^lg^{-1}Q=\pi^n\frac{\mathrm{d}}{\mathrm{d}Z}Q''.
		\end{equation*}
		Hence $\pi^lQ=\pi^ng\frac{\mathrm{d}}{\mathrm{d}Z}Q''=\pi^n\partial_{m+1}Q''\in\pi^n \partial_{m+1}U_{\B}(\pi^n\L'')$, as required.
	\end{enumerate}
\end{proof}

\begin{lem}
	\label{transses}
	Let $P=\partial_{m+1}$ or $P=\frac{\mathrm{d}}{\mathrm{d}Z}$. Multiplication by $P$ (from the left or the right) is an injective map on $D'_n$ and on $D''_n$.
\end{lem}
\begin{proof}
	By Lemma \ref{transferpitor}.(i), there is a short exact sequence
	\begin{equation*}
		\begin{xy}
			\xymatrix{0\ar[r]& U_{\B}(\pi^n\L')\ar[r]^{\cdot\pi^n \partial_{m+1}}& U_{\B}(\pi^n\L')\ar[r]& \frac{U_{\B}(\pi^n\L')}{U_{\B}(\pi^n\L')\partial_{m+1}}\ar[r]&0}
		\end{xy}
	\end{equation*}
	Since $\frac{U_{\B}(\pi^n\L')}{U_{\B}(\pi^n\L')\pi^n\partial_{m+1}}$ has bounded $\pi$-torsion by Lemma \ref{transferpitor}.(ii), it follows (similarly to the argument in \cite[Lemma 2.13]{Bode1}) that 
	\begin{equation*}
		0\to U_B(\T_X(X))\to U_B(\T_X(X))\to \frac{U_B(\T_X(X))}{U_B(\T_X(X))\pi^n\partial_{m+1}}\to 0
	\end{equation*}
	is a short strictly exact sequence of semi-normed $K$-vector spaces (with the semi-norm induced by the lattice $U_{\B}(\pi^n\L')$). The sequence thus remains strictly short exact after completion by \cite[Propositions 1.1.9/4, 5]{BGR}. Hence right multiplication with $\pi^n\partial_{m+1}$ is injective on $D'_n$, so the same is true of multiplication by $\partial_{m+1}$.
	
	The other cases work entirely analogously.
\end{proof}

In particular, we have a short strictly exact sequence
\begin{equation*}
	\begin{xy}
		\xymatrix{0\ar[r]& D'_n\ar[r]^{\cdot \partial_{m+1}}&D'_n\ar[r]& \frac{D'_n}{D'_n\partial_{m+1}}\ar[r]&0}
	\end{xy}
\end{equation*}
of complete bornological left $D'_n$-modules, similarly in the other cases.

We note that $\sigma$ induces an injective algebra morphism $D_n\to D''_n$, and the sequence
\begin{equation*}
	\begin{xy}
		\xymatrix{0\ar[r]& D''_n\ar[r]^{\cdot \frac{\mathrm{d}}{\mathrm{d}Z}}&D''_n\ar[r]& \frac{D''_n}{D''_n\frac{\mathrm{d}}{\mathrm{d}Z}}\ar[r]& 0}
	\end{xy}
\end{equation*}
is even a strictly exact sequence of complete bornological $(D''_n, D_n)$-bimodules, since $\frac{\mathrm{d}}{\mathrm{d}Z}$ commutes with $D_n$.

Likewise, the sequence corresponding to left multiplication by $\frac{\mathrm{d}}{\mathrm{d}Z}$ is a strictly exact sequence of $(D_n, D''_n)$-bimodules. This is the main difference between $D''_n$ and $D'_n$ (which does not even carry a natural $D_n$-action), which necessitates the introduction of $D''_n$.

\begin{cor}
	\label{primebasechange}
	\leavevmode
	\begin{enumerate}[(i)]
		\item There is a natural isomorphism
		\begin{equation*}
			D''_n\h{\otimes}^{\mathbb{L}}_{D'_n}f^*D_n\cong \frac{D''_n}{D''_n \frac{\mathrm{d}}{\mathrm{d}Z}}
		\end{equation*}
		in $\mathrm{D}(D''_n\h{\otimes}_K D_n^{\mathrm{op}})$.
		\item The morphism $D'_n\to D''_n$ induces an isomorphism
		\begin{equation*}
			\frac{D'_n}{\frac{\mathrm{d}}{\mathrm{d}Z}D'_n}\cong \frac{D''_n}{\frac{\mathrm{d}}{\mathrm{d}Z}D''_n}
		\end{equation*}
		of complete bornological right $D'_n$-modules.
		\item There is an isomorphism
		\begin{equation*}
			\mathrm{R}\underline{\mathrm{Hom}}_{D'_n}(f^*D_n, D'_n)\cong \frac{D''_n}{\frac{\mathrm{d}}{\mathrm{d}Z}D''_n}[-1]
		\end{equation*}
		in $\mathrm{D}(D_n\h{\otimes}_K {D'_n}^{\mathrm{op}})$.
		\item There is an isomorphism
		\begin{equation*}
			\mathrm{R}\underline{\mathrm{Hom}}_{D'_n}(\frac{D''_n}{\frac{\mathrm{d}}{\mathrm{d}Z}D''_n}, D'_n)\cong f^*D_n[-1]
		\end{equation*}
		in $\mathrm{D}(D'_n\h{\otimes}_K D_n^{\mathrm{op}})$.
	\end{enumerate}
\end{cor}
\begin{proof}
	\begin{enumerate}[(i)]
		\item Firstly, we have an isomorphism of complete bornological left $D''_n$-modules
		\begin{equation*}
			D''_n\h{\otimes}_{D'_n}f^*D_n\cong \frac{D''_n}{D''_n\partial_{m+1}}\cong  \frac{D''_n}{D''_n\frac{\mathrm{d}}{\mathrm{d}Z}},
		\end{equation*}
		since $\partial_{m+1}=g\frac{\mathrm{d}}{\mathrm{d}Z}$.
		
		The section $\sigma$ turns $D''_n$ into a right $D_n$-module, with $\frac{D''_n}{D''_n\frac{\mathrm{d}}{\mathrm{d}Z}}\cong K\langle Z\rangle \h{\otimes}_K D_n$. In particular, the natural map $f^*D_n\to D''_n\h{\otimes}_{D'_n}f^*D_n$ is actually an isomorphism, and $D''_n\h{\otimes}_{D'_n}f^*D_n\cong \frac{D''_n}{D''_n\frac{\mathrm{d}}{\mathrm{d}Z}}$ even as $(D''_n, D_n)$-bimodules. 
		
		The vanishing of other Tor groups follows from the short strictly exact sequences arising from Lemma \ref{transses}.
		\item  Consider the short exact sequence
		\begin{equation*}
			0\to U_{\B}(\pi^n\L')\cap \pi^n\frac{\mathrm{d}}{\mathrm{d}Z}U_{\B}(\pi^n\L'')\to U_{\B}(\pi^n\L')\to \frac{U_{\B}(\pi^n\L'')}{\pi^n\frac{\mathrm{d}}{\mathrm{d}Z}U_{\B}(\pi^n\L'')}\to 0.
		\end{equation*}
		Note that this is indeed exact, since 
		\begin{equation*}
			U_{\B}(\pi^n\L')\to \frac{U_{\B}(\pi^n\L'')}{\pi^n\frac{\mathrm{d}}{\mathrm{d}Z}U_{\B}(\pi^n\L'')}\cong U_{\A}(\pi^n\L)\otimes_{\A}\B
		\end{equation*}
		is surjective: For $i\leq m$, $\partial_i=\sigma(\mathrm{d}_i)+g_i\frac{\mathrm{d}}{\mathrm{d}Z}$ gets mapped to $\sigma(\mathrm{d}_i)-\frac{\mathrm{d}g_i}{\mathrm{d}Z}$, and surjectivity of this filtered morphism can be detected after passing to the associated graded.
		
		Since $\frac{U_{\B}(\pi^n\L'')}{\pi^n\frac{\mathrm{d}}{\mathrm{d}Z}U_{\B}(\pi^n\L'')}$ has bounded $\pi$-torsion by Lemma \ref{transferpitor}, the same reasoning as before yields a strictly exact sequence
		\begin{equation*}
			0\to (U_{\B}(\pi^n\L')\cap \pi^n\frac{\mathrm{d}}{\mathrm{d}Z}U_{\B}(\pi^n\L''))\ \h{}_K\to D'_n\to \frac{D''_n}{\frac{\mathrm{d}}{\mathrm{d}Z}D''_n}\to 0,
		\end{equation*}
		and it remains to verify that the first term is isomorphic to $\frac{\mathrm{d}}{\mathrm{d}Z}D'_n$.
		
		Note that for any $Q\in U_{\B}(\pi^n\L'')$ there exists some $l$ such that $\pi^lQ\in U_{\B}(\pi^n\L')$, so the $R$-module $\frac{U_{\B}(\pi^n\L')\cap \pi^n\frac{}{}U_{\B}(\pi^n\L'')}{\pi^n\frac{}{}U_{\B}(\pi^n\L')}$ is $\pi$-torsion.
		
		Now
		\begin{equation*}
			\frac{U_{\B}(\pi^n\L')\cap \pi^n\frac{\mathrm{d}}{\mathrm{d}Z}U_{\B}(\pi^n\L'')}{\pi^n\frac{\mathrm{d}}{\mathrm{d}Z}U_{\B}(\pi^n\L')}\subseteq \frac{U_{\B}(\pi^n\L')}{\pi^n\frac{\mathrm{d}}{\mathrm{d}Z}U_{\B}(\pi^n\L')}
		\end{equation*}
		has bounded $\pi$-torsion by Lemma \ref{transferpitor}, so there exists an integer $l$ such that
		\begin{equation*}
			\pi^l\left(U_{\B}(\pi^n\L')\cap \pi^n\frac{\mathrm{d}}{\mathrm{d}Z}U_{\B}(\pi^n\L'')\right)\subseteq \pi^n\frac{\mathrm{d}}{\mathrm{d}Z}U_{\B}(\pi^n\L')\subseteq U_{\B}(\pi^n\L')\cap \pi^n\frac{\mathrm{d}}{\mathrm{d}Z}U_{\B}(\pi^n\L'').
		\end{equation*}
		Hence the corresponding semi-norms on $\frac{\mathrm{d}}{\mathrm{d}Z}U_{B}(\T_X(X))$ are equivalent, and
		\begin{equation*}
			(U_{\B}(\pi^n\L')\cap \pi^n\frac{\mathrm{d}}{\mathrm{d}Z}U_{\B}(\pi^n\L''))\ \h{}_K\cong \frac{\mathrm{d}}{\mathrm{d}Z}\cdot D'_n,
		\end{equation*}
		as desired.
		\item Since $f^*D_n\cong \frac{D'_n}{D'_n\frac{\mathrm{d}}{\mathrm{d}Z}}$ as left $D'_n$-modules, the short strictly exact sequence arising from Lemma \ref{transses} yields an isomorphism
		\begin{equation*}
			\mathrm{R}\underline{\mathrm{Hom}}_{D'_n}(f^*D_n, D'_n)\cong \frac{D'_n}{\frac{\mathrm{d}}{\mathrm{d}Z}D'_n}[-1]
		\end{equation*}
		in $\mathrm{D}({D'_n}^{\mathrm{op}})$, which by (ii) above is isomorphic to $\frac{D''_n}{\frac{\mathrm{d}}{\mathrm{d}Z}D''_n}[-1]$. It remains to verify that this is also an isomorphism of left $D_n$-modules.
		
		The natural morphism $D'_n\to D''_n$ induces a map
		\begin{equation*}
			\mathrm{R}\underline{\mathrm{Hom}}_{D_n}(f^*D_n, D'_n)\to \mathrm{R}\underline{\mathrm{Hom}}_{D'_n}(f^*D_n, D''_n)\cong \mathrm{RHom}_{D''_n}(D''_n\h{\otimes}^{\mathbb{L}}_{D'_n}f^*D_n, D''_n),
		\end{equation*}
		fitting into a commutative diagram
		\begin{equation*}
			\begin{xy}
				\xymatrix{\mathrm{R}\underline{\mathrm{Hom}}_{D'_n}(f^*D_n, D'_n)\ar[r]\ar[d]& \frac{D'_n}{\frac{\mathrm{d}}{\mathrm{d}Z}D'_n}[-1]\ar[d]\\
					\mathrm{R}\underline{\mathrm{Hom}}_{D''_n}(D''_n\h{\otimes}^{\mathbb{L}}_{D'_n}f^*D_n, D''_n)\ar[r]& \frac{D''_n}{\frac{\mathrm{d}}{\mathrm{d}Z}D''_n}[-1]}
			\end{xy}
		\end{equation*}
		where the top horizontal arrow and the right vertical arrow are the morphisms used in the first part of the proof.
		
		The left vertical arrow and the bottom horizontal arrow are actually morphisms in $\mathrm{D}(D_n\h{\otimes}_K {D''_n}^{\mathrm{op}})$, so the map
		\begin{equation*}
			\mathrm{R}\underline{\mathrm{Hom}}_{D'_n}(f^*D_n, D'_n)\to \frac{D''_n}{\frac{}{}D''_n}[-1]
		\end{equation*}
		is indeed a morphism in $\mathrm{D}(D_n\h{\otimes}_K {D''_n}^{\mathrm{op}})$, and is an isomorphism by the above.
		\item Apply $\mathrm{R}\underline{\mathrm{Hom}}_{D'_n}(-, D'_n)$ to (iii), noting that Lemma \ref{transses} shows that the natural morphism
		\begin{equation*}
			f^*D_n\to \mathrm{R}\underline{\mathrm{Hom}}_{D'_n}(\mathrm{R}\underline{\mathrm{Hom}}_{D'_n}(f^*D_n, D'_n), D'_n)
		\end{equation*}
		is an isomorphism in $\mathrm{D}(D'_n)$, and hence in $\mathrm{D}(D'_n\h{\otimes}_K D_n^{\mathrm{op}})$.
	\end{enumerate}
\end{proof}

\begin{prop}
	\label{dualoftransfer}
	\leavevmode
	\begin{enumerate}[(i)]
		\item For $M$ a finitely generated left $D_n$-module, there is a natural isomorphism
		\begin{equation*}
			\mathrm{R}\underline{\mathrm{Hom}}_{D'_n}(f^*M, D'_n)\cong \mathrm{R}\underline{\mathrm{Hom}}_{D_n}(M, D_n)\h{\otimes}^{\mathbb{L}}_{D_n} \frac{D''_n}{\frac{\mathrm{d}}{\mathrm{d}Z}D''_n}[-1]
		\end{equation*}
		in $\mathrm{D}({D'_n}^{\mathrm{op}})$.
		
		In particular, there is an isomorphism of complete bornological $K$-vector spaces
		\begin{equation*}
			\underline{\mathrm{Ext}}^i_{D'_n}(f^*M, D'_n)\cong \underline{\mathrm{Ext}}^{i-1}_{D_n}(M, D_n)\h{\otimes}_K K\langle Z \rangle
		\end{equation*}
		for each $i$.
		\item For $N$ a finitely generated right $D_n$-module, there is a natural isomorphism
		\begin{equation*}
			\mathrm{R}\underline{\mathrm{Hom}}_{D'_n}(N\h{\otimes}_{D'_n}\frac{D''_n}{\frac{\mathrm{d}}{\mathrm{d}Z}D''_n}, D'_n)\cong \frac{D''_n}{D''_n\frac{\mathrm{d}}{\mathrm{d}Z}}\h{\otimes}^{\mathbb{L}}_{D_n}\mathrm{R}\underline{\mathrm{Hom}}_{D_n}(N, D_n)[-1]
		\end{equation*}
		in $\mathrm{D}(D'_n)$.
		
		In particular, there is an isomorphism of complete bornological $K$-vector spaces
		\begin{equation*}
			\underline{\mathrm{Ext}}^i_{D'_n}(N\h{\otimes}_{D'_n}\frac{D''_n}{\frac{\mathrm{d}}{\mathrm{d}Z}D''_n}, D'_n)\cong K\langle Z\rangle \h{\otimes}_K\underline{\mathrm{Ext}}^{i-1}_{D_n}(N, D_n)
		\end{equation*}
		for each $i$.
	\end{enumerate}		
\end{prop}
\begin{proof}
	\begin{enumerate}[(i)]
		\item Tensor-hom adjunction yields a natural isomorphism
		\begin{equation*}
			\mathrm{R}\underline{\mathrm{Hom}}_{D'_n}(f^*M, D'_n)\cong \mathrm{R}\underline{\mathrm{Hom}}_{D_n}(M, \mathrm{R}\underline{\mathrm{Hom}}_{D'_n}(f^*D_n, D'_n)).
		\end{equation*}
		By Corollary \ref{primebasechange}.(iii), this is naturally isomorphic to 
		\begin{equation*}
			\mathrm{R}\underline{\mathrm{Hom}}_{D_n}(M, \frac{D''_n}{\frac{\mathrm{d}}{\mathrm{d}Z}\cdot \D''_n})[-1].
		\end{equation*}
		Since $M$ is a finitely generated $D_n$-module, $\mathrm{R}\underline{\mathrm{Hom}}_{D_n}(M, D_n)$ can be represented by a bounded below complex consisting of finitely generated free right $D_n$-modules, equipped with their canonical Banach structures, and is in particular strict. It then follows from the flatness of $K\langle Z\rangle$ that
		\begin{align*}
			\mathrm{R}\underline{\mathrm{Hom}}_{D_n}(M, D_n)\h{\otimes}^{\mathbb{L}}_{D_n}\frac{D''_n}{\frac{\mathrm{d}}{\mathrm{d}Z}D''_n}&\cong \mathrm{R}\underline{\mathrm{Hom}}_{D_n}(M, D_n)\h{\otimes}_K K\langle Z\rangle\\
			&\cong \mathrm{R}\underline{\mathrm{Hom}}_{D_n}(M, D_n\h{\otimes}_K K\langle Z\rangle)\\
			&\cong \mathrm{R}\underline{\mathrm{Hom}}_{D_n}(M, \frac{D''_n}{\frac{\mathrm{d}}{\mathrm{d}Z}\cdot D''_n}),
		\end{align*}
		via the natural morphism, which proves the result.
		\item Since $\frac{D''_n}{\frac{\mathrm{d}}{\mathrm{d}Z}D''_n}\cong D_n\h{\otimes}_K K\langle Z\rangle$ is flat over $D_n$, we can apply tensor-hom adjunction to obtain
		\begin{equation*}
			\mathrm{R}\underline{\mathrm{Hom}}_{D'_n}(N\h{\otimes}_{D_n} \frac{D''_n}{\frac{\mathrm{d}}{\mathrm{d}Z}D''_n}, D'_n)\cong \mathrm{R}\underline{\mathrm{Hom}}_{D_n}(N, \mathrm{R}\underline{\mathrm{Hom}}_{D'_n}(\frac{D''_n}{\frac{\mathrm{d}}{\mathrm{d}Z}D''_n}, D'_n))
		\end{equation*}
		Applying Corollary \ref{primebasechange}.(iv), this is isomorphic to
		\begin{equation*}
			\mathrm{R}\underline{\mathrm{Hom}}_{D_n}(N, f^*D_n)[-1].
		\end{equation*}
		Since $f^*D_n\cong K\langle Z\rangle \h{\otimes}_K D_n$ as right $D_n$-modules, the same argument as above produces the desired isomorphism.
	\end{enumerate}
\end{proof}
\begin{prop}
	\label{smoothAuslander}
	\leavevmode
	\begin{enumerate}[(i)]
		\item If $D'_n$ satisfies the Auslander condition, then so does $D_n$.
		\item If $\mathrm{inj.dim.}D'_n\leq d$, then $\mathrm{inj.dim.}D_n\leq d-1$.
		\item If $D'_n$ has finite global dimension, then so does $D_n$.
	\end{enumerate}
\end{prop}
\begin{proof}
	\begin{enumerate}[(i)]
		
		\item Let $M$ be a finitely generated left $D_n$-module, and let $N\subseteq \underline{\mathrm{Ext}}^i_{D_n}(M, D_n)$. We wish to show that $\underline{\mathrm{Ext}}^j_{D_n}(N, D_n)=0$ for all $j<i$.
		
		Note that
		\begin{equation*}
			N\h{\otimes}_{D_n}\frac{D''_n}{\frac{\mathrm{d}}{\mathrm{d}Z}D''_n}\subseteq \underline{\mathrm{Ext}}^i_{D_n}(M, D_n)\h{\otimes}_{D_n}\frac{D''_n}{\frac{\mathrm{d}}{\mathrm{d}Z}D''_n}
		\end{equation*}
		by flatness of $K\langle Z\rangle$. By Proposition \ref{dualoftransfer}.(i), the right-hand side is naturally isomorphic to $\underline{\mathrm{Ext}}^{i+1}_{D'_n}(f^*M, D'_n)$. Since $D'_n$ satisfies the Auslander condition, this shows that
		\begin{equation*}
			\underline{\mathrm{Ext}}^j_{D'_n}(N\h{\otimes}_{D_n}\frac{D''_n}{\frac{\mathrm{d}}{\mathrm{d}Z}D''_n}, D'_n)=0
		\end{equation*}
		for any $j<i+1$.
		
		But by Proposition \ref{dualoftransfer}.(ii), we thus have
		\begin{equation*}
			\underline{\mathrm{Ext}}^j_{D_n}(N, D_n)\h{\otimes}_K K\langle Z\rangle=0
		\end{equation*}
		for any $j<i$. But then $\underline{\mathrm{Ext}}^j_{D_n}(N, D_n)$, which is a quotient of the above, is also equal to $0$, proving the Auslander condition for $D_n$.
		\item This is now immediate from the above identifications: if $\underline{\mathrm{Ext}}^i_{D'_n}(M, D'_n)=0$ for all $i>d$ and any finitely generated $D'_n$-module $M$, then Proposition \ref{dualoftransfer} shows that $\underline{\mathrm{Ext}}^i_{D_n}(M, D_n)\h{\otimes}_K K\langle Z\rangle=0$ for any $i>d-1$ and any finitely generated $D_n$-module $M$. We can thus conclude as before that $\underline{\mathrm{Ext}}^i_{D_n}(M, D_n)=0$, proving the result.
		\item Let $M$ be a finitely generated left $D_n$-module. Then $f^*M$ is a finitely generated left $D'_n$-module by Lemma \ref{pullback}. If $D'_n$ has finite global dimension, then $f^*M$ admits a finite resolution by finitely generated $D'_n$-modules which are summands of free finitely generated $D'_n$-modules, write $P^\bullet\cong f^*M$.
		
		Now consider the isomorphism
		\begin{equation*}
			M\cong \frac{D''_n}{ZD''_n}\h{\otimes}^{\mathbb{L}}_{D''_n}\frac{D''_n}{D''_n\frac{\mathrm{d}}{\mathrm{d}Z}}\h{\otimes}^{\mathbb{L}}_{D_n}M,
		\end{equation*}
		coming from the fact that 
		\begin{equation*}
			M\cong \frac{A\langle Z\rangle}{ZA\langle Z\rangle}\h{\otimes}^{\mathbb{L}}_{A\langle Z\rangle} A\langle Z\rangle \h{\otimes}^{\mathbb{L}}_A M.
		\end{equation*}
		
		By Corollary \ref{primebasechange}.(i), we have $\frac{D''_n}{D''_n\frac{\mathrm{d}}{\mathrm{d}Z}}\cong D''_n\h{\otimes}^{\mathbb{L}}_{D'_n}f^*D_n$, so we can rewrite the above as
		\begin{equation*}
			M\cong \frac{D''_n}{ZD''_n}\h{\otimes}^{\mathbb{L}}_{D''_n}(D''_n\h{\otimes}^{\mathbb{L}}_{D'_n}f^*M).
		\end{equation*}
		
		Thus $M$ can be represented by the finite complex
		\begin{equation*}
			\frac{D''_n}{ZD''_n}\h{\otimes}_{D'_n}P^\bullet,
		\end{equation*}
		and it suffices to prove that each term in this complex is a projective object in $\mathrm{Mod}_{\h{\B}c_K}(D_n)$. By the assumption on $P^\bullet$, we can reduce to the case of a free $D'_n$-module of rank one, and it suffices to prove that 
		\begin{equation*}
			\frac{D''_n}{ZD''_n}
		\end{equation*}
		is a projective complete bornological left $D_n$-module.
		
		But $\frac{D''_n}{ZD''_n}\cong D_n\h{\otimes}_K K\langle \frac{\mathrm{d}}{\mathrm{d}Z}\rangle$. Since $K\langle \frac{\mathrm{d}}{\mathrm{d}Z}\rangle$ is a flat projective object in $\h{\B}c_K$ by \cite[Lemma 4.16]{SixOp}, the result follows.
	\end{enumerate}
\end{proof}

\subsection{Proof of Theorem \ref{MainThm}}
\begin{defn}
Let $\iota: X=\Sp A/I\to Y=\Sp A$ be a weakly standard embedding of affinoid $K$-varieties of codimension $r$. Let $\B$ be an affine formal model of $A/I$ and let $\L$ be an $(R, \B)$-Lie lattice inside $\T_X(X)$. We say that the pair $(\B, \L)$ is obtained by \textbf{reduction via $\iota$} if there exists an affine formal model $\A$ of $A$ and a weakly $I$-standard basis $d_1, \dots, d_m\in \T_Y(Y)$ such that 
\begin{enumerate}[(i)]
\item $\B\cong \A/(\A\cap I)$ via the natural morphism
\item $\L':=\oplus_{i=1}^m \A d_i$ is an $(R, \A)$-Lie lattice in $\T_Y(Y)$, and
\item $\L\cong \B\otimes_\A\C$, where 
\begin{equation*}
\C=\oplus_{i=r+1}^m \A d_i.
\end{equation*}
\end{enumerate} 
\end{defn}
We also say that $(\B, \L)$ is obtained by reduction via $\iota$ from $(\A, \{d_1, \dots, d_m\})$.

\begin{thm}
\label{standardgldim}

Let $X=\Sp A$ be a smooth affinoid and let $\iota: X\to \mathbb{D}^m=\Sp T_m$ be a weakly standard embedding of codimension $r$ into the $m$-dimensional polydisc. If $\A$ is an admissible affine formal model of $A$, and $\L$ an $(R, \A)$-Lie lattice of $\T_X(X)$ such that the pair $(\A, \L)$ is obtained by reduction via $\iota$, then $\h{U_{\A}(\pi^n\L)}_K$ is Auslander regular for all sufficiently large $n$, with
\begin{equation*}
\mathrm{gl.dim.}\h{U_\A(\pi^n\L)}_K\leq 2\mathrm{dim} X+r.
\end{equation*}
\end{thm}

\begin{proof}
Note that $\h{U_{\A}(\pi^n\L)}_K$ is Noetherian by \cite[Proposition 2.10]{Bodegl}.	
	
Using the same argument as in Proposition \ref{varymodel}, we reduce to the case where $\A$ is the image of $\R\langle x\rangle$. We show that in this case, the theorem holds for all $n\geq 1$.	

By assumption, there exists a weakly standard basis $d_1, \hdots, d_m$ of $\mathrm{Der}_K(T_m)$ such that $\oplus_{i=1}^m R\langle x\rangle d_i$ is a Lie lattice and $\L\cong \A\otimes_{R\langle x\rangle} (\oplus_{i=r+1}^m d_i)$.
	
We proceed by induction on $r$. The case $r=0$ was already done in Theorem \ref{regforpolydisc}. 

Let $I$ be the kernel of the surjection $T_m\to A$. Let $f_1, \dots, f_r$ be generators of $I$ exhibiting $d_1, \dots, d_m$ as a weakly $I$-standard basis.

We now consider the ideal $I'=(f_1, \dots, f_{r-1})\subseteq T_m$, the quotient $A'=T_m/I'$, $X'=\Sp A'$, and the corresponding closed embeddings $\iota_1: X\to X'$, $\iota_2: X'\to \mathbb{D}^m$.

Note that $\mathrm{Der}_K(A')$ is freely generated by (the reductions of) $d_r, d_{r+1}, \dots, d_m$. In particular, both $\iota_1$ and $\iota_2$ are weakly standard embeddings with the obvious choice of weakly standard bases and generators. Let 
\begin{equation*}
\A'=R\langle x\rangle/(R\langle x\rangle \cap I')
\end{equation*}
and let $\L_{ j, n}= \pi^j\A'd_r\oplus (\oplus_{i=r+1}^m \pi^n\A'd_i)\subseteq \mathrm{Der}_K(A')$ for $j\geq n\geq 1$. Then $(\A', \L_{ j, n})$ is obtained by reduction via $\iota_2$ from $(R\langle x\rangle, \{\pi^jd_1, \pi^jd_2, \hdots, \pi^jd_r, \pi^nd_{r+1}, \hdots, \pi^nd_m\})$ (which is a weakly $I'$-standard basis spanning a Lie lattice thanks to Lemma \ref{rescaledlattice}), and hence by induction, $D'_{ j, n}:=\h{U_{\A'}(\L_{ j, n})}_K$ is Auslander regular of global dimension $\leq 2(\mathrm{dim}X+1)+(r-1)=2\mathrm{dim}X+r+1$ for each $j$.

But now we can apply Corollary \ref{clAuslanderpullback} to deduce that $D_n=\h{U_{\A}(\pi^n\L)}_K$ is Auslander regular of global dimension $\leq 2\mathrm{dim}X+r$.
\end{proof}

We now apply this directly to the weakly standard embedding from Lemma \ref{tubeiswse}:

\begin{cor}
\label{Aregfortube}
Let $t\in T_m$, $P\in T_m[Y]$, $s\in \mathbb{Z}$, $\epsilon\in K^\times$. 

Let $B'=T_m\langle t^{-1}\rangle \langle \pi^sY\rangle \langle \epsilon^{-1}P\rangle$, and let $\B'=R\langle x_1, \dots, x_m\rangle \langle t^{-1}\rangle \langle \pi^sY\rangle \langle \epsilon^{-1}P\rangle$. 

Viewing $\Sp B'$ as an affinoid subdomain of $\Sp T_m\langle \pi^sY\rangle\cong \Sp T_{m+1}$, let $\L'=\oplus_{i=1}^m \B' \partial_i\oplus \B' \partial_Y$, where $\partial_i=\pi^r\frac{\mathrm{d}}{\mathrm{d}x_i}$ for some integer $r\geq0$, and $\partial_Y=\pi^{r'}\frac{\mathrm{d}}{\mathrm{d}Y}$ for some $r'\geq 0$.\\
Then $\h{U_{\B'}(\pi^n\L')}_K$ is Auslander regular for all sufficiently large $n$, with global dimension $\leq 2m+4$.
\end{cor}

\begin{proof}
The lattice $\L'$ can be obtained via reduction from the weakly standard basis in Lemma \ref{tubeiswse}, up to possibly rescaling some elements (which does not affect being a weakly standard basis). Since $\Sp B'$ is of dimension $m+1$ and the weakly standard embedding of Lemma \ref{tubeiswse} was of codimension $2$, the result follows directly from above.
\end{proof}

The proof of Theorem \ref{MainThm} is now almost finished: by Proposition \ref{tubular}, any smooth rigid analytic space admits an admissible covering by affinoid spaces $X_i=\Sp A_i$ of Kiehl standard form, such that the space $\Sp A_i\langle Z\rangle$ is of the shape as in Corollary \ref{Aregfortube} for each $i$. We caution the reader however that the well-behaved lattice produced by Corollary \ref{Aregfortube} is not a priori generated by a basis of $\T_X(X_i)$ and $\frac{\mathrm{d}}{\mathrm{d}Z}$, so that we cannot apply the theory of weakly standard embeddings directly to pass to $\Sp A_i$.

Instead, our final step consists in invoking Proposition \ref{smoothAuslander}.

Let $A=T_m\langle t^{-1}\rangle \langle \pi^sY\rangle/P(Y))$ be of Kiehl standard form as in Proposition \ref{localform}, and let $\epsilon \in K^\times$ such that 
\begin{equation*}
	B':=T_m\langle t^{-1}\rangle \langle \pi^sY\rangle\langle \epsilon^{-1}P\rangle\cong A\langle Z\rangle. 
\end{equation*}

As in subsection 5.3, we need to be careful about the various derivations on $\Sp B'$. Corollary \ref{Aregfortube} yields a basis of $\mathrm{Der}(B')$, written $\partial_1, \hdots, \partial_m, \partial_{m+1}:=\partial_{Y}$, obtained from the \'etale morphism
\begin{equation*}
	\Sp B'\to \Sp T_m\langle \pi^sY\rangle.
\end{equation*}

As a second basis, consider the \'etale morphism given by the composition
\begin{equation*}
	\Sp A\to \Sp B'\to \Sp T_m\langle \pi^sY\rangle \to \Sp T_m.
\end{equation*}
Pulling back $\pi^r\frac{\mathrm{d}}{\mathrm{d}x_i}$ from $\Sp T_m$ yields a basis $\mathrm{d}_i$ of $\mathrm{Der}_K(A)$, which together with $\frac{\mathrm{d}}{\mathrm{d}Z}$ yield a basis of $\mathrm{Der}_K(A\langle Z\rangle)=\mathrm{Der}_K(B')$.

As in subsection 5.3, we write $\sigma(\mathrm{d}_i)\in \mathrm{Der}_K(B')$ for a lift of $\mathrm{d}_i$. Explicitly, we take
\begin{equation*}
	\sigma(\mathrm{d}_i)=\partial_i-\partial_i(P)\cdot \partial_Y(P)^{-1} \partial_Y\in \mathrm{Der}_K(B'),
\end{equation*}
a derivation preserving $(P)$ whose induced derivation on the closed subvariety $\Sp A$ is precisely $\mathrm{d}_i$. Note that the above is well-defined, as $\partial_Y(P)$ can be assumed to be a unit in $B$ by \cite[Hilfssatz 1.5]{Kiehl}.

Similarly,
\begin{equation*}
	\frac{\mathrm{d}}{\mathrm{d}Z}=\epsilon\partial_Y(P)^{-1}\cdot \partial_Y,
\end{equation*}
since the isomorphism $B'\cong A\langle Z\rangle$ sends $\epsilon^{-1}P$ to $Z$.

Rearranging yields the equations
\begin{equation*}
	\partial_{m+1}=\partial_Y=\partial_Y(\epsilon^{-1}P)\frac{\mathrm{d}}{\mathrm{d}Z}
\end{equation*}
and
\begin{equation*}
	\partial_i=\sigma(\mathrm{d}_i)+\epsilon^{-1}\partial_i(P)\frac{\mathrm{d}}{\mathrm{d}Z}.
\end{equation*}

We now choose $r$ large enough so that 
\begin{equation*}
	g_i:=\partial_i(\epsilon^{-1}P)=\pi^r\epsilon^{-1}\frac{\mathrm{d}P}{\mathrm{d}x_i}\in \B'
\end{equation*}
and $r'$ large enough so that
\begin{equation*}
	g:=\partial_{m+1}(\epsilon^{-1}P)=\pi^{r'}\epsilon^{-1}\frac{\mathrm{d}P}{\mathrm{d}Y}\in \B'.
\end{equation*}

Set $\L':=\oplus_{i=1}^{m+1} \B' \partial_i$ and $\L''=\oplus_{i=1}^m \B'\sigma(\mathrm{d}_i)\oplus \B'\frac{\mathrm{d}}{\mathrm{d}Z}$. Replacing $\L'$ by $\pi^n\L'$ and $\L''$ by $\pi^n\L''$ for large $n$, we can assume that $\L'$ and $\L''$ are $(R, \B')$-Lie lattices of $\mathrm{Der}_K(B')$.

We have thus verified that $\L'$, $\L''$ and $\L:=\oplus \A \mathrm{d}_i$ satisfy the assumptions from subsection 5.3.

\begin{thm}
\label{passagefromtube}
Let $A=T_m\langle t^{-1}\rangle \langle \pi^sY\rangle/P(Y))$ be of Kiehl standard form as in Proposition \ref{localform}. Let $\A$ be the image of $R\langle x\rangle\langle t^{-1}\rangle\langle \pi^sY\rangle$ in $A$. Let $\L=\oplus_{i=1}^m \A d_i$ for $d_i$ as above. Then $\h{U_\A(\pi^n\L)}_K$ is Auslander regular for all sufficiently large $n$, of global dimension $\leq 2m+3$.
\end{thm}

\begin{proof}
Combine Corollary \ref{Aregfortube} and Proposition \ref{smoothAuslander}.
\end{proof}

\begin{thm}
Let $X$ be a smooth rigid analytic $K$-variety. Then there exists an admisssible affinoid covering $(X_i)$, $X_i=\mathrm{Sp}A_i$, such that for each $i$, there exists an admissible affine formal model $\A_i\subseteq A_i$ and an $(R, \A_i)$-Lie lattice $\L_i\subseteq \T_X(X_i)$ such that
\begin{equation*}
\h{U_{\A_i}(\pi^n\L_i)}_K
\end{equation*}
is Auslander regular of global dimension $\leq 2\mathrm{dim} X+3$ for all $n\geq 0$.
\end{thm}
\begin{proof}
Combine Theorem \ref{passagefromtube} and Proposition \ref{localform}.
\end{proof}

\begin{defn}
	Let $X=\Sp A$ be a smooth affinoid $K$-variety. We say that $X$ is \textbf{D-regular} if there exists an admissible affine formal model $\A$ and an $(R, \A)$-Lie lattice $\L\subseteq \T_X(X)$ which is free as an $\A$-module such that $\h{U_{\A}(\pi^n\L)}_K$ is Auslander regular for all $n\geq 0$.
	We also call such an $\L$ a \textbf{regular Lie lattice}.
\end{defn}

The above then says that any smooth rigid analytic $K$-variety has an admissible covering by D-regular affinoids, as any affinoid of Kiehl standard form is D-regular.

\section{Applications to coadmissible $\w{\D}$-modules}
\subsection{The sheaves $\w{\D}$ and $\D_n$}
Before giving applications of Theorem \ref{MainThm}, we recall the geometric setup from \cite{DcapOne} and \cite{SixOp}.

Let $X$ be a smooth rigid analytic $K$-variety. There are two ways of embedding $\h{\B}c_K$ into a category which admits a well-behaved sheaf theory on $X$ (compare \cite[subsections 3.4, 4.2, 6.1]{SixOp}).

Firstly, the left heart $LH(\h{\B}c_K)$ is a Grothendieck abelian category with compact projective generators, equipped with a closed symmetric monoidal structure which makes the canonical embedding $I: \h{\B}c_K\to LH(\h{\B}c_K)$ lax symmetric monoidal. 

Secondly, the category $\mathrm{Ind}(\mathrm{Ban}_K)$ is an elementary quasi-abelian category, and the completed tensor product on Banach spaces induces a closed symmetric monoidal structure such that the dissection functor
\begin{align*}
	\mathrm{diss}:&\h{\B}c_K\to \mathrm{Ind}(\mathrm{Ban}_K)\\
	&V\mapsto ``\varinjlim" V_B,
\end{align*}
with the colimit ranging over all bounded $\pi$-adically complete $R$-submodules $B$ of $V$, is again lax symmetric monoidal.

The category of sheaves $\mathrm{Shv}(X, \mathrm{Ind}(\mathrm{Ban}_K))$ is then a closed symmetric monoidal quasi-abelian category admitting flat resolutions. Its left heart is a Grothendieck abelian category, and there are equivalences of closed symmetric monoidal categories
\begin{equation*}
	LH(\mathrm{Shv}(X, \mathrm{Ind}(\mathrm{Ban}_K)))\cong \mathrm{Shv}(X, LH(\mathrm{Ind}(\mathrm{Ban}_K)))\cong \mathrm{Shv}(X, LH(\h{\B}c_K)).
\end{equation*}
In particular, there is an equivalence of their derived categories
\begin{equation*}
	\mathrm{D}(\mathrm{Shv}(X, \mathrm{Ind}(\mathrm{Ban}_K)))\cong \mathrm{D}(\mathrm{Shv}(X, LH(\h{\B}c_K))).
\end{equation*}
and if $\A$ is a monoid object in $\mathrm{Shv}(X, \mathrm{Ind}(\mathrm{Ban}_K))$, then it can also be regarded as a monoid object in $\mathrm{Shv}(X, LH(\h{\B}c_K))$, with an equivalence of derived categories
\begin{equation*}
	\mathrm{D}(\mathrm{Mod}_{\mathrm{Shv}(X, \mathrm{Ind}(\mathrm{Ban}_K))}(\A))\cong \mathrm{D}(\mathrm{Mod}_{\mathrm{Shv}(X, LH(\h{\B}c_K))}(\A)).
\end{equation*}

As $\mathrm{diss}: \h{\B}c_K\to \mathrm{Ind}(\mathrm{Ban}_K)$ and $I: \mathrm{Ind}(\mathrm{Ban}_K)\to LH(\mathrm{Ind}(\mathrm{Ban}_K))$ are fully faithful and $LH(\h{\B}c_K)\cong LH(\mathrm{Ind}(\mathrm{Ban}_K))$, we will generally work in $LH(\h{\B}c_K)$ and regard the other categories as full subcategories, likewise for the corresponding sheaf categories. In particular, we routinely suppress the embedding functors from our notation. 

We adopt the following notation for tensor products: the completed tensor product $\h{\otimes}_K$ on $\h{\B}c_K$ induces the tensor product $-\widetilde{\otimes}_K-=\mathrm{H}^0(-\h{\otimes}^{\mathbb{L}}-)$ on $LH(\h{\B}c_K)$, while the tensor product on $\mathrm{Ind}(\mathrm{Ban}_K)$ is denoted $\overset{\rightarrow}{\otimes}_K$. We note that $\overset{\rightarrow}{\otimes}_K$ is exact, and the equivalence $\widetilde{\mathrm{diss}}: LH(\h{\B}c_K)\to LH(\mathrm{Ind}(\mathrm{Ban}_K))$ identifies $\widetilde{\otimes}_K$ with $\mathrm{H}^0(-\overset{\rightarrow}{\otimes}_K^{\mathbb{L}}-)$. In particular, if $V, W\in \mathrm{Ind}(\mathrm{Ban}_K)$, then
\begin{equation*}
	V\widetilde{\otimes}^{\mathbb{L}}_KW\cong V\widetilde{\otimes}_K W\cong V\overset{\rightarrow}{\otimes}_K W.
\end{equation*}

It might thus seem like no separate notation is necessary, but we caution the reader that for a monoid $A\in \mathrm{Ind}(\mathrm{Ban}_K)$, the relative tensor products $\widetilde{\otimes}_A$ and $\overset{\rightarrow}{\otimes}_A$ may differ (as $I: \mathrm{Ind}(\mathrm{Ban}_K)\to LH(\mathrm{Ind}(\mathrm{Ban}_K))$ does not preserve arbitrary coequalizers). However, \cite[Lemma 3.13]{SixOp} ensures that 
\begin{equation*}
	-\widetilde{\otimes}_A-\cong \mathrm{H}^0(-\overset{\rightarrow}{\otimes}_A^{\mathbb{L}}-),
\end{equation*}
so this distinction disappears on the derived level.

For a monoid $\A\in \mathrm{Shv}(X, LH(\h{\B}c))$, we write $\mathrm{D}(\A)=\mathrm{D}(\mathrm{Mod}_{\mathrm{Shv}(X, LH(\h{\B}c_K))}(\A))$. We will now introduce sheaves of differential operators as monoids in $\mathrm{Shv}(X, LH(\h{\B}c_K))$.

For any admissible open affinoid subvariety $Y=\Sp A\subseteq X$, we set 
\begin{equation*}
	\w{\D}_X(Y)=\w{U_A(\T_X(Y))}=\varprojlim_n \h{U_\A(\pi^n\L)}_K,
\end{equation*}
where $\A\subseteq A$ is an admissible affine formal model of $A$ and $\L\subseteq \T_X(Y)$ is an $(R, \A)$-Lie lattice. By \cite [Theorem 6.4]{SixOp}, this defines a sheaf of complete bornological $K$-algebras $\w{\D}_X$, i.e. a monoid in $\mathrm{Shv}(X, LH(\h{\B}c_K))$ taking values in $\h{\B}c_K$. Likewise, the structure sheaf $\O_X$ can be regarded as a sheaf of complete bornological $K$-algebras.

Suppose from now on that $X=\Sp A$ is affinoid with a free tangent sheaf, and fix an admissible affine formal model $\A\subseteq A$ and an $(R, \A)$-Lie lattice $\L\subseteq \T_X(X)$. Suppose that $\L$ is free as an $\A$-module, and abbreviate
\begin{equation*}
U_n=\h{U_\A(\pi^n\L)}_K.
\end{equation*}
We now have an isomorphism
\begin{equation*}
\w{\D}_X\cong \varprojlim_n \left(\O_X\widetilde{\otimes}_A U_n\right)
\end{equation*}
in $\mathrm{Mod}_{\mathrm{Shv}(X, LH(\h{\B}c_K))}(\O_X)$, but we caution the reader that the sheaf 
\begin{equation*}
\O_X\widetilde{\otimes}_AU_n
\end{equation*}
is not in general a sheaf of algebras: if $U=\Sp B$ is an affinoid subdomain, it is not clear that we can find an admissible affine formal model $\B$ such that the induced action of $\pi^n\L$ preserves $\B$. The example $A=K\langle x\rangle$, $\A=R\langle x\rangle$, $\L=\A \frac{\mathrm{d}}{\mathrm{d}x}$, $B=K\langle \pi^{-1}x\rangle$ may be instructive here.

To resolve this difficulty (and ensure exactness of localization), \cite[subsections 4.6--4.9]{DcapOne} introduces the subsite $X_n=X(\pi^n\L)$ of $\pi^n\L$-accessible subdomains. We will not recall the full definition, but only note the following key properties:
\begin{enumerate}[(i)]
\item the sites $X_n$ form a tower, i.e. any $\pi^n\L$-accessible subdomain is also $\pi^{n+1}\L$-accessible, similarly for coverings.
\item every affinoid subdomain is $\pi^n\L$-accessible for all sufficiently large $n$.
\item every finite affinoid covering is an admissible covering in $X_n$ for all sufficiently large $n$.
\item if $U=\Sp B$ is a $\pi^n\L$-accessible affinoid subdomain of $X$, then it admits an admissible affine formal model $\B\subseteq B$ such that $\B\otimes_\A \pi^n\L$ is an $(R, \B)$-Lie lattice in $\T_X(U)$. In particular, 
\begin{equation*}
B\widetilde{\otimes}_A U_n\cong B\h{\otimes}_AU_n\cong \h{U_\B(\B\otimes_\A\pi^n\L)}_K
\end{equation*}
by \cite[Proposition 2.3]{DcapOne}, making $\D_n:=\O_X\widetilde{\otimes}_A U_n|_{X_n}$ a sheaf of complete bornological $K$-algebras on $X_n$. 
\item if $U=\Sp B$ is a $\pi^n\L$-accessible affinoid subdomain of $X$, then the restriction map $\D_n(X)\to \D_n(U)$ is algebraically flat on both sides (\cite[Theorem 4.9]{DcapOne}).
\end{enumerate} 
In particular, we can write in a slight abuse of notation that $\w{\D}_X\cong \varprojlim \D_n$, meaning $\w{\D}_X|_{X_m}\cong\varprojlim_{n\geq m} \D_n|_{X_m}$ for all $m$. We will quite often omit the restriction symbols for the various subsites and write e.g. $\D_n\widetilde{\otimes}_{\w{\D}_X} \M$ instead of 
\begin{equation*}
\D_n\widetilde{\otimes}_{\w{\D}_X|_{X_n}} \M|_{X_n}
\end{equation*}
when $\M$ is a $\w{\D}_X$-module.

We now briefly review the theory of modules over these sheaves. As $\D_n(U)$ is Noetherian Banach for all $\pi^n\L$-accessible subdomains $U$, and the restriction morphism $\D_n(U)\to \D_n(U')$ is (algebraically) flat for any $\pi^n\L$-accessible subdomains $U'\subseteq U$ by \cite[Theorem 4.9]{DcapOne}, the notion of a coherent $\D_n$-module on $X_n$ is well-behaved (cf. \cite[Theorem 5.5]{DcapOne}).

Equipping the sections of a coherent $\D_n$-module on some affinoid with the canonical Banach structure yields an exact and fully faithful embedding of coherent $\D_n$-modules into $\mathrm{Mod}_{\mathrm{Shv}(X_n, LH(\h{\B}c_K))}(\D_n)$, allowing us to regard $\mathrm{D}^b_{\mathrm{coh}}(\D_n)$ as a full triangulated subcategory of $\mathrm{D}(\mathrm{Mod}_{\mathrm{Shv}(X_n, LH(\h{\B}c_K))}(\D_n))$.

By \cite[Theorem 6.7]{DcapOne} and \cite[Proposition 5.5]{SixOp}, $\w{\D}_X(X)$ is a Fr\'echet--Stein algebra, nuclear over $A$. A $\w{\D}_X$-module $\M\in \mathrm{Mod}_{\mathrm{Shv}(X, LH(\h{\B}c_K))}(\w{\D}_X)$ is called a coadmissible $\w{\D}_X$-module if it is the localisation of a coadmissible $\w{\D}_X(X)$-module, i.e. if $\M(X)$ is a coadmissible $\w{\D}_X(X)$-module (with its canonical Fr\'echet structure) and the natural morphism
\begin{equation*}
\w{\D}_X(U)\widetilde{\otimes}_{\w{\D}_X(X)} \M(X)\to \M(U)
\end{equation*}
is an isomorphism for every affinoid subdomain $U$.

\begin{lem}[{\cite[Lemma 6.6]{SixOp}}]
Let $X$ be an affinoid $K$-variety with free tangent sheaf. A complete bornological $\w{\D}_X$-module $\M$ is coadmissible if and only if $\D_n\widetilde{\otimes}_{\w{\D}_X}\M$ is a coherent $\D_n$-module on $X_n$ for each $n$ and $\M\cong \varprojlim \D_n\widetilde{\otimes}_{\w{\D}_X} \M$ via the natural morphism.
\end{lem}

In \cite{SixOp}, we considered the derived category $\mathrm{D}(\w{\D}_X)=\mathrm{D}(\mathrm{Mod}_{\mathrm{Shv}(X, LH(\h{\B}c_K))}(\w{\D}_X))$ of sheaves of complete bornological $\w{\D}_X$-modules and suggested the following derived analogue of coadmissibility.

\begin{defn}
Let $X$ be a smooth affinoid with free tangent sheaf. An object $\M^\bullet\in \mathrm{D}(\w{\D}_X)$ is called a \textbf{$\C$-complex} if
\begin{enumerate}[(i)]
\item $\M_n^\bullet:=\D_n\widetilde{\otimes}^\mathbb{L}_{\w{\D}_X} \M^\bullet \in \mathrm{D}^b_{\mathrm{coh}}(\D_n)$ for each $n$.
\item the natural morphism 
\begin{equation*}
\mathrm{H}^j(\M^\bullet)\to \varprojlim \mathrm{H}^j(\M_n^\bullet)
\end{equation*}
is an isomorphism for each $j$.
\end{enumerate}
\end{defn}

We denote the full subcategory of $\C$-complexes by $\mathrm{D}_\C(\w{\D}_X)$.\\
We also note that the second condition in the definition can be phrased as follows: for any $m$, the natural morphisms $\M^\bullet|_{X_m}\to \M_n^\bullet|_{X_m}$ for any $n\geq m$ exhibit $\M^\bullet|_{X_m}$ as a homotopy limit of $\M_n^\bullet|_{X_m}$ (see \cite[Corollary 8.17]{SixOp}).

\begin{prop}[{\cite[Proposition 8.4, Corollary 8.6, Proposition 8.2, Proposition 8.5]{SixOp}}]
Let $X$ be a smooth affinoid $K$-variety.
\begin{enumerate}[(i)]
\item The category $\mathrm{D}_\C(\w{\D}_X)$ is a full triangulated subcategory, stable under the truncation functors.
\item If $\M^\bullet\in \mathrm{D}_\C(\w{\D}_X)$, then $\mathrm{H}^j(\M^\bullet)$ is coadmissible for each $j$, with 
\begin{equation*}
\D_n\widetilde{\otimes}_{\w{\D}_X}\mathrm{H}^j(\M^\bullet)\cong \mathrm{H}^j(\D_n\widetilde{\otimes}^\mathbb{L}_{\w{\D}_X}\M^\bullet)
\end{equation*}
for all $n$.
\item If $\M^\bullet\in \mathrm{D}(\w{\D}_X)$, then $\M^\bullet$ is a $\C$-complex if and only if each of its cohomology groups is coadmissible and $\D_n\widetilde{\otimes}^\mathbb{L}_{\w{\D}_X} \M^\bullet$ is bounded for each $n$.

In particular, the definition of $\C$-complex does not depend on the choice of formal model and Lie lattice, and a bounded object $\M^\bullet\in \mathrm{D}^b(\w{\D}_X)$ is a $\C$-complex if and only if each of its cohomology groups is coadmissible.
\end{enumerate}
\end{prop}

It thus becomes straightforward to develop a theory of $\C$-complexes on general smooth rigid analytic $K$-varieties by working locally, but we stress that usually, the sheaves $\D_n$ are only considered on an affinoid, since they depend on a choice of Lie lattice.

As in the classical theory of algebraic $\D$-modules, tensoring with the sheaf of top differentials $\Omega_X:=(\Omega_X^1)^{\wedge \mathrm{dim}X}$ yields an equivalence between left $\w{\D}_X$-modules and right $\w{\D}_X$-modules (side-changing, \cite[Theorem 6.11]{SixOp}).

We can also define analogues of the six $\D$-module operations (see \cite[section 7]{DcapOne}):
If $f: X\to Y$ is a morphism between smooth rigid analytic $K$-varieties, we can form the transfer bimodule 
\begin{equation*}
\w{\D}_{X\to Y}:=\O_X\widetilde{\otimes}_{f^{-1}\O_Y} f^{-1} \w{\D}_Y,
\end{equation*}
a $(\w{\D}_X, f^{-1}\w{\D}_Y)$-bimodule in $\mathrm{Shv}(X, LH(\h{\B}c_K))$.\\
This allows us to define the extraordinary inverse image functor 
\begin{align*}
f^!: &\mathrm{D}(\w{\D}_Y)\to \mathrm{D}(\w{\D}_X)\\
&\M^\bullet \mapsto \w{\D}_{X\to Y} \widetilde{\otimes}^\mathbb{L}_{f^{-1}\w{\D}_Y} f^{-1} \M^\bullet[\mathrm{dim} X-\mathrm{dim}Y],
\end{align*}
as well as the direct image functor for right modules
\begin{align*}
f_+^r: &\mathrm{D}(\w{\D}_X^\mathrm{op})\to \mathrm{D}(\w{\D}_Y^\mathrm{op})\\
&\M^\bullet \mapsto \mathrm{R}f_*(\M^\bullet \widetilde{\otimes}^\mathbb{L}_{\w{\D}_X} \w{\D}_{X\to Y}).
\end{align*}
The direct image functor $f_+$ for left modules is then obtained via side-changing.

\begin{prop}[{\cite[Theorem 9.11, Corollary 9.16]{SixOp}}]
	\label{Ccomplexstable}
	\leavevmode
\begin{enumerate}[(i)]
\item If $f$ is smooth, then $f^!$ preserves $\C$-complexes.
\item If $f$ is projective and $K$ is discretely valued, then $f_+$ preserves $\C$-complexes.
\end{enumerate}
\end{prop}

We also define the duality functor
\begin{align*}
\mathbb{D}=\mathbb{D}_X: &\mathrm{D}(\w{\D}_X)\to \mathrm{D}(\w{\D}_X)^\mathrm{op}\\
&\M^\bullet \mapsto \mathrm{R}\mathcal{H}om_{\w{\D}_X} (\M^\bullet, \w{\D}_X)\widetilde{\otimes}_{\O_X} \Omega_X^{\otimes-1}[\mathrm{dim}X].
\end{align*}

As a first application of Theorem \ref{MainThm}, we can remove the discreteness assumption from \cite[Theorem 9.17, Corollary 9.18]{SixOp}: 

\begin{prop}[{cf. \cite[Theorem 9.17, Corollary 9.18]{SixOp}}]
\label{dualsofCcomplexes}
\leavevmode
\begin{enumerate}[(i)]
\item The duality functor $\mathbb{D}$ preserves $\C$-complexes.
\item If $\M^\bullet\in \mathrm{D}_\C(\w{\D}_X)$, then the natural morphism $\M^\bullet\to \mathbb{D}\mathbb{D}\M^\bullet$ is an isomorphism.
\item If $\M$ is a coadmissible $\w{\D}_X$-module, then $\mathcal{E}xt^j_{\w{\D}_X}(\M, \w{\D}_X)$ is a right coadmissible $\w{\D}_X$-module such that on any affinoid subdomain $U$ with free tangent sheaf, we have
\begin{equation*}
\mathcal{E}xt^j_{\w{\D}_X}(\M, \w{\D}_X)(U)=\underline{\mathrm{Ext}}^j_{\w{\D}_X(U)} (\M(U), \w{\D}_X(U)).
\end{equation*}
\end{enumerate}
\end{prop}
\begin{proof}
	Suppose that $X$ is a D-regular affinoid, so that $\w{\D}_X  \cong \varprojlim \D_n$ with $\D_n(X)$ Auslander regular. In particular, $\mathrm{R}\mathcal{H}om_{\D_n}(-, \D_n)$ sends $\mathrm{D}^b_{\mathrm{coh}}(\D_n)$ to $\mathrm{D}^b_{\mathrm{coh}}(\D_n^{\mathrm{op}})^{\mathrm{op}}$. The rest of the proof is then as in \cite[Theorem 9.17]{SixOp}. For (iii), see \cite[Proposition 5.3]{DcapThree} together with the argument in \cite[Theorem 9.17]{SixOp}.
\end{proof}

\begin{prop}
	\label{Bernstein}
	Let $X$ be a smooth rigid analytic $K$-variety of dimension $d$. If $\M$ is a coadmissible $\w{\D}_X$-module, then $\mathcal{E}xt^j_{\w{\D}_X}(\M, \w{\D}_X)=0$ for all $j>d$.
	
	In particular, Bernstein's inequality holds: $j(\M)\leq d$ for any coadmissible $\w{\D}_X$-module.
	
	If $\N\subseteq \mathcal{E}xt^j_{\w{\D}_X}(\M, \w{\D}_X)$ is a coadmissible right submodule, then $\mathcal{E}xt^i_{\w{\D}_X}(\N, \w{\D}_X)=0$ for all $i<j$.
\end{prop}
\begin{proof}
	In the polydisc case, this follows from Theorem \ref{regforpolydisc}. We can thus repeat the same argument as in \cite[subsection 6.2]{DcapThree}, where the result is given for discretely valued fields.
\end{proof}

Given a morphism $f: X\to Y$ of smooth rigid analytic $K$-varieties, the duality functor allows us to define the inverse image functor $f^+$ by writing $f^+:=\mathbb{D}_Xf^!\mathbb{D}_Y$. In the same way, we have the extraordinary direct image functor 
\begin{equation*}
f_!:=\mathbb{D}_Yf_+\mathbb{D}_X.
\end{equation*}

Lastly, the derived tensor product $\widetilde{\otimes}^\mathbb{L}_{\O_X}$ gives rise to a bifunctor on $\mathrm{D}(\w{\D}_X)$ by \cite[subsection 7.1]{SixOp}. 

We also point out that there is the usual functor
\begin{equation*}
-\widetilde{\otimes}^\mathbb{L}_{\w{\D}_X}-: \mathrm{D}(\w{\D}_X^\mathrm{op})\times \mathrm{D}(\w{\D}_X)\to \mathrm{D}(\mathrm{Shv}(X, LH(\h{\B}c_K))),
\end{equation*}
compare \cite[Proposition 3.35]{SixOp}. There is the following relation between side-changing and tensor products:

\begin{lem}
\label{sidechangetensor}
Let $X$ be a smooth rigid analytic $K$-variety. There is a natural isomorphism
\begin{equation*}
(\Omega_X\widetilde{\otimes}_{\O_X} \M^\bullet)\widetilde{\otimes}^\mathbb{L}_{\w{\D}_X}\N^\bullet\cong \Omega_X\widetilde{\otimes}^\mathbb{L}_{\w{\D}_X}(\M^\bullet\widetilde{\otimes}^\mathbb{L}_{\O_X} \N^\bullet)
\end{equation*}
in $\mathrm{D}(\mathrm{Shv}(X, LH(\h{\B}c_K)))$ for $\M^\bullet, \N^\bullet\in \mathrm{D}(\w{\D}_X)$.
\end{lem}
 \begin{proof}
Suppose first that $\M$, $\N$ are the $LH(\h{\B}c_K)$-sheafifications of $\w{\D}_X$-modules in $\mathrm{Preshv}(X, \h{\B}c_K)$. Then the sheaf $(\Omega_X\widetilde{\otimes}_{\O_X}\M)\widetilde{\otimes}_{\w{\D}_X}\N$ is the cokernel of the natural map
\begin{align*}
\alpha: &(\Omega\widetilde{\otimes}_{\O_X} \M)\widetilde{\otimes}_{\O_X} \w{\D}_X \widetilde{\otimes}_{\O_X} \N\to (\Omega\widetilde{\otimes}_{\O_X} \M) \widetilde{\otimes}_{\O_X} \N\\
& (d\otimes m)\otimes P\otimes n\mapsto (d\otimes m)P\otimes n-(d\otimes m)\otimes Pn,
\end{align*}
while $\Omega\widetilde{\otimes}_{\w{\D}_X}(\M\widetilde{\otimes}_{\O_X} \N)$ is the cokernel of 
\begin{align*}
\beta: &\Omega\widetilde{\otimes}_{\O_X} \w{\D}_X \widetilde{\otimes}_{\O_X} (\M\widetilde{\otimes}_{\O_X} \N)\to \Omega\widetilde{\otimes}_{\O_X} (\M\widetilde{\otimes}_{\O_X} \N)\\
& d\otimes P \otimes (m\otimes n)\mapsto dP\otimes (m\otimes n)-d \otimes P(m\otimes n).
\end{align*}

Note that here we are using implicitly \cite[Lemma 3.20]{SixOp} (i.e. sheafification is strong symmetric monoidal on $\mathrm{Preshv}(X, LH(\h{\B}c_K))$) to describe the tensor products above as suitable sheafifications.

We now verify that the natural isomorphisms between these respective sheaves identify $\alpha$ and $\beta$. This is immediate from writing out explicitly the $\w{\D}_X$-module structures in each case, as in \cite[Proposition 1.2.9]{Hotta}: note that if $P\in \T_X$, then
\begin{equation*}
(d\otimes m)P=dP\otimes m-d\otimes Pm
\end{equation*}
and
\begin{equation*}
P(m\otimes n)=Pm\otimes n+m\otimes Pn
\end{equation*}
for $d\in \Omega_X$, $m\in \M$, $n\in \N$ by definition.

Note that any $\M\in \mathrm{Shv}(X, LH(\h{\B}c_K))$ admits an epimorphism from some flat object which is a sheafification as above: for any $U\subseteq X$, $\M(U)$ admits an epimorphism
\begin{equation*}
	\oplus_{i\in J_U} I(c_0(X_i))\to \M(U)
\end{equation*}
by \cite[Corollary 4.17]{SixOp}, where the Banach spaces $c_0(X_i)$ from \cite[Lemma 4.16]{SixOp} yield flat objects in $LH(\h{\B}c_K)$ thanks to \cite[Lemma 4.19.(v)]{SixOp}. Writing $\F_U^{\mathrm{pre}}$ for the presheaf
\begin{equation*}
	V\mapsto \begin{cases}
		\oplus_{i\in J_U} c_0(X_i)\ \text{if } V\subseteq U\\
		0 \ \text{otherwise,}
	\end{cases}
\end{equation*}
and $\F$ for the sheafification of $\oplus_U \F_U^{\mathrm{pre}}$, we obtain the desired epimorphism
\begin{equation*}
	\w{\D}_X\widetilde{\otimes}_K \F\to \M,
\end{equation*}
and $\w{\D}_X\widetilde{\otimes}_K \F$ is indeed the sheafification of $V\mapsto \oplus_U\w{\D}_X(V)\h{\otimes}_K \F_U^{\mathrm{pre}}(V)$ by \cite[Lemma 3.20, Proposition 4.25, Proposition 4.22.(v)]{SixOp}.

Since any $\M^\bullet\in \mathrm{D}(\mathrm{Shv}(X, LH(\h{\B}c_K)))$ is the homotopy colimit of its truncations $\tau^{\leq n}\M^\bullet$ and $\widetilde{\otimes}^{\mathbb{L}}$ commutes with homotopy colimits (because it commutes with direct sums), we can assume now that $\M^\bullet$ and $\N^\bullet$ are bounded above. 

Picking flat resolution of $\M^\bullet$ and $\N^\bullet$ with each term of the form $\w{\D}_X\widetilde{\otimes}_K \F$ for some flat object $\F\in \mathrm{Shv}(X, LH(\h{\B}c_K))$ such that $\w{\D}_X\widetilde{\otimes}_K \F$ is the sheafification of a complete bornological presheaf, the isomorphism follows from the case discussed above.
\end{proof}

In the following subsections, we will prove several statements about the behaviour of $\C$-complexes under the six operations. The proof strategy can be summarized as follows: Arguing locally, we work on a D-regular affinoid $X$ and prove analogous statements for $\mathrm{D}^b_{\mathrm{coh}}(\D_n)$, exploiting the fact that $\D_n(X)$ is Auslander regular. We then take the limit as $n$ goes to infinity. 

To make this approach work, we recall the following result from \cite{SixOp} -- we refer to \cite[subsection 3.2]{SixOp} for a discussion of the notion of a homotopy limit:

\begin{lem}[{\cite[Corollary 8.17]{SixOp}}]
	\label{Ccomplexesasholim}
	Let $X=\Sp A$ be an affinoid $K$-variety with free tangent sheaf. Let $\A$ be an admissible affine formal model and $\L\subseteq \T_X(X)$ a Lie lattice which is a free $\A$-module, yielding the sheaves $\D_n$ as above.
	
	If $\M^\bullet\in\mathrm{D}_\C(\w{\D}_X)$ is a $\C$-complex and $\M_n^\bullet:=\D_n\widetilde{\otimes}^\mathbb{L}_{\w{\D}_X}\M^\bullet$, then $\M^\bullet\cong \mathrm{holim} \M_n^\bullet$, i.e. the natural morphisms $\M^\bullet\to \M_n^\bullet$ yield a distinguished triangle
	\begin{equation*}
		\M^\bullet\to \prod_{n\geq m} \M_n^\bullet\to \prod_{n\geq m} \M_n^\bullet
	\end{equation*} 
	in $\mathrm{D}(\mathrm{Mod}_{\mathrm{Shv}(X_m, LH(\h{\B}c_K))}(\w{\D}_X|_{X_m}))$ for each $m$.
\end{lem}

We also record the following result, the proof of which is given in the appendix, Corollary \ref{tensorlimitapp}:

\begin{thm}
	\label{tensorlimit}
	Let $X=\Sp A$ be an affinoid $K$-variety with free tangent sheaf. If $\M^\bullet\in \mathrm{D}^-_{\C}(\w{\D}_X^{\mathrm{op}})$, $\N^\bullet\in \mathrm{D}^-_{\C}(\w{\D}_X)$, let $\M_n^\bullet:=\M^\bullet\widetilde{\otimes}_{\w{\D}_X}^{\mathbb{L}}\D_n$, $\N_n^\bullet:=\D_n\widetilde{\otimes}_{\w{\D}_X}^{\mathbb{L}}\N^\bullet$. Then
	\begin{equation*}
		\M^\bullet\widetilde{\otimes}_{\w{\D}_X}^{\mathbb{L}}\N^\bullet\cong \mathrm{holim} (\M_n^\bullet\widetilde{\otimes}_{\w{\D}_X}^{\mathbb{L}}\N^\bullet)\cong \mathrm{holim}(\M_n^\bullet\widetilde{\otimes}_{\D_n}^{\mathbb{L}}\N_n^\bullet),
	\end{equation*}
	in the sense that for each $m$, the natural morphisms yield a distinguished triangle
	\begin{equation*}
		\M^\bullet\widetilde{\otimes}_{\w{\D}_X}^{\mathbb{L}}\N^\bullet\to \prod_{n\geq m} (\M_n^\bullet\widetilde{\otimes}_{\D_n}^{\mathbb{L}}\N_n^\bullet)\to \prod_{n\geq m} (\M_n^\bullet\widetilde{\otimes}_{\D_n}^{\mathbb{L}}\N_n^\bullet)
	\end{equation*} 
	in $\mathrm{D}(\mathrm{Shv}(X_m, LH(\h{\B}c_K)))$.
	
	The same applies if $\M^\bullet\in \mathrm{D}_{\C}^b(\w{\D}_X^{\mathrm{op}})$ and $\N^\bullet\in \mathrm{D}_{\C}(\w{\D}_X)$.
\end{thm}

\begin{rmk}
	It seems that Theorem \ref{tensorlimit} might fail in general for unbounded $\C$-complexes.
	
	If one is willing to work with $\infty$-categories, Lemma \ref{Ccomplexesasholim} suggests an alternative approach: Soor showed in \cite[Theorem 1.6.(iii)]{Soor} that $\mathrm{D}_{\C}(\w{\D}_X)\cong \varprojlim \mathrm{D}^b_{\mathrm{coh}}(\D_n)$ in $\mathrm{Cat}_{\infty}$. This is again consistent with our interpretation of coadmissible $\w{\D}_X$-modules as quantizations of coherent modules on the cotangent bundle.
	
	Following this approach, it would then be natural to define a (potentially different) tensor product $\mathrm{D}_{\C}(\w{\D}_X^{\mathrm{op}})\times \mathrm{D}_{\C}(\w{\D}_X)\to \mathrm{D}(\mathrm{Shv}(X, LH(\h{\B}c_K)))$ by glueing the tensor products over $\D_n$: this yields a tensor product for which Theorem \ref{tensorlimit} holds by definition, even for unbounded complexes.
	
	We content ourselves with bounded above $\C$-complexes here, as it allows us to argue purely with $1$-categories and use the sheaf-theoretic machinery already developed in \cite{SixOp}. This imposes some boundedness assumptions in some of our later statements, but working with $\D_n$-modules for as long as possible before taking the limit minimizes this issue.
\end{rmk}

\subsection{A projection formula for closed embeddings}
In this subsection and the next, we prove Theorem \ref{projandadj}.(i), a projection formula for bounded $\C$-complexes. We will consider separately the cases of a closed embedding and a projection morphism and then invoke composition results.

Consider the following setting:

Let $X=\Sp A$ be a smooth affinoid $K$-variety of Kiehl standard form. Let $\A$ be an admissible affine formal model and $\L\subseteq \T_X(X)$ a regular $(R, \A)$-Lie lattice. Moreover, let $\iota: X\to Y=X\times \mathbb{D}^l=\Sp A\langle z_1, \hdots, z_l\rangle$ be the closed embedding corresponding to the vanishing of all $z_i$. We will throughout use this embedding to view $X$ as a closed subvariety of $Y$. Then $\A'=\A\langle z\rangle$ is an affine formal model for $A\langle z\rangle$, and
\begin{equation*}
\L':=(\A'\otimes_\A \L)\oplus (\oplus_{i=1}^l \A'\cdot \frac{\mathrm{d}}{\mathrm{d}z_i})
\end{equation*}
is an $(R, \A')$-Lie lattice in $\T_Y(Y)$.

Note that $\mathrm{Sp} A\langle z\rangle$ is also of Kiehl standard form and hence D-regular, with $\L'$ as a regular Lie lattice by Theorem \ref{passagefromtube}.

With these choices of lattices, we have sites $Y_m$ of $\pi^m\L'$-accessible subdomains of $Y$ and $X_m$ of $\pi^m\L$-accessible subdomains of $X$, on which the sheaves $\D'_m:=\D_{Y_m}$, resp. $\D_m:=\D_{X_m}$ can be defined. 

Moreover, we set
\begin{equation*}
	\L'_{ m, n}=(\A'\otimes_\A \pi^m\L)\oplus (\oplus \A'\pi^n\frac{\mathrm{d}}{\mathrm{d}z_i})
\end{equation*}
for any $n\geq m\geq 1$. As Corollary \ref{Aregfortube} allows for the rescaling of the $\frac{\mathrm{d}}{\mathrm{d}z_i}$, we know that $\L'_{m, n}$ is still a regular Lie lattice for each $n$. Note that any $\pi^m\L'$-accessible subdomain is also $\L'_{m, n}$-accessible. Analogously to subsection 5.2, we define $\D'_{m, n}$ to be the sheaf of complete bornological $K$-algebras on $Y_m$ given by
\begin{equation*}
	\D'_{m,n}(V)=\h{U_\B(\B\otimes_{\A'}\L_{m,n})}_K
\end{equation*}
for any $V=\Sp B\in Y_n$ with suitable admissible affine formal model $\B$. For example, $\D'_{m,m}=\D'_m$.

We then set
\begin{equation*}
	\D'_{m, \infty}=\varprojlim_n \D'_{m,n},
\end{equation*}
again a sheaf of complete bornological $K$-algebras on $Y_m$.

In fact, $\D'_{m, \infty}(Y)=\varprojlim \D'_{m,n}(Y)$ is a Fr\'echet--Stein algebra, nuclear over the Banach algebra $\h{U_{\A'}(\A'\otimes_{\A}\pi^m\L)}_K$.

By construction, if $U\subseteq Y$ is open in $Y_m$, then $U\cap X$ is open in $X_m$. In particular, we can form the functor
\begin{align*}
\iota_+:&\mathrm{D}(\D_m^\mathrm{op})\to \mathrm{D}(\w{\D}_Y^{\mathrm{op}}|_{Y_m})\\
&\M_m^\bullet \mapsto \mathrm{R}\iota_*(\M_m^\bullet \widetilde{\otimes}^\mathbb{L}_{\w{\D}_X} (\O_X\widetilde{\otimes}^\mathbb{L}_{\iota^{-1}\O_Y} \iota^{-1}\w{\D}_Y)).
\end{align*}
Note that this is isomorphic to $\mathrm{R}\iota_*(\M_m^\bullet\widetilde{\otimes}_K K\{\mathrm{d}\})$ for $\mathrm{d}_i=\frac{\mathrm{d}}{\mathrm{d}z_i}$. 

It is then easy to see that for any $\M_m^\bullet\in \mathrm{D}(\D_m^\mathrm{op})$, 
\begin{equation*}
	\iota_+\M_m^\bullet\cong \mathrm{R}\iota_*(\M_m^\bullet\widetilde{\otimes}^\mathbb{L}_{\D_m} (\O_X\h{\otimes}^\mathbb{L}_{\iota^{-1}\O_Y} \iota^{-1}\D'_{m, \infty}))	
\end{equation*}
is in fact in $\mathrm{D}((\D'_{m, \infty})^\mathrm{op})$.

Moreover, if $\M_m$ is a coherent right $\D_m$-module, then $\iota_+\M_m$ is a coadmissible right $\D'_{m, \infty}$-module, by the same reasoning as in Lemma \ref{Kashiwaraequiv}.

In this way, if $\M_m^\bullet\in \mathrm{D}^b_{\mathrm{coh}}(\D_m)$, we can regard $\iota_+\M_m^\bullet$ as a $\C$-complex of right $\D'_{m, \infty}$-modules. Likewise, we claim that if $\N^\bullet\in \mathrm{D}^-_{\C}(\w{\D}_Y)$, then $\D'_{m, \infty}\widetilde{\otimes}_{\w{\D}_Y}^{\mathbb{L}}\N^\bullet\in \mathrm{D}^-_{\C}(\D'_{m, \infty})$. In fact,
\begin{equation*}
	\D'_{m,n}\widetilde{\otimes}_{\w{\D}_Y}^{\mathbb{L}}\N\cong \D'_{m,n}\widetilde{\otimes}_{\w{\D}_Y}\N
\end{equation*}
for any $\N\in \C_Y$ by \cite[Corollary 5.38]{SixOp} and then
\begin{align*}
	\D'_{m, \infty}\widetilde{\otimes}_{\w{\D}_Y}^{\mathbb{L}}\N&\cong \D'_{m, \infty}\widetilde{\otimes}_{\w{\D}_Y}\N \\
	&\cong \varprojlim \D'_{m,n}\widetilde{\otimes}_{\w{\D}_Y}\N
\end{align*}
by \cite[Proposition 5.33]{SixOp}.

Hence, if $\N^\bullet\in \mathrm{D}^-_{\C}(\w{\D}_Y)$, then
\begin{equation*}
	\mathrm{H}^i(\D'_{m, \infty}\widetilde{\otimes}_{\w{\D}_Y}^{\mathbb{L}}\N^\bullet)\cong \D'_{m, \infty}\widetilde{\otimes}_{\w{\D}_Y} \mathrm{H}^i(\N^\bullet)
\end{equation*}
is a coadmissible $\D'_{m, \infty}$-module with $\D'_{m,n}\widetilde{\otimes}_{\D'_{m, \infty}}\mathrm{H}^i(\D'_{m, \infty}\widetilde{\otimes}_{\w{\D}_Y}^{\mathbb{L}}\N^\bullet)\cong \D'_{m,n}\widetilde{\otimes}_{\w{\D}_Y}\mathrm{H}^i(\N^\bullet)$, so that $\D'_{m, \infty}\widetilde{\otimes}_{\w{\D}_Y}^{\mathbb{L}}\N^\bullet\in \mathrm{D}^-_{\C}(\D'_{m, \infty})$.

We can now formulate a natural analogue of Theorem \ref{tensorlimit} for $\C$-complexes of $\D'_{m, \infty}$-modules, since each $\D'_{m,n}(Y)$ is Auslander reguar.

\begin{lem}
	\label{tensorlimiota}
	Let $\M^\bullet\in \mathrm{D}^-_{\C}(\w{\D}_X^{\mathrm{op}})$ and $\N^\bullet\in \mathrm{D}^-_{\C}(\w{\D}_Y)$ and write $\M_m^\bullet=\M^\bullet\widetilde{\otimes}_{\w{\D}_X}^{\mathbb{L}}\D_m$. Then
	\begin{align*}
		\iota_+(\M^\bullet_m)\widetilde{\otimes}_{\w{\D}_Y}^{\mathbb{L}}\N^\bullet&\cong \iota_+(\M^\bullet_m)\widetilde{\otimes}_{\D'_{m, \infty}}^{\mathbb{L}}(\D'_{m, \infty}\widetilde{\otimes}_{\w{\D}_Y}^{\mathbb{L}}\N^\bullet)\\
		&\cong \mathrm{holim}_{n\geq m} (\iota_+(\M_m^\bullet)_{m,n} \widetilde{\otimes}_{\D'_{m,n}}^{\mathbb{L}} \N_{m,n}^\bullet) 
	\end{align*}
	in $\mathrm{D}(\mathrm{Shv}(Y_m, LH(\h{\B}c_K)))$, where
	\begin{equation*}
		\N_{m,n}^\bullet:=\D'_{m,n}\widetilde{\otimes}_{\w{\D}_Y}^{\mathbb{L}}\N^\bullet
	\end{equation*}
	and
	\begin{equation*}
		\iota_+(\M_m^\bullet)_{m,n}=\iota_+(\M_m^\bullet)\widetilde{\otimes}_{\D'_{m, \infty}}^{\mathbb{L}} \D'_{m,n}.
	\end{equation*}
	
	The same applies if $\M^\bullet\in \mathrm{D}_{\C}^b(\w{\D}_X^{\mathrm{op}})$ and $\N^\bullet\in \mathrm{D}_{\C}(\w{\D}_Y)$.
\end{lem}
\begin{proof}
	Since $\iota_+(\M_m^\bullet)\in \mathrm{D}^-_{\C}({\D'}_{m, \infty}^{\mathrm{op}})$, this is Corollary \ref{tensorlimitapp} with $\mathscr{U}_n=\D'_{m,n}$.
\end{proof}

\begin{lem}
	\label{njbasechange}
	Let $\M^\bullet\in \mathrm{D}_{\C}(\w{\D}_X^{\mathrm{op}})$. 
	\begin{enumerate}[(i)]
		\item There is a natural isomorphism
		\begin{equation*}
			\iota_+\M^{\bullet}\widetilde{\otimes}_{\w{\D}_Y}^{\mathbb{L}}\D'_{m, \infty}\cong \iota_+(\M^{\bullet}\widetilde{\otimes}_{\w{\D}_X}^{\mathbb{L}}\D_m)
		\end{equation*}
		in $\mathrm{D}(\mathrm{Mod}_{\mathrm{Shv}(Y_m, LH(\h{\B}c_K))}({\D'}^{\mathrm{op}}_m))$.
		
		In particular, there is a natural isomorphism
		\begin{equation*}
			\iota_+\M^\bullet\widetilde{\otimes}^{\mathbb{L}}_{\w{\D}_Y}\D'_{m,n}\cong \iota_+(\M_m^\bullet)\widetilde{\otimes}_{\D'_{m, \infty}}^{\mathbb{L}}\D'_{m,n}
		\end{equation*}
		in $\mathrm{D}(\mathrm{Mod}_{\mathrm{Shv}(Y_m, LH(\h{\B}c_K))}({\D'}_{m,n}^{\mathrm{op}}))$ for each $n\geq m\geq 1$.
		\item There is a natural isomorphism
		\begin{equation*}
			\iota_+\M^{\bullet}\widetilde{\otimes}_{\w{\D}_Y}^{\mathbb{L}}\D'_m\cong \mathrm{R}\iota_*(\M^{\bullet}\widetilde{\otimes}_{\w{\D}_X}^{\mathbb{L}} \iota^*\D'_m)\cong \mathrm{R}\iota_*(\M^\bullet\widetilde{\otimes}_{\w{\D}_X}^{\mathbb{L}}\D_m \widetilde{\otimes}_{\D_m}^{\mathbb{L}} \iota^*\D'_m)
		\end{equation*}
		in $\mathrm{D}(\mathrm{Mod}_{\mathrm{Shv}(Y_m, LH(\h{\B}c_K))}({\D'}^{\mathrm{op}}_m))$.
	\end{enumerate}
\end{lem}

\begin{proof}
	\begin{enumerate}[(i)]
		\item Let $I\subseteq A\langle z\rangle$ denote the ideal generated by the $z_i$. If $M$ is a coadmissible right $\w{\D}_X(X)$-module, then 
		\begin{equation*}
			M\widetilde{\otimes}_{\w{\D}_X(X)}^{\mathbb{L}}\frac{\w{\D}_Y(Y)}{I\w{\D}_Y(Y)}\cong M\widetilde{\otimes}_{\w{\D}_X(X)} \frac{\w{\D}_Y(Y)}{I\w{\D}_Y(Y)}
		\end{equation*}
		is a coadmissible right $\w{\D}_Y(Y)$-module by \cite[Proposition 5.33]{SixOp}. It thus follows from \cite[Proposition 5.37]{SixOp} that
		\begin{align*}
			(M\widetilde{\otimes}_{\w{\D}_X(X)}^{\mathbb{L}} \frac{\w{\D}_Y(Y)}{I\w{\D}_Y(Y)})\widetilde{\otimes}_{\w{\D}_Y(Y)}^{\mathbb{L}} \D'_{m, \infty}(Y)&\cong (M\widetilde{\otimes}_{\w{\D}_X(X)}\frac{\w{\D}_Y(Y)}{I\w{\D}_Y(Y)})\widetilde{\otimes}_{\w{\D}_Y(Y)} \D'_{m, \infty}(Y)\\
			&\cong M\widetilde{\otimes}_{\w{\D}_X(X)} \D_m(X)\widetilde{\otimes}_{\D_m(X)} \frac{\D'_{m, \infty}(Y)}{I\D'_{m, \infty}(Y)}.
		\end{align*} 
		In this way, $\iota_+\M\widetilde{\otimes}_{\w{\D}_Y}^{\mathbb{L}} \D'_{m, \infty}\cong \iota_+(\M\widetilde{\otimes}^{\mathbb{L}}_{\w{\D}_X} \D_m)$ for any coadmissible right $\w{\D}_X$-module $\M$.
		
		Since coadmissible modules are acyclic for all of the operations involved, this also proves the corresponding isomorphism for $\C$-complexes: explicitly, if $\M^\bullet\in \mathrm{D}_{\C}(\w{\D}_X^{\mathrm{op}})$, then \cite[Lemma 9.3]{SixOp} and \cite[Proposition 5.37]{SixOp} imply that
		\begin{equation*}
			\mathrm{H}^i(\iota_+\M^\bullet\widetilde{\otimes}_{\w{\D}_X}^{\mathbb{L}} \D'_{m, \infty})\cong \iota_+(\mathrm{H}^i(\M^\bullet))\widetilde{\otimes}_{\w{\D}_Y}\D'_{m, \infty},
		\end{equation*}
		and moreover
		\begin{align*}
			\mathrm{H}^i(\iota_+(\M^\bullet\widetilde{\otimes}_{\w{\D}_X}^{\mathbb{L}}\D_m))&\cong \mathrm{H}^i(\mathrm{R}\iota_*(\M^\bullet\widetilde{\otimes}_{\w{\D}_X}^{\mathbb{L}}\D_m\widetilde{\otimes}_K K\{\mathrm{d}\}))\\
			&\cong \mathrm{H}^i(\mathrm{R}\iota_*(\M^\bullet\widetilde{\otimes}_{\w{\D}_X}^{\mathbb{L}}\D_m))\widetilde{\otimes}_K K\{\mathrm{d}\}\\
			&\cong \iota_*(\mathrm{H}^i(\M^\bullet\widetilde{\otimes}^{\mathbb{L}}_{\w{\D}_X} \D_m))\widetilde{\otimes}_K K\{\mathrm{d}\}\\
			&\cong \iota_+(\mathrm{H}^i(\M^\bullet)\widetilde{\otimes}^{(\mathbb{L})}_{\w{\D}_X} \D_m)
		\end{align*}
		by \cite[Corollary 5.36]{SixOp}, the acyclicity of coherent $\D_m$-modules, and \cite[proof of Lemma 7.9]{SixOp}.
		
		The corresponding isomorphism over $\D'_{m,n}$ follows by applying $-\widetilde{\otimes}_{\D'_{m, \infty}}^{\mathbb{L}}\D'_{m,n}$.
		\item For any coadmissible right $\w{\D}_X(X)$-module $M$, there is a natural isomorphism
		\begin{align*}
			M\widetilde{\otimes}_{\w{\D}_X(X)}^{(\mathbb{L})} \frac{\w{\D}_Y(Y)}{I\w{\D}_Y(Y)}\widetilde{\otimes}_{\w{\D}_Y(Y)}^{(\mathbb{L})} \D'_m(Y)&\cong M\widetilde{\otimes}_{\w{\D}_X(X)}^{\mathbb{L}} \frac{\D'_m(Y)}{I\D'_m(Y)} \\
			&\cong M\widetilde{\otimes}_{\w{\D}_X(X)}^{(\mathbb{L})} \D_m(X)\widetilde{\otimes}_{\D_m(X)}^{(\mathbb{L})} \frac{\D'_m(Y)}{I\D'_m(Y)}
		\end{align*}
		of finitely generated right $\D'_m(Y)$-modules, equipped with their canonical Banach structures. The rest of the proof is analogous to (i).
	\end{enumerate}
\end{proof}

\begin{cor}
	\label{holimKashiwara}
	If $\M^\bullet\in \mathrm{D}^-_{\C}(\w{\D}_X^{\mathrm{op}})$ and $\N^\bullet\in \mathrm{D}^-_{\C}(\w{\D}_Y)$, then
	\begin{equation*}
		\iota_+\M^\bullet\widetilde{\otimes}_{\w{\D}_Y}^{\mathbb{L}}\N^\bullet\cong \mathrm{holim} (\iota_+(\M_m^\bullet)\widetilde{\otimes}_{\w{\D}_Y}^{\mathbb{L}}\N^\bullet).
	\end{equation*}
	The same applies if $\M^\bullet\in \mathrm{D}_{\C}^b(\w{\D}_X^{\mathrm{op}})$ and $\N^\bullet\in \mathrm{D}_{\C}(\w{\D}_Y)$.
\end{cor}
\begin{proof}
	We have
	\begin{align*}
		\iota_+(\M^\bullet)\widetilde{\otimes}_{\w{\D}_Y}^{\mathbb{L}}\N^\bullet&\cong \mathrm{holim}_m (\iota_+(\M^\bullet)_m\widetilde{\otimes}_{\D'_m}^{\mathbb{L}} \N_m^\bullet)\\
		&\cong \mathrm{holim}_{m,n} (\iota_+(\M^\bullet)_{m,n}\widetilde{\otimes}_{\D'_{m,n}}^{\mathbb{L}}\N_{m,n}^\bullet)
	\end{align*}
	by Theorem \ref{tensorlimit} and by considering the diagonal copy of $\mathbb{N}$ inside $\mathbb{N}^2$ as a cofinal set, noting that $\D'_{m,m}=\D'_m$.
	
	Now invoke Lemma \ref{njbasechange} and Lemma \ref{tensorlimiota} to deduce that
	\begin{equation*}
		\iota_+(\M^\bullet)\widetilde{\otimes}_{\w{\D}_Y}^{\mathbb{L}}\N^\bullet\cong \mathrm{holim}_m (\iota_+(\M_m^\bullet)\widetilde{\otimes}_{\w{\D}_Y}^{\mathbb{L}}\N^\bullet),
	\end{equation*}
	as required.
\end{proof}

\begin{prop}
	\label{unboundedprojectionclosed}
	Let $\iota: X\to Y$ be as above. If $\M^\bullet\in \mathrm{D}_{\C}(\w{\D}_X^{\mathrm{op}})$ and $\N^\bullet\in \mathrm{D}(\w{\D}_Y)$, then there is a natural isomorphism
	\begin{equation*}
		\iota_+\M_n^\bullet\widetilde{\otimes}_{\w{\D}_Y}^{\mathbb{L}}\N^\bullet\cong \mathrm{R}\iota_*(\M^\bullet_n\widetilde{\otimes}_{\w{\D}_X}^{\mathbb{L}}\mathbb{L}\iota^*\N^\bullet)
	\end{equation*}
	in $\mathrm{D}(\mathrm{Shv}(Y_m, LH(\h{\B}c_K)))$ for any $n\geq m$.
\end{prop}
\begin{proof}
	As $\D_n(X)$ is Auslander regular, we can reduce to the case $\M_n^\bullet=\D_n$, and we need to show that
	\begin{equation*}
		\iota_+(\D_n)\widetilde{\otimes}_{\w{\D}_Y}^{\mathbb{L}}\N^\bullet\cong \mathrm{R}\iota_*(\D_n\widetilde{\otimes}_{\w{\D}_X}^{\mathbb{L}}\mathbb{L}\iota^*\N^\bullet).
	\end{equation*}
	Since both sides are supported on $X$, it suffices to prove the isomorphism after applying $\iota^{-1}$. But since
	\begin{equation*}
		\D_n\widetilde{\otimes}_{\w{\D}_X}^{\mathbb{L}}\iota^*\w{\D}_Y \widetilde{\otimes}_{\iota^{-1}\w{\D}_Y}^{\mathbb{L}}\iota^{-1}\N^\bullet\cong \D_n\widetilde{\otimes}^{\mathbb{L}}_{\w{\D}_X}\mathbb{L}\iota^*\N^\bullet,
	\end{equation*}
	this finishes the proof.
\end{proof}

\begin{prop}
	\label{closedprojection}
	Let $\iota: X\to Y$ be as above. If $\M^\bullet\in \mathrm{D}^-_\C(\w{\D}_X^\mathrm{op})$ and $\N^\bullet\in \mathrm{D}^-_\C(\w{\D}_Y)$ such that $\iota^!\N^\bullet\in \mathrm{D}^-_\C(\w{\D}_X)$, then there is a natural isomorphism
	\begin{equation*}
		\iota_+\M^\bullet\widetilde{\otimes}^\mathbb{L}_{\w{\D}_Y} \N^\bullet \cong \mathrm{R}\iota_*(\M^\bullet\widetilde{\otimes}^\mathbb{L}_{\w{\D}_X} \iota^!\N^\bullet)[l]
	\end{equation*}
	in $\mathrm{D}(\mathrm{Shv}(Y, LH(\h{\B}c_K)))$.
	
	The same applies if $\M^\bullet\in \mathrm{D}_{\C}^b(\w{\D}_X^{\mathrm{op}})$ and $\N^\bullet\in \mathrm{D}_{\C}(\w{\D}_Y)$.
\end{prop}
\begin{proof}
	It suffices to show that the natural morphism
	\begin{equation*}
		\iota_+\M^\bullet \widetilde{\otimes}^\mathbb{L}_{\w{\D}_Y} \N^\bullet \to \mathrm{R}\iota_*(\M^\bullet \widetilde{\otimes}^\mathbb{L}_{\w{\D}_X} \iota^! \N^\bullet)[l]\cong \mathrm{R}\iota_*(\M^\bullet\widetilde{\otimes}^\mathbb{L}_{\w{\D}_X} \mathbb{L}\iota^*\N^\bullet)
	\end{equation*}
	is an isomorphism on $Y_m$ for each $m$. 
	
	By Corollary \ref{holimKashiwara}, $\iota_+\M^\bullet\widetilde{\otimes}_{\w{\D}_Y}^{\mathbb{L}}\N^\bullet$ is a homotopy limit of 
	\begin{equation*}
		\iota_+(\M_n^\bullet)\widetilde{\otimes}_{\w{\D}_Y}^{\mathbb{L}}\N^\bullet. 
	\end{equation*}

	Similarly, 
	\begin{equation*}
		\mathrm{R}\iota_*(\M^\bullet \widetilde{\otimes}^\mathbb{L}_{\w{\D}_X} \iota^!\N^\bullet)[l]
	\end{equation*}
	is a homotopy limit of $\mathrm{R}\iota_*(\M^\bullet_n\widetilde{\otimes}^\mathbb{L}_{\w{\D}_X} \mathbb{L}\iota^*\N^\bullet)$ by Theorem \ref{tensorlimit}, where we use the assumption that $\iota^!\N^\bullet$ is a $\C$-complex and the fact that $\mathrm{R}\iota_*$ commutes with products. It thus suffices to show that on $Y_m$, the natural morphism
	\begin{equation*}
		\mathrm{R}\iota_*(\M_n^{\bullet}\widetilde{\otimes}_{\w{\D}_X}^{\mathbb{L}}\iota^*\w{\D}_Y)\widetilde{\otimes}_{\w{\D}_Y}^{\mathbb{L}}\N^{\bullet}\to \mathrm{R}\iota_*(\M_n^{\bullet}\widetilde{\otimes}_{\w{\D}_X}^{\mathbb{L}}\mathbb{L}\iota^*\N^{\bullet})
	\end{equation*}
	is an isomorphism for each $n\geq m$. 
	
	We can thus apply Proposition \ref{unboundedprojectionclosed} to conclude.
\end{proof}

\begin{cor}
	\label{holimofpullback}
	Let $\N^\bullet\in \mathrm{D}_{\C}(\w{\D}_Y)$ such that $\iota^!\N^\bullet\in \mathrm{D}_{\C}(\w{\D}_X)$. Then $\iota^!\N^\bullet\cong \mathrm{holim} \mathbb{L}\iota^*(\N_n^\bullet)[-l]$.
\end{cor}
\begin{proof}
	As $\mathrm{R}\iota_*$ reflects isomorphisms and commutes with products, it suffices to show that
	\begin{equation*}
		\mathrm{R}\iota_*(\iota^!\N^\bullet)\cong \mathrm{holim}\mathrm{R}\iota_*(\mathbb{L}\iota^*\N_n^\bullet)[-l].
	\end{equation*}
	But now the left hand side is naturally isomorphic to $\iota_+\w{\D}_X\widetilde{\otimes}_{\w{\D}_Y}^{\mathbb{L}}\N^\bullet[-l]$ by Proposition \ref{closedprojection}, while
	\begin{equation*}
		\mathrm{R}\iota_*(\mathbb{L}\iota^*\N_n^\bullet)\cong \iota_+\w{\D}_X\widetilde{\otimes}_{\w{\D}_Y}^{\mathbb{L}}\N_n^\bullet
	\end{equation*}
	by Lemma \ref{njbasechange}.(ii). Invoking Theorem \ref{tensorlimit} proves the result.
\end{proof}

This provides us with a projection formula for closed embeddings.
\begin{cor}
	Let $\iota: X\to Y$ be a closed embedding of smooth rigid analytic $K$-varieties. If $\M^\bullet\in \mathrm{D}^-_\C(\w{\D}_X^\mathrm{op})$ and $\N^\bullet\in \mathrm{D}^-_\C(\w{\D}_Y)$ such that $\iota^!\N^\bullet\in \mathrm{D}^-_\C(\w{\D}_X)$, then there is a natural isomorphism
	\begin{equation*}
		\iota_+\M^\bullet\widetilde{\otimes}^\mathbb{L}_{\w{\D}_Y} \N^\bullet \cong \mathrm{R}\iota_*(\M^\bullet\widetilde{\otimes}^\mathbb{L}_{\w{\D}_X} \iota^!\N^\bullet)[\mathrm{dim}Y-\mathrm{dim}X]
	\end{equation*}
	in $\mathrm{D}(\mathrm{Shv}(Y, LH(\h{\B}c_K)))$.
	
	The same applies if $\M^\bullet\in \mathrm{D}_{\C}^b(\w{\D}_X^{\mathrm{op}})$ and $\N^\bullet\in \mathrm{D}_{\C}(\w{\D}_Y)$.
\end{cor}

\begin{proof}
	We always have the natural morphism
	\begin{equation*}
		\iota_+\M^\bullet \widetilde{\otimes}^\mathbb{L}_{\w{\D}_Y} \N^\bullet \to \mathrm{R}\iota_*(\M^\bullet \widetilde{\otimes}^\mathbb{L}_{\w{\D}_X} \iota^! \N^\bullet)[\mathrm{dim}Y-\mathrm{dim}X],
	\end{equation*}
	so we argue locally on $Y$.
	
	Using Theorem \ref{MainThm} and the existence of tubular neighbourhoods (\cite[Theorem 1.18]{Kiehl}), we can reduce to the setting of Proposition \ref{closedprojection}.
\end{proof}

\subsection{A projection formula for projection maps}

We now also consider the case of a projection map.

Let $X, Y$ be smooth rigid analytic $K$-varieties. Suppose that $X$ is separated with countably many connected components, and $Y=\Sp A$ is an affinoid with free tangent sheaf. Let $\A$ be an admissible affine formal model for $A$ and $\L\subseteq \T_Y(Y)$ an $(R, \A)$-Lie lattice, and consider the corresponding sheaf of algebras $\D_n$ on the site $Y_n$ of $\pi^n\L$-accessible subdomains of $Y$.

We consider the projection $f: Z=X\times Y\to Y$. Let $Z_n=X\times Y_n$, where we consider $X$ with the weak Grothendieck topology. In particular, if $\N\in \mathrm{Mod}_{\mathrm{Shv}(Y_n, LH(\h{\B}c_K))}(\O_Y)$, we can consider its pullback $\mathbb{L}f^*\N=\O_Z\widetilde{\otimes}^{\mathbb{L}}_{f^{-1}\O_Y}f^{-1}\N$ to $Z_n$.

Note that $f^*\D_n$ is naturally a $(\w{\D}_Z, f^{-1}\D_n)$-bimodule on $Z_n$, which on any affinoid of the form $U\times V\subseteq X\times Y$ for $U$ resp. $V$ an affinoid subdomain of $X$ resp. $Y_n$, is given by
\begin{equation*}
	f^*\D_n(U\times V)\cong \O_X(U)\widetilde{\otimes}_K \D_n(V). 
\end{equation*}

\begin{lem}
	\label{nbasechangelem}
	Let $\M^\bullet\in \mathrm{D}^b_\C(\w{\D}_Z^\mathrm{op})$. Then the natural morphism
	\begin{equation*}
		\mathrm{R}\Gamma(X\times Y, \M^\bullet\widetilde{\otimes}^\mathbb{L}_{\w{\D}_Z}f^*\w{\D}_Y)\widetilde{\otimes}^\mathbb{L}_{\w{\D}_Y(Y)}\D_n(Y)\to \mathrm{R}\Gamma(X\times Y, \M^\bullet\widetilde{\otimes}^\mathbb{L}_{\w{\D}_Z}f^*\D_n)
	\end{equation*}
	is an isomorphism in $\mathrm{D}(\mathrm{Mod}_{LH(\h{\B}c_K)}(\D_n(Y)^\mathrm{op}))$. 
\end{lem}
\begin{proof}
	See Appendix, Lemma \ref{Cechbasechangeapp}.
\end{proof}

\begin{lem}
	\label{smoothprojforDn}
	Let $\M^\bullet\in \mathrm{D}^b_\C(\w{\D}_Z^\mathrm{op})$ and assume that $f_+\M^\bullet\in \mathrm{D}_\C(\w{\D}_Y^\mathrm{op})$. Then the natural map
	\begin{equation*}
		f_+\M^\bullet\widetilde{\otimes}^\mathbb{L}_{\w{\D}_Y} \D_n\to \mathrm{R}f_*(\M^\bullet\widetilde{\otimes}^\mathbb{L}_{\w{\D}_Z} f^*\D_n)
	\end{equation*}
	is an isomorphism in $\mathrm{D}(\mathrm{Mod}_{\mathrm{Shv}(Y_n, LH(\h{\B}c_K))}(\D_n^\mathrm{op}))$.
	
	The same holds for any $\M^\bullet \in \mathrm{D}_{\C}(\w{\D}_Z^{\mathrm{op}})$ with $f_+\M^\bullet\in\mathrm{D}_{\C}(\w{\D}_Y^{\mathrm{op}})$, provided that $f_*$ has finite cohomological dimension.
\end{lem}
\begin{proof}
	By assumption, $\mathrm{H}^j(f_+\M^\bullet)$ is a coadmissible right $\w{\D}_Y$-module, and 
	\begin{equation*}
		\mathrm{H}^j(f_+\M^\bullet\widetilde{\otimes}^\mathbb{L}_{\w{\D}_Y}\D_n)(Y)\cong \mathrm{H}^j(f_+\M^\bullet)(Y)\widetilde{\otimes}^{(\mathbb{L})}_{\w{\D}_Y(Y)}\D_n(Y)
	\end{equation*}
	by \cite[Proposition 8.2, proof of Lemma 6.6]{SixOp}.
	
	By the Lemma above, the latter is naturally isomorphic to
	\begin{equation*}
		\mathrm{H}^j(X\times Y, \M^\bullet\widetilde{\otimes}_{\w{\D}_Z}^{\mathbb{L}}f^*\D_n),
	\end{equation*}
	which by the acyclicity of coherent modules is
	\begin{equation*}
		\mathrm{H}^j(\mathrm{R}f_*(\M^\bullet\widetilde{\otimes}_{\w{\D}_Z}^{\mathbb{L}}f^*\D_n))(Y).
	\end{equation*}
	
	As Lemma \ref{nbasechangelem} applies in the same way to any $\pi^n\L$-accessible subdomain of $Y$, we thus have
	\begin{align*}
		\mathrm{H}^j(f_+\M^\bullet\widetilde{\otimes}_{\w{\D}_Y}^{\mathbb{L}}\D_n)&\cong \mathrm{H}^j(f_+\M^\bullet)\widetilde{\otimes}^{(\mathbb{L})}_{\w{\D}_Y} \D_n\\
		&\cong\mathrm{H}^j(\mathrm{R}f_*(\M^\bullet\widetilde{\otimes}_{\w{\D}_Z}^{\mathbb{L}}f^*\D_n)),
	\end{align*}
	so the natural morphism
	\begin{equation*}
		f_+\M^\bullet\widetilde{\otimes}^\mathbb{L}_{\w{\D}_Y}\D_n\to \mathrm{R}f_*(\M^\bullet\widetilde{\otimes}^\mathbb{L}_{\w{\D}_Z}f^*\D_n)
	\end{equation*}
	is a quasi-isomorphism.
	
	If $\M^\bullet$ is unbounded and $f_*$ has finite cohomological dimension, we note that in subsection A.2, both $f^*\w{\D}_Y$ and $f^*\D_n$ admit a finite resolution with which we can calculate the derived tensor product, so for each $j$ there exists $a, b$ such that
	\begin{equation*}
		\mathrm{H}^j(f_+\M^\bullet)\cong \mathrm{H}^j(f_+\tau^{\geq a}\tau^{\leq b}\M^\bullet)\in \C_Y
	\end{equation*}
	and
	\begin{equation*}
		\mathrm{H}^j(\mathrm{R}f_*(\M^\bullet\widetilde{\otimes}_{\w{\D}_Z}^{\mathbb{L}}f^*\D_n))\cong \mathrm{H}^j(\mathrm{R}f_*(\tau^{\geq a}\tau^{\leq b}\M^\bullet\widetilde{\otimes}_{\w{\D}_Z}^{\mathbb{L}}f^*\D_n))
	\end{equation*}
	thanks to \cite[tag 07K7]{stacksproj}. We can thus reduce to the bounded case.
\end{proof}

\begin{cor}
	\label{smoothprofforNn}
	Assume that $D_n(Y)$ is Auslander regular. Let $\M^\bullet\in \mathrm{D}^b_{\C}(\w{\D}_Z^{\mathrm{op}})$ such that $f_+\M^\bullet\in \mathrm{D}_{\C}(\w{\D}_Y^{\mathrm{op}})$ and let $\N^\bullet\in \mathrm{D}_{\C}(\w{\D}_Y)$. Then the natural map
	\begin{equation*}
		f_+\M^\bullet\widetilde{\otimes}_{\w{\D}_Y}^{\mathbb{L}}\N_n^\bullet\to \mathrm{R}f_*(\M^\bullet\widetilde{\otimes}_{\w{\D}_Z}^{\mathbb{L}}\mathbb{L}f^*\N_n^\bullet)
	\end{equation*}
	is an isomorphism in $\mathrm{D}(\mathrm{Shv}(Y_m, LH(\h{\B}c_K)))$ for any $n\geq m$.
	
	The same holds for any $\M^\bullet\in \mathrm{D}_{\C}(\w{\D}_Z^{\mathrm{op}})$ with $f_+\M^\bullet\in \mathrm{D}_{\C}(\w{\D}_Y^{\mathrm{op}})$, provided that $f_*$ has finite cohomological dimension.
\end{cor}
\begin{proof}
	We can reduce to $\N_n^\bullet=\D_n$ and apply the Lemma above.
\end{proof}

We also point out that if $\N^\bullet\in \mathrm{D}_{C}(\w{\D}_Y)$ as above, then 
\begin{equation*}
	\mathbb{L}f^*\N^\bullet\cong \mathrm{holim} \mathbb{L}f^*\N_n^\bullet,
\end{equation*}
by \cite[Corollary 5.21]{SixOp} and the same argument as in Lemma \ref{tensorasholimapp}.

Letting $U\subseteq X$ be an affinoid subspace with free tangent sheaf, we can form a suitable sheaf $\D_{U, n}$ on $U$, and $\D_{U\times Y, n}$ on $U\times Y$, with global sections
\begin{equation*}
	\D_{U\times Y, n}(Y)=\D_{U, n}(U)\widetilde{\otimes}_K \D_n(Y).
\end{equation*}
If $\N_n$ is a coherent $\D_n$-module, then $f^*\N_n|_{U\times Y}$ is a coherent $\D_{U\times Y, n}$-module, associated to the sections
\begin{equation*}
	f^*\N_n(U\times Y)\cong \O_X(U)\widetilde{\otimes}_K \N_n(Y).
\end{equation*}

In particular, for any $\C$-complex $\N^\bullet$ on $Y$, the description above shows that
\begin{equation*}
	\D_{U\times Y, n}\widetilde{\otimes}_{\w{\D}_Z}^{\mathbb{L}} f^!\N^\bullet|_{U\times Y}\cong \mathbb{L}f^*\N_n^\bullet|_{U\times Y}[\mathrm{dim}X].
\end{equation*}

Applying Theorem \ref{tensorlimit} locally, we can deduce that if $\N^\bullet\in \mathrm{D}_{\C}^-(\w{\D}_Y)$ and $\M^\bullet\in \mathrm{D}^-_{\C}(\w{\D}_Z^{\mathrm{op}})$, then
\begin{equation*}
	\M^\bullet\widetilde{\otimes}_{\w{\D}_Z}^{\mathbb{L}}f^!\N^\bullet\cong \mathrm{holim} \M^\bullet\widetilde{\otimes}_{\w{\D}_Z}^{\mathbb{L}} \mathbb{L}f^*(\N_n^\bullet)[\mathrm{dim}X].
\end{equation*}

\begin{lem}
Let $X$, $Y$ be smooth rigid analytic $K$-varieties. Assume that $X$ is separated with countably many connected components, and consider the natural projection $f:Z=X\times Y\to Y$. If $\M^\bullet\in \mathrm{D}^b_\C(\w{\D}_Z^\mathrm{op})$ such that $f_+\M^\bullet\in \mathrm{D}^b_\C(\w{\D}_Y^\mathrm{op})$, and $\N^\bullet\in \mathrm{D}^-_\C(\w{\D}_Y)$, then there is a natural isomorphism
\begin{equation*}
f_+\M^\bullet\widetilde{\otimes}^\mathbb{L}_{\w{\D}_Y} \N^\bullet \cong \mathrm{R}f_*(\M^\bullet \widetilde{\otimes}^\mathbb{L}_{\w{\D}_Z} f^!\N^\bullet)[-\mathrm{dim}X]
\end{equation*} 
in $\mathrm{D}(\mathrm{Shv}(Y, LH(\h{\B}c_K)))$.
\end{lem}

\begin{proof}
To show that the natural morphism
\begin{equation*}
f_+\M^\bullet\widetilde{\otimes}^\mathbb{L}_{\w{\D}_Y} \N^\bullet \to \mathrm{R}f_*(\M^\bullet \widetilde{\otimes}^\mathbb{L}_{\w{\D}_Z} f^!\N^\bullet)[-\mathrm{dim}X]=\mathrm{R}f_*(\M^\bullet\widetilde{\otimes}^\mathbb{L}_{\w{\D}_Z} \mathbb{L}f^*\N^\bullet)
\end{equation*}
is an isomorphism, we can argue locally on $Y$ and suppose that $Y$ is a D-regular affinoid. Let $Y=\Sp A$ and choose an admissible affine formal model $\A$, a regular $(R, \A)$-Lie lattice $\L$, and the corresponding sheaf $\D_n$ on $Y_n$.

By Theorem \ref{tensorlimit}, for each $m$, $f_+\M^\bullet\widetilde{\otimes}^\mathbb{L}_{\w{\D}_Y} \N^\bullet$ is the homotopy limit of $f_+\M^\bullet \widetilde{\otimes}^\mathbb{L}_{\w{\D}_Y} \N^\bullet_n$ for $n\geq m$ on $Y_m$. As $f$ is smooth, $f^!\N^\bullet\in \mathrm{D}^-_{\C}(\w{\D}_Z)$ by Proposition \ref{Ccomplexstable}, and by the above, $\mathrm{R}f_*(\M^\bullet \widetilde{\otimes}^\mathbb{L}_{\w{\D}_Z} f^!\N^\bullet)$ is the homotopy limit of $\mathrm{R}f_*(\M^\bullet \widetilde{\otimes}^\mathbb{L}_{\w{\D}_Z} f^!\N^\bullet_n)$. It thus suffices to show that
\begin{equation*}
f_+\M^\bullet \widetilde{\otimes}^\mathbb{L}_{\w{\D}_Y} \N^\bullet_n\cong \mathrm{R}f_*(\M^\bullet\widetilde{\otimes}^\mathbb{L}_{\w{\D}_Z}f^!\N^\bullet_n)[-\mathrm{dim}X]
\end{equation*}
via the natural morphism. But this is precisely Corollary \ref{smoothprojforDn}.
\end{proof}

Combining with our earlier results, we obtain a projection formula for separated morphisms.

\begin{thm}
\label{projectionformula}
Let $f: X\to Y$ be a separated morphism of smooth rigid analytic $K$-varieties. Assume that $Y$ has an admissible covering by D-regular affinoid subdomains $(Y_i)$ such that for each $i$, $f^{-1}Y_i$ has countably many connected components. Let $\M^\bullet \in \mathrm{D}^b_\C(\w{\D}_X^\mathrm{op})$ such that $f_+\M^\bullet\in \mathrm{D}^b_\C(\w{\D}_Y^\mathrm{op})$, and let $\N^\bullet\in \mathrm{D}^-_\C(\w{\D}_Y)$ such that $f^!\N^\bullet\in \mathrm{D}^-_\C(\w{\D}_X)$. Then there is a natural isomorphism
\begin{equation*}
f_+\M^\bullet\widetilde{\otimes}^\mathbb{L}_{\w{\D}_Y} \N^\bullet \cong \mathrm{R}f_*(\M^\bullet\widetilde{\otimes}^\mathbb{L}_{\w{\D}_X} f^!\N^\bullet)[\mathrm{dim}Y-\mathrm{dim}X]
\end{equation*}
in $\mathrm{D}(\mathrm{Shv}(Y, LH(\h{\B}c_K)))$.
\end{thm}

\begin{proof}
There is a natural morphism 
\begin{equation*}
f_+\M^\bullet\widetilde{\otimes}^\mathbb{L}_{\w{\D}_Y} \N^\bullet \to \mathrm{R}f_*(\M^\bullet\widetilde{\otimes}^\mathbb{L}_{\w{\D}_X} f^!\N^\bullet)[\mathrm{dim}Y-\mathrm{dim}X],
\end{equation*}
so that we can argue locally on $Y$. Assuming that $Y$ is a D-regular affinoid, we can assume in particular that the graph embedding $\iota: X\to X\times Y$ is a closed embedding, and $X=f^{-1}Y$ has countably many connected components. Hence writing $f$ as the composition of $\iota$ and the projection $\mathrm{pr}: X\times Y\to Y$ reduces to the cases discussed earlier, since $\iota_+$ and $\mathrm{pr}^!$ preserve $\C$-complexes (\cite[Theorem 1.3]{SixOp}), and $f_+\cong \mathrm{pr}_+\iota_+$, $f^!\cong \iota^!\mathrm{pr}^!$ by \cite[Lemma 7.8, Proposition 7.5]{SixOp}.
\end{proof}

\subsection{Adjunction formulas}

\begin{prop}
\label{homwithdual}
Let $X$ be a smooth rigid analytic $K$-variety. Let $\M^\bullet\in \mathrm{D}_\C^+(\w{\D}_X)$ and $\N^\bullet\in \mathrm{D}_\C^-(\w{\D}_X)$. Then there is a natural isomorphism
\begin{equation*}
\mathrm{R}\mathcal{H}om_{\w{\D}_X}(\M^\bullet, \N^\bullet)\cong \mathrm{R}\mathcal{H}om_{\w{\D}_X}(\M^\bullet, \w{\D}_X)\widetilde{\otimes}^\mathbb{L}_{\w{\D}_X}\N^\bullet.
\end{equation*}
\end{prop}
\begin{proof}
To prove that the natural morphism 
\begin{equation*}
\mathrm{R}\mathcal{H}om_{\w{\D}_X}(\M^\bullet, \w{\D}_X)\widetilde{\otimes}^\mathbb{L}_{\w{\D}_X}\N^\bullet\to \mathrm{R}\mathcal{H}om_{\w{\D}_X}(\M^\bullet, \N^\bullet)
\end{equation*}
is an isomorphism, we can argue locally and reduce to the case when $X$ is a D-regular affinoid. Let $\A$, $\L$, $\D_n$ be defined as before, with $\D_n(X)$ Auslander regular for all $n$. Note that by Proposition \ref{dualsofCcomplexes} and \cite[Lemma 8.4]{ST}, $\mathrm{R}\mathcal{H}om_{\w{\D}_X}(\M^\bullet, \w{\D}_X)$ is in $\mathrm{D}^{-}_\C(\w{\D}_X^\mathrm{op})$, and
\begin{equation*}
\mathrm{R}\mathcal{H}om_{\w{\D}_X}(\M^\bullet, \w{\D}_X)\widetilde{\otimes}^\mathbb{L}_{\w{\D}_X}\D_n\cong \mathrm{R}\mathcal{H}om_{\D_n}(\M_n^\bullet, \D_n).
\end{equation*}
By Theorem \ref{tensorlimit}, we have a distinguished triangle
\begin{equation*}
\mathrm{R}\mathcal{H}om_{\w{\D}_X}(\M^\bullet, \w{\D}_X)\widetilde{\otimes}^\mathbb{L}_{\w{\D}_X}\N^\bullet \to \prod_{n\geq m} \mathrm{R}\mathcal{H}om_{\D_n}(\M_n^\bullet, \D_n)\widetilde{\otimes}^\mathbb{L}_{\D_n}\N_n^\bullet\to \prod_{n\geq m} \mathrm{R}\mathcal{H}om_{\D_n}(\M_n^\bullet, \D_n)\widetilde{\otimes}^\mathbb{L}_{\D_n}\N_n^\bullet
\end{equation*}
on $X_m$ for any $m\geq 0$.

Now since $\D_n(X)$ is Auslander regular, the natural morphism
\begin{equation*}
\mathrm{R}\mathcal{H}om_{\D_n}(\M_n^\bullet, \D_n)\widetilde{\otimes}^\mathbb{L}_{\D_n} \N^\bullet_n\to \mathrm{R}\mathcal{H}om_{\D_n}(\M^\bullet_n, \N^\bullet_n)
\end{equation*} 
is an isomorphism, as we can reduce to the case $\M_n^\bullet=\D_n$.

As 
\begin{equation*}
\mathrm{R}\mathcal{H}om_{\D_n}(\M_n^\bullet, \N_n^\bullet)\cong \mathrm{R}\mathcal{H}om_{\w{\D}_X}(\M^\bullet, \N_n^\bullet)
\end{equation*}
by tensor-hom adjunction, this shows that the natural morphisms induce an isomorphism of distinguished triangles
\begin{equation*}
\begin{xy}
\xymatrix{
\mathrm{R}\mathcal{H}om_{\w{\D}_X}(\M^\bullet, \w{\D}_X)\widetilde{\otimes}^\mathbb{L}_{\w{\D}_X}\N^\bullet\ar[r] \ar[d]& \prod \mathrm{R}\mathcal{H}om_{\D_n}(\M_n^\bullet, \D_n)\widetilde{\otimes}^\mathbb{L}_{\D_n} \N_n^\bullet \ar[r] \ar[d] &\prod \mathrm{R}\mathcal{H}om_{\D_n}(\M_n^\bullet, \D_n)\widetilde{\otimes}^\mathbb{L}_{\D_n} \N_n^\bullet\ar[d]\\
\mathrm{R}\mathcal{H}om_{\w{\D}_X}(\M^\bullet, \N^\bullet)\ar[r] & \prod \mathrm{R}\mathcal{H}om_{\D_n}(\M_n^\bullet, \N_n^\bullet)\ar[r] & \prod \mathrm{R}\mathcal{H}om_{\D_n}(\M_n^\bullet, \N_n^\bullet),
}
\end{xy}
\end{equation*}
and in particular,
\begin{equation*}
\mathrm{R}\mathcal{H}om_{\w{\D}_X}(\M^\bullet, \w{\D}_X)\widetilde{\otimes}^\mathbb{L}_{\w{\D}_X}\N^\bullet \cong \mathrm{R}\mathcal{H}om_{\w{\D}_X}(\M^\bullet, \N^\bullet)
\end{equation*}
via the natural morphism.
\end{proof}

\begin{rmk}
	Note that we can drop the boundedness assumption on $\M^\bullet$ if instead $\N^\bullet\in \mathrm{D}_{\C}^b(\w{\D}_X)$, by Corollary \ref{tensorlimitapp}.
	
	Similarly, if $X$ is a D-regular affinoid and $\M^\bullet, \N^\bullet\in \mathrm{D}_{\C}(\w{\D}_X)$, then there is a natural isomorphism
	\begin{align*}
		\mathrm{R}\mathcal{H}om_{\w{\D}_X}(\M^\bullet, \N_n^\bullet)&\cong \mathrm{R}\mathcal{H}om_{\D_n}(\M_n^\bullet, \N_n^\bullet)\\
		&\cong \mathrm{R}\mathcal{H}om_{\D_n}(\M_n^\bullet, \D_n)\widetilde{\otimes}_{\D_n}^{\mathbb{L}}\N_n^\bullet\\
		&\cong \mathrm{R}\mathcal{H}om_{\w{\D}_X}(\M^\bullet, \w{\D}_X)\widetilde{\otimes}_{\w{\D}_X}^{\mathbb{L}}\N_n^\bullet
	\end{align*}
	in $\mathrm{D}(\mathrm{Shv}(X_m, LH(\h{\B}c_K)))$ for any $n\geq m$.
\end{rmk}

\begin{cor}
	\label{projfortransfer}
	\leavevmode
	\begin{enumerate}[(i)]
		\item Let $\iota: X\to Y$ be a closed embedding of smooth rigid analytic $K$-varieties. If $\M^\bullet\in \mathrm{D}_{\C}(\w{\D}_X)$, then there is a natural isomorphism
		\begin{equation*}
			\mathrm{R}\mathcal{H}om_{\w{\D}_X}(\M^\bullet, \w{\D}_X)\widetilde{\otimes}_{\w{\D}_X}^{\mathbb{L}} \iota^!\w{\D}_Y\cong \mathrm{R}\mathcal{H}om_{\w{\D}_X}(\M^\bullet,\iota^!\w{\D}_Y).
		\end{equation*}
		\item Let $f: X\to Y$ be a smooth morphism of smooth rigid analytic $K$-varieties. If $\M^\bullet\in \mathrm{D}_{\C}(\w{\D}_X)$, then there is a natural isomorphism
		\begin{equation*}
			\mathrm{R}\mathcal{H}om_{\w{\D}_X}(\M^\bullet, \w{\D}_X)\widetilde{\otimes}_{\w{\D}_X}^{\mathbb{L}}f^!\w{\D}_Y\cong \mathrm{R}\mathcal{H}om_{\w{\D}_X}(\M^\bullet, f^!\w{\D}_Y).
		\end{equation*}
	\end{enumerate}
\end{cor}
\begin{proof}
	\begin{enumerate}[(i)]
		\item We can argue locally, reducing to the case where $X$ and $Y$ are D-regular affinoids admitting local coordinates, and $X$ is the vanishing locus of some of the coordinates ($x_{r+1}, \hdots, x_m$, say). But then 
		\begin{equation*}
			\iota^!\w{\D}_Y\cong \w{\D}_X\widetilde{\otimes}_K K\{ \partial_{r+1}, \hdots, \partial_m\}[\mathrm{dim}X-\mathrm{dim}Y].
		\end{equation*}
		 Since
		\begin{equation*}
			\w{\D}_X\widetilde{\otimes}_K K\{\partial_{r+1}, \hdots, \partial_m\}\cong \mathrm{holim}(\D_n\widetilde{\otimes}_K K\{\partial_{r+1}, \hdots, \partial_m\})
		\end{equation*}
		and
		\begin{equation*}
			\mathrm{R}\mathcal{H}om_{\w{\D}_X}(\M^\bullet, \w{\D}_X)\widetilde{\otimes}_K K\{\partial_{r+1}, \hdots, \partial_m\}\cong \mathrm{holim} \mathrm{R}\mathcal{H}om_{\w{\D}_X}(\M^\bullet, \D_n)\widetilde{\otimes}_K K\{\partial_{r+1}, \hdots, \partial_m\}
		\end{equation*}
		by \cite[Corollary 5.21]{SixOp}, the same argument as above yields the desired isomorphism.
		\item Since $f$ is smooth, $f^!\w{\D}_Y$ is a bounded $\C$-complex (in fact, a coadmissible module in degree $\mathrm{dim}Y-\mathrm{dim}X$). By Proposition \ref{homwithdual} and the remark following it, we have thus
		\begin{equation*}
			\mathrm{R}\mathcal{H}om_{\w{\D}_X}(\M^\bullet, \w{\D}_X)\widetilde{\otimes}_{\w{\D}_X}^{\mathbb{L}}f^!\w{\D}_Y\cong \mathrm{R}\mathcal{H}om_{\w{\D}_X}(\M^\bullet, f^!\w{\D}_Y),
		\end{equation*}
		as desired.
	\end{enumerate}
\end{proof}

\begin{cor}
\label{diminthelimit}
Let $X$ be a smooth rigid analytic $K$-variety, and let $\M$, $\N$ be coadmissible $\w{\D}_X$-modules. Then $\mathcal{E}xt^j_{\w{\D}_X}(\M, \N)=0$ for all $j>\mathrm{dim} X$.
\end{cor}
\begin{proof}
By Proposition \ref{homwithdual}, 
\begin{equation*}
\mathrm{R}\mathcal{H}om_{\w{\D}_X}(\M, \N)\cong \mathrm{R}\mathcal{H}om_{\w{\D}_X}(\M, \w{\D}_X)\widetilde{\otimes}^\mathbb{L}_{\w{\D}_X} \N,
\end{equation*}
and $\mathrm{R}\mathcal{H}om(\M, \w{\D}_X)$ can be represented by a complex which is 0 in degrees larger than $\mathrm{dim}X$ by \cite[Theorem A.(ii)]{DcapThree} and Theorem \ref{Bernstein}.
\end{proof}

\begin{prop}
\label{homastensor}
Let $X$ be a smooth rigid analytic $K$-variety, and let $\M^\bullet \in \mathrm{D}^{+}_\C(\w{\D}_X)$, $\N^\bullet\in \mathrm{D}^{-}_\C(\w{\D}_X)$. We have natural isomorphisms 
\begin{align*}
\mathrm{R}\mathcal{H}om_{\w{\D}_X}(\M^\bullet, \N^\bullet)&\cong (\Omega_X \widetilde{\otimes}^\mathbb{L}_{\O_X} \mathbb{D}\M^\bullet)\widetilde{\otimes}^\mathbb{L}_{\w{\D}_X} \N^\bullet[-\mathrm{dim} X]\\
&\cong \Omega_X \widetilde{\otimes}^\mathbb{L}_{\w{\D}_X}(\mathbb{D}\M^\bullet\widetilde{\otimes}^\mathbb{L}_{\O_X} \N^\bullet)[-\mathrm{dim} X].
\end{align*}
\end{prop}

\begin{proof}
The first isomorphism is an immediate consequence of Proposition \ref{homwithdual}. The second isomorphism follows from Lemma \ref{sidechangetensor}.
\end{proof}

\begin{lem}
\label{homsofduals}
Let $X$ be a smooth rigid analytic $K$-variety. If $\M^\bullet, \N^\bullet\in \mathrm{D}_\C(\w{\D}_X)$, then there is a natural isomorphism
\begin{equation*}
\mathrm{R}\mathcal{H}om_{\w{\D}_X}(\M^\bullet, \N^\bullet)\cong \mathrm{R}\mathcal{H}om_{\w{\D}_X}(\mathbb{D}\N^\bullet, \mathbb{D}\M^\bullet).
\end{equation*}
\end{lem}

\begin{proof}
For bounded $\C$-complexes, this follows from Proposition \ref{homastensor}, since $\N^\bullet\cong \mathbb{D}\mathbb{D}\N^\bullet$. In general, tensor-hom adjunction induces a natural morphism
\begin{equation*}
	\M^\bullet\widetilde{\otimes}_K^{\mathbb{L}} \mathrm{R}\mathcal{H}om_{\w{\D}_X}(\M^\bullet, \N^\bullet)\to \N^\bullet,
\end{equation*} 
which in turn induces a morphism
\begin{equation*}
	\mathrm{R}\mathcal{H}om_{\w{\D}_X}(\N^\bullet, \w{\D}_X)\widetilde{\otimes}_{\w{\D}_X}^{\mathbb{L}}\M^\bullet\widetilde{\otimes}_K^{\mathbb{L}} \mathrm{R}\mathcal{H}om_{\w{\D}_X}(\M^\bullet, \N^\bullet)\to \w{\D}_X,
\end{equation*}
yielding by adjunction the natural morphism
\begin{equation*}
	\mathrm{R}\mathcal{H}om_{\w{\D}_X}(\M^\bullet, \N^\bullet)\to \mathrm{R}\mathcal{H}om_{\w{\D}_X}(\mathrm{R}\mathcal{H}om_{\w{\D}_X}(\N^\bullet, \w{\D}_X), \mathrm{R}\mathcal{H}om_{\w{\D}_X}(\M^\bullet, \w{\D}_X)).
\end{equation*}
As the degree shifts in the definition of $\mathbb{D}$ cancel each other out for $\mathbb{D}\M^\bullet$ and $\mathbb{D}\N^\bullet$, side-changing yields the desired natural morphism.

To show that this is an isomorphism, we can argue locally and assume that $X$ is a D-regular affinoid. Now
\begin{align*}
	\mathrm{R}\mathcal{H}om_{\w{\D}_X}(\M^\bullet, \N^\bullet)&\cong \mathrm{holim}\mathrm{R}\mathcal{H}om_{\w{\D}_X}(\M^\bullet, \N_n^\bullet)\\
	&\cong \mathrm{holim}\mathrm{R}\mathcal{H}om_{\D_n}(\M_n^\bullet, \N_n^\bullet),
\end{align*}
analogously for the dual terms.

Since the natural morphism
\begin{equation*}
	\mathrm{R}\mathcal{H}om_{\D_n}(\M_n^\bullet, \N_n^\bullet)\to \mathrm{R}\mathcal{H}om_{\D_n}(\mathrm{R}\mathcal{H}om_{\D_n}(\N_n^\bullet, \D_n), \mathrm{R}\mathcal{H}om_{\D_n}(\M_n^\bullet, \D_n))
\end{equation*}
is an isomorphism for each $n$, as we can reduce to the case $\N_n^\bullet\cong \D_n$,
this proves the result.
\end{proof}

We now prove adjunction statements between our direct and inverse image operations. To produce the desired natural morphisms, we consider separately the case of closed embeddings and projection morphisms, as before.

Let $\iota: X\to Y$ be a closed embedding of smooth rigid analytic $K$-varieties, let $\M^\bullet\in \mathrm{D}_{\C}(\w{\D}_X)$ and $\N^\bullet\in \mathrm{D}_{\C}(\w{\D}_Y)$ such that $\iota^!\N^\bullet\in \mathrm{D}_{\C}(\w{\D}_X)$.

By Lemma \ref{homsofduals} and a standard argument using tensor-hom adjunction and adjunction for inverse and direct image, there are natural morphisms
\begin{align*}
	\mathrm{R}\iota_*\mathrm{R}\mathcal{H}om_{\w{\D}_X}(\M^\bullet, \iota^!\N^\bullet)\cong \mathrm{R}\iota_*\mathrm{R}\mathcal{H}om_{\w{\D}_X}(\mathbb{D}\iota^!\N^\bullet, \mathbb{D}\M^\bullet)\\
	\to \mathrm{R}\iota_*\mathrm{R}\mathcal{H}om_{\iota^{-1}\w{\D}_Y}(\mathrm{R}\mathcal{H}om_{\w{\D}_X}(\iota^!\N^\bullet, \w{\D}_X)\widetilde{\otimes}_{\w{\D}_X}\iota^*\w{\D}_Y, \mathrm{R}\mathcal{H}om_{\w{\D}_X}(\M^\bullet, \w{\D}_X)\widetilde{\otimes}_{\w{\D}_X}\iota^*\w{\D}_Y)\\
	\to \mathrm{R}\iota_*\mathrm{R}\mathcal{H}om_{\iota^{-1}\w{\D}_Y}(\iota^{-1}\mathrm{R}\iota_*(\mathrm{R}\mathcal{H}om_{\w{\D}_X}(\iota^!\N^\bullet, \w{\D}_X)\widetilde{\otimes}_{\w{\D}_X}^{\mathbb{L}}\iota^*\w{\D}_Y), \mathrm{R}\mathcal{H}om_{\w{\D}_X}(\M^\bullet, \w{\D}_X)\widetilde{\otimes}^{\mathbb{L}}_{\w{\D}_X}\iota^*\w{\D}_Y)\\
	\cong \mathrm{R}\iota_*\mathrm{R}\mathcal{H}om_{\iota^{-1}\w{\D}_Y}(\mathrm{R}\iota_*(\mathrm{R}\mathcal{H}om_{\w{\D}_X}(\iota^!\N^\bullet, \w{\D}_X)\widetilde{\otimes}_{\w{\D}_X}^{\mathbb{L}}\iota^*\w{\D}_Y), \mathrm{R}\iota_*(\mathrm{R}\mathcal{H}om_{\w{\D}_X}(\M^\bullet, \w{\D}_X)\widetilde{\otimes}_{\w{\D}_X}^{\mathbb{L}}\iota^*\w{\D}_Y))\\
	\cong \mathrm{R}\mathcal{H}om_{\w{\D}_Y}(\iota_+(\mathbb{D}\iota^!\N^\bullet), \iota_+(\mathbb{D}\M^\bullet))\\
	\cong \mathrm{R}\mathcal{H}om_{\w{\D}_Y}(\iota_!\M^\bullet, \iota_!\iota^!\N^\bullet).
\end{align*}

We now produce a natural morphism $\iota_!\iota^!\N^\bullet\to \N^\bullet$, yielding thus a natural morphism
\begin{equation*}
	\mathrm{R}\iota_*\mathrm{R}\mathcal{H}om_{\w{\D}_X}(\M^\bullet, \iota^!\N^\bullet)\to \mathrm{R}\mathcal{H}om_{\w{\D}_Y}(\iota_!\M^\bullet, \N^\bullet).
\end{equation*}

The adjunctions yield a natural morphism
\begin{equation*}
	\mathrm{R}\mathcal{H}om_{\w{\D}_Y}(\N^\bullet, \w{\D}_Y)\to \mathrm{R}\iota_*\mathrm{R}\mathcal{H}om_{\w{\D}_X}(\mathbb{L}\iota^*\N^\bullet, \iota^*\w{\D}_Y).
\end{equation*}
Applying Corollary \ref{projfortransfer}.(i), the latter is isomorphic to
\begin{equation*}
	\mathrm{R}\iota_*(\mathrm{R}\mathcal{H}om_{\w{\D}_X}(\iota^!\N^\bullet, \w{\D}_X)\widetilde{\otimes}_{\w{\D}_X}^{\mathbb{L}}\iota^*\w{\D}_Y)[\mathrm{dim}Y-\mathrm{dim}X],
\end{equation*}
so that we have produced a natural morphism
\begin{equation*}
	\mathbb{D}\N^\bullet\to \iota_+\mathbb{D}\iota^!\N^\bullet.
\end{equation*}
Applying duality again produces the desired morphism.

Similarly, suppose that $f: X\to Y$ is a smooth morphism of smooth rigid analytic $K$-varieties. Let $\M^\bullet\in \mathrm{D}_{\C}(\w{\D}_X)$ such that $f_!\M^\bullet\in \mathrm{D}_{\C}(\w{\D}_Y)$ and let $\N^\bullet\in \mathrm{D}_{\C}(\w{\D}_Y)$. The same strategy as before yields a natural morphism
\begin{equation*}
	\mathrm{R}f_*\mathcal{H}om_{\w{\D}_X}(\M^\bullet, f^!\N^\bullet)\to \mathrm{R}\mathcal{H}om_{\w{\D}_Y}(f_!\M^\bullet, \N^\bullet),
\end{equation*}
this time using Corollary \ref{projfortransfer}.(ii) for the isomorphism
\begin{equation*}
	\mathrm{R}\mathcal{H}om_{\w{\D}_X}(f^!\N^\bullet, \w{\D}_X)\widetilde{\otimes}_{\w{\D}_X}^{\mathbb{L}}f^*\w{\D}_Y\cong \mathrm{R}\mathcal{H}om_{\w{\D}_X}(f^!\N^\bullet, f^*\w{\D}_Y).\qedhere
\end{equation*}

\begin{thm}
\label{adjunctiontheorem}
Let $f: X\to Y$ be a separated morphism between smooth rigid analytic $K$-varieties. Assume that $Y$ has an admissible covering by D-regular affinoid subdomains $Y_i$ such that for each $i$, the following is satisfied:
\begin{enumerate}[(i)]
	\item $f^{-1}Y_i$ has countably many connected components.
	\item the direct image functor ${p_i}_*$ for the natural projection map $p_i: f^{-1}Y_i\times Y_i\to Y_i$ has finite cohomological dimension.
\end{enumerate}
Let $\M^\bullet \in \mathrm{D}_\C(\w{\D}_X)$ such that $f_!\M^\bullet \in \mathrm{D}_\C(\w{\D}_Y)$, and let $\N^\bullet \in \mathrm{D}_\C(\w{\D}_Y)$ such that $f^!\N^\bullet \in \mathrm{D}_\C(\w{\D}_X)$. Then we have a natural isomorphism
\begin{equation*}
\mathrm{R}\mathcal{H}om_{\w{\D}_Y}(f_!\M^\bullet, \N^\bullet)\cong \mathrm{R}f_*\mathrm{R}\mathcal{H}om_{\w{\D}_X}(\M^\bullet, f^!\N^\bullet).
\end{equation*}
Similarly, if $f_+\M^\bullet \in \mathrm{D}_\C(\w{\D}_Y)$ and $f^+\N^\bullet \in \mathrm{D}_\C(\w{\D}_X)$, then 
\begin{equation*}
\mathrm{R}f_*\mathrm{R}\mathcal{H}om_{\w{\D}_X}(f^+ \N^\bullet, \M^\bullet)\cong \mathrm{R}\mathcal{H}om_{\w{\D}_Y}(\N^\bullet, f_+\M^\bullet).
\end{equation*}
\end{thm}

\begin{proof}
Let $\iota: X\to X\times Y$ denote the graph embedding, which is a closed embedding, since $f$ is separated, and let $\mathrm{pr}: X\times Y\to Y$ be the (smooth) projection.

Since $\mathrm{pr}$ is smooth, $\mathrm{pr}^!\N^\bullet$ is a $\C$-complex such that $\iota^!\mathrm{pr}^!\N^\bullet\cong f^!\N^\bullet$ is a $\C$-complex on $X$. Likewise, $\iota_+\M^\bullet\in \mathrm{D}_{\C}(\w{\D}_{X\times Y})$ such that $\mathrm{pr}_+\iota_+\M^\bullet$ is a $\C$-complex.

Thanks to the construction above, we have a natural morphism
\begin{align*}
	\mathrm{R}f_*\mathrm{R}\mathcal{H}om_{\w{\D}_X}(\M^\bullet, f^!\N^\bullet)&\cong \mathrm{R}\mathrm{pr}_*\mathrm{R}\iota_*\mathrm{R}\mathcal{H}om_{\w{\D}_X}(\M^\bullet, \iota^!\mathrm{pr}^!\N^\bullet)\\
	&\to \mathrm{R}\mathrm{pr}_*\mathrm{R}\mathcal{H}om_{\w{\D}_{X\times Y}}(\iota_!\M^\bullet, \mathrm{pr}^!\N^\bullet)\\
	&\to \mathrm{R}\mathcal{H}om_{\w{\D}_Y}(f_!\M^\bullet, \N^\bullet).
\end{align*}

By construction of the map, we can now treat separately the case of a closed embedding and that of a suitable projection morphism, arguing locally in each case.

Let $\iota: X\to Y$ be a closed embedding of smooth affinoid. Assume $Y$ is D-regular, and $X$ is given as the vanishing set of some local coordinates. If $\N^\bullet\in \mathrm{D}_{\C}(\w{\D}_Y)$ such that $\iota^!\N^\bullet\in \mathrm{D}_{\C}(\w{\D}_X)$, then Corollary \ref{holimofpullback} yields
\begin{equation*}
	\iota^!\N^\bullet\cong \mathrm{holim} \iota^!(\N_n^\bullet).
\end{equation*}
Thus we have
\begin{equation*}
	\mathrm{R}\mathcal{H}om_{\w{\D}_Y}(\iota_!\M^\bullet, \N^\bullet)\cong \mathrm{holim}\mathrm{R}\mathcal{H}om_{\w{\D}_Y}(\iota_!\M^\bullet, \N_n^\bullet)
\end{equation*}
and
\begin{equation*}
	\mathrm{R}\iota_*\mathrm{R}\mathcal{H}om_{\w{\D}_X}(\M^\bullet, \iota^!\N^\bullet)\cong\mathrm{holim} \mathrm{R}\iota_*\mathrm{R}\mathcal{H}om_{\w{\D}_X}(\M^\bullet, \iota^!\N_n^\bullet).
\end{equation*}

By the same argument as in remark following Proposition \ref{homwithdual}, one obtains an isomorphism
\begin{equation*}
	\mathrm{R}\mathcal{H}om_{\w{\D}_X}(\M^\bullet, \iota^!\N_n^\bullet)\cong \mathrm{R}\mathcal{H}om_{\w{\D}_X}(\M^\bullet, \w{\D}_X)\widetilde{\otimes}_{\w{\D}_X}^{\mathbb{L}}\iota^!\N_n^\bullet.
\end{equation*}

Applying Lemma \ref{njbasechange}.(ii), we thus have the natural isomorphism
\begin{align*}
	\mathrm{R}\iota_*\mathrm{R}\mathcal{H}om_{\w{\D}_X}(\M^\bullet, \mathbb{L}\iota^*\N_n^\bullet)&\cong \mathrm{R}\iota_*(\mathrm{R}\mathcal{H}om_{\w{\D}_X}(\M^\bullet, \w{\D}_X)\widetilde{\otimes}_{\w{\D}_X}^{\mathbb{L}}\mathbb{L}\iota^*\N_n^\bullet)\\
	&\cong \iota_+\mathrm{R}\mathcal{H}om_{\w{\D}_X}(\M^\bullet, \w{\D}_X)\widetilde{\otimes}_{\w{\D}_Y}^{\mathbb{L}}\N_n^\bullet\\
	&\cong \mathrm{R}\mathcal{H}om_{\w{\D}_Y}(\iota_!\M^\bullet, \N_n^\bullet)[\mathrm{dim}Y-\mathrm{dim}X],
\end{align*}
as required.

Similarly, if $Y$ is a D-regular affinoid, $X$ has countably many connected components and the functor $f_*$ for $f: Z=X\times Y\to Y$ has finite cohomological dimension, we have again
\begin{equation*}
	\mathrm{R}\mathcal{H}om_{\w{\D}_Y}(f_!\M^\bullet, \N^\bullet)\cong \mathrm{holim}\mathrm{R}\mathcal{H}om_{\w{\D}_Y}(f_!\M^\bullet, \N_n^\bullet)
\end{equation*}
and
\begin{equation*}
	\mathrm{R}f_*\mathrm{R}\mathcal{H}om_{\w{\D}_Z}(\M^\bullet, f^!\N^\bullet)\cong \mathrm{holim}\mathrm{R}f_*\mathrm{R}\mathcal{H}om_{\w{\D}_Z}(\M^\bullet, f^!\N_n^\bullet).
\end{equation*}

Applying the remark following Proposition \ref{homwithdual} and Corollary \ref{smoothprofforNn} yields again
\begin{equation*}
	\mathrm{R}f_*\mathrm{R}\mathcal{H}om_{\w{\D}_Z}(\M^\bullet, f^!\N_n)\cong \mathrm{R}\mathcal{H}om_{\w{\D}_Y}(f_!\M^\bullet, \N_n^\bullet)
\end{equation*}
by the same argument as before. This proves the first adjunction.

The second adjunction follows by applying the first adjunction to the complexes $\mathbb{D}\M^\bullet $ and $\mathbb{D}\N^\bullet$ and invoking Lemma \ref{homsofduals} twice:
\begin{align*}
\mathrm{R}\mathcal{H}om_{\w{\D}_Y}(\N^\bullet, f_+\M^\bullet)&\cong \mathrm{R}\mathcal{H}om_{\w{\D}_Y}(f_!\mathbb{D}\M^\bullet, \mathbb{D}\N^\bullet)\\
&\cong \mathrm{R}f_*\mathrm{R}\mathcal{H}om_{\w{\D}_X}(\mathbb{D}\M^\bullet, f^!\mathbb{D}\N^\bullet)\\
&\cong \mathrm{R}f_*\mathrm{R}\mathcal{H}om_{\w{\D}_X}(f^+\N^\bullet, \M^\bullet).\qedhere
\end{align*}

\end{proof}

\appendix
\section{Further results on tensor products}
In this appendix, we provide the somewhat technical proofs of Theorem \ref{tensorlimit}, Lemma \ref{tensorlimiota}, and Lemma \ref{nbasechangelem}.
\subsection{Tensor products of $\C$-complexes}
In order to exploit the Auslander regularity results from section 5, we need to investigate the link between a $\C$-complex $\M^\bullet$ and its base change $\M_n^\bullet\in \mathrm{D}^b_{\mathrm{coh}}(\D_n)$ to `finite level'.

In the case when $\M^\bullet=\M$ is a coadmissible $\w{\D}_X$-module, we have $\M\cong \varprojlim \M_n$, so that we can infer properties of coadmissible $\w{\D}_X$-modules from the study of coherent $\D_n$-modules simply by taking the inverse limit.

In \cite{SixOp}, a derived analogue of this statement was formulated. As we have recalled in Lemma \ref{Ccomplexesasholim}, any $\C$-complex $\M^\bullet$ can locally be realized as a homotopy limit of the corresponding localisations $\M_n^\bullet:=\D_n\widetilde{\otimes}^{\mathbb{L}}_{\w{\D}_X}\M^\bullet$.

In what follows, it will therefore be crucial to understand how products (and thus, homotopy limits) behave under various functors, in particular under derived tensor products.

To motivate our argument, we first consider the case of modules (rather than sheaves of modules).

Recall from \cite[Definition 5.16]{SixOp} that we call a metrisable locally convex $K$-vector space $V$ \textbf{pseudo-nuclear} if $V\cong \varprojlim V_n$, where the $V_n$ are semi-normed spaces with transition maps $\rho_n: V_{n+1}\to V_n$ such that for each $n$, the following holds:
\begin{enumerate}[(i)]
	\item $\rho_n$ has dense image.
	\item if $B\subseteq V_{n+1}$ is bounded and $U\subseteq V_n$ is an open $R$-submodule, then there exists a bounded subset $B'\subseteq V$ such that 
	\begin{equation*}
		\rho_n(B)\subseteq U+p_n(B'),
	\end{equation*}
	where $p_n: V\to V_n$ is the natural morphism.
\end{enumerate}

We point out that any Banach space is pseudo-nuclear, and if $A$ is a Noetherian Banach $K$-algebra, then any Fr\'echet $A$-module which is nuclear relative to $A$ is pseudo-nuclear (\cite[Lemma 5.7.(ii)]{SixOp}).

One of the crucial properties of pseudo-nuclearity is the following:

\begin{lem}[{\cite[Corollary 5.23]{SixOp}}]
	\label{pnprod}
	Let $V_n$ be a countable collection of pseudo-nuclear Fr\'echet $K$-vector spaces of countable type, regarded as complete bornological spaces. If $W$ is another pseudo-nuclear Fr\'echet $K$-vector space of countable type, then
	\begin{equation*}
		(\prod_nV_n)\widetilde{\otimes}_K W\cong \prod_n(V_n\widetilde{\otimes}_K W)
	\end{equation*}
	in $LH(\h{\B}c_K)$ via the natural morphism.
\end{lem}

We now record the following observation.

\begin{lem}
	\label{flatprod}
	Let $A$ be a Fr\'echet $K$-algebra. For each $n\in \mathbb{N}$, let $N_n$ be a Fr\'echet left $A$-module. Let $M$ be a Fr\'echet right $A$-module. Assume that $A, M$ and $N_n$ are pseudo-nuclear of countable type for each $n$. 
	
	If $M\widetilde{\otimes}_A^{\mathbb{L}}N_n\cong M\widetilde{\otimes}_A N_n$ for each $n$, then
	\begin{equation*}
		M\widetilde{\otimes}_A^{\mathbb{L}}\prod_n N_n\cong \prod_n (M\widetilde{\otimes}_A^{(\mathbb{L})}N_n)
	\end{equation*}
	in $\mathrm{D}(LH(\h{\B}c_K))$.
	
	In particular, $\prod_n N_n\in LH(\h{\B}c_K)$ is $M\widetilde{\otimes}_A$-acyclic.
\end{lem}
\begin{proof}
	By \cite[Lemma 3.13]{SixOp}, we have
	\begin{equation*}
		M\widetilde{\otimes}_A^{\mathbb{L}}\prod N_n\cong M\overset{\rightarrow}{\otimes}_A^{\mathbb{L}}\prod N_n.
	\end{equation*}
	
	Since $\overset{\rightarrow}{\otimes}_K$ is exact, it follows from \cite[Lemma 2.9]{BamStein} and \cite[Proposition 4.25]{SixOp} that an $M\widetilde{\otimes}_A$-acyclic resolution of $\prod_n N_n$ is given by the Bar resolution
	\begin{equation*}
		\hdots\to A^{\widetilde{\otimes}r}\widetilde{\otimes}_K \prod_n N_n\to A^{\widetilde{\otimes}r-1}\widetilde{\otimes}_K \prod_n N_n\to \hdots \to A\widetilde{\otimes}_K \prod_n N_n\to 0.
	\end{equation*}
	Tensoring with $M$, each term of the complex is of the form 
	\begin{equation*}
		M\widetilde{\otimes}_K A^{\widetilde{\otimes}r-1}\widetilde{\otimes}_K \prod N_n,
	\end{equation*}
	and invoking Lemma \ref{pnprod} together with the exactness of products in $LH(\h{\B}c_K)$ (\cite[Proposition 2.1.15]{Schneiders}) yields
	\begin{equation*}
		M\widetilde{\otimes}_A^{\mathbb{L}}\prod_n N_n\cong \prod_n M\widetilde{\otimes}_A^{\mathbb{L}}N_n.
	\end{equation*}
	By assumption, each $N_n$ is $M\widetilde{\otimes}_A$-acyclic, so it then follows from the exactness of products in $LH(\h{\B}c_K)$ that the augmented complex
	\begin{equation*}
		\hdots \to M\widetilde{\otimes}_K A^{\widetilde{\otimes}r-1}\widetilde{\otimes}_K \prod_n N_n\to\hdots \to M\widetilde{\otimes}_K \prod_n N_n\to M\widetilde{\otimes}_A \prod_n N_n\to 0
	\end{equation*}
	is exact, and $\prod_nN_n$ is $M\widetilde{\otimes}_A$-acyclic.
\end{proof}

Let $U\cong \varprojlim U_n$ be a Fr\'echet--Stein $K$-algebra, nuclear over some Noetherian Banach algebra $A$ with (almost) Noetherian unit ball. Assume that $A$ and $U_n$ are of countable type.

Recall from \cite[Definition 8.10]{SixOp} that $M^\bullet\in \mathrm{D}(\mathrm{Mod}_{LH(\h{\B}c_K)}(U))$ is called a $\C$-complex if $M_n^\bullet:=U_n\widetilde{\otimes}_{U}M^\bullet\in \mathrm{D}^b_{\mathrm{f.g.}}(U_n)$ and the natural morphism $\mathrm{H}^j(M^\bullet)\to \varprojlim \mathrm{H}^j(M_n^\bullet)$ is an isomorphism. We denote the full subcategory of $\C$-complexes of $U$-modules by $\mathrm{D}_{\C}(U)$.

\begin{lem}
	\label{tensorholimmodule}
	Let $U\cong \varprojlim U_n$ be a Fr\'echet--Stein algebra as above.
	
	Let $M^\bullet\in \mathrm{D}_{\C}^-(U^{\mathrm{op}})$, $N^\bullet\in \mathrm{D}_{\C}^-(U)$. Then 
	\begin{equation*}
		M^\bullet\widetilde{\otimes}_U^{\mathbb{L}}N^\bullet\cong \mathrm{holim}_n M_n^\bullet\widetilde{\otimes}_{U_n}^{\mathbb{L}}N_n^\bullet
	\end{equation*}
	in $\mathrm{D}(LH(\h{\B}c_K))$, where $M_n^\bullet=M^\bullet\widetilde{\otimes}^{\mathbb{L}}_U U_n$, $N_n^\bullet=U_n\widetilde{\otimes}_U^{\mathbb{L}}N^\bullet$.
\end{lem}

\begin{proof}
	Note that by \cite[Lemma 8.11]{SixOp} $N^\bullet\cong \mathrm{holim} N_n^\bullet$, i.e. there is a distinguished triangle
	\begin{equation*}
		N^\bullet\to \prod_n N_n\to\prod_n N_n.
	\end{equation*}
	
	It thus suffices to show that the natural morphism
	\begin{equation*}
		M^\bullet\widetilde{\otimes}_U^{\mathbb{L}}\prod_n N_n^\bullet\to\prod_n (M^\bullet\widetilde{\otimes}_U^{\mathbb{L}}N_n^\bullet)
	\end{equation*}
	is an isomorphism.
	
	First, suppose that $M^\bullet=M$ is a coadmissible right $U$-module. Without loss of generality, $N^\bullet\in \mathrm{D}^{\leq 0}_{\C}(U)$, so by assumption, each $N_n^\bullet\in \mathrm{D}^{\leq 0}_{\mathrm{f.g.}}(U_n)$ can be represented by a complex $P_n^\bullet$ concentrated in non-positive degrees, with each $P_n^i$ a finitely generated projective left $U_n$-module.
		
	By exactness of products in $LH(\h{\B}c_K)$, $\prod_n P_n^\bullet$ is quasi-isomorphic to $\prod_n N_n^\bullet$.	
	
	Note further that $M\widetilde{\otimes}_U^{\mathbb{L}}U_n\cong M\widetilde{\otimes}_U U_n$ (compare \cite[Corollary 5.38]{SixOp}). Since $P_n^i$ is a direct summand of $U_n^{\oplus r}$ for some $r$, it follows from Lemma \ref{flatprod} that $\prod_n P_n^i$ is $M\widetilde{\otimes}_U$-acyclic for each $i$, and thus
	\begin{equation*}
		M\widetilde{\otimes}_U^{\mathbb{L}}\prod_n N_n^\bullet\cong M\widetilde{\otimes}_U \prod_n P_n^\bullet.
	\end{equation*}	
	
	But then Lemma \ref{flatprod} also implies that
	\begin{equation*}
		M\widetilde{\otimes}_U \prod_n P_n^i\cong \prod_n (M\widetilde{\otimes}_U P_n^i)
	\end{equation*}
	for each $i$. As $M\widetilde{\otimes}_U P_n^\bullet\cong M\widetilde{\otimes}_U^{\mathbb{L}}N_n^\bullet$ and products are exact, we can thus deduce that
	\begin{align*}
		M\widetilde{\otimes}_U^{\mathbb{L}}\prod_n N_n^\bullet&\cong M\widetilde{\otimes}_U \prod_n P_n^\bullet\\
		&\cong \prod_n M\widetilde{\otimes}_U P_n^\bullet\\
		&\cong \prod_n M\widetilde{\otimes}_U^{\mathbb{L}} N_n^\bullet, 
	\end{align*}
	as desired, proving the Lemma in the case where $M^\bullet$ is concentrated in one degree.
	
	We immediately obtain the same isomorphism for $M^\bullet\in \mathrm{D}^b_{\C}(U^{\mathrm{op}})$. Now let $M^\bullet\in \mathrm{D}^-_{\C}(U^{\mathrm{op}})$. Assuming without loss of generality that $N^\bullet$ is in $\mathrm{D}^{\leq 0}_{\C}(U)$, we have
	\begin{align*}
		\mathrm{H}^j(M^\bullet\widetilde{\otimes}_U^{\mathbb{L}}\prod N_n^\bullet)&\cong \mathrm{H}^j(\tau^{\geq j}M^\bullet\widetilde{\otimes}_U^{\mathbb{L}}\prod N_n^\bullet)\\
		&\cong \mathrm{H}^j(\prod \tau^{\geq j}M^\bullet\widetilde{\otimes}_U^{\mathbb{L}}N_n^\bullet)\\
		&\cong \prod \mathrm{H}^j(\tau^{\geq j}M^\bullet\widetilde{\otimes}_U^{\mathbb{L}}N_n^\bullet)\\
		&\cong \prod \mathrm{H}^j(M^\bullet\widetilde{\otimes}_U^{\mathbb{L}}N_n^\bullet)\\
		&\cong \mathrm{H}^j(\prod M^\bullet\widetilde{\otimes}_U^{\mathbb{L}}N_n^\bullet)
	\end{align*}
	for each $j$, proving the result.
\end{proof}

Our goal is now to carefully prove the analogous statement for sheaves on some rigid analytic $K$-variety $X$. The overall proof strategy is exactly the same as above, but this requires some extra care, as products are in general not exact in $\mathrm{Shv}(X, LH(\h{\B}c_K))$.

Note that products in the homotopy category $K(\mathrm{Shv}(X, LH(\h{\B}c_K)))$ exist and can be calculated termwise (see \cite[Lemma 1.4.2]{Schneiders}).

Moreover, products exist in $\mathrm{D}(\mathrm{Shv}(X, LH(\h{\B}c_K)))$, but they need to be calculated in the following way (see \cite[tag 07D9]{stacksproj}): if $\F_n^\bullet$ is a chain complex of sheaves for each $n\geq 0$, take a quasi-isomorphism $\F_n^\bullet\to \I_n^\bullet$ to some K-injective complex $\I_n^\bullet$. Then the product of the $\F_n^\bullet$ in $\mathrm{D}(\mathrm{Shv}(X, LH(\h{\B}c_K)))$ can be represented by the chain complex which is the termwise product of the $\I_n^\bullet$.

To avoid ambiguity, we let $\prod$ denote the direct product in the derived category and use $\prod^K$ whenever we take products in the homotopy category (i.e., termwise products).

In this way, the statement above reads: $\prod\F_n^\bullet$ is represented by $\prod^K \I_n^\bullet$.  

We will not make this distinction when working with chain complexes in $LH(\h{\B}c_K)$, where products are exact.

We will call $\F\in \mathrm{Shv}(X, LH(\h{\B}c_K))$ \textbf{locally acyclic} if $\mathrm{H}^j(U, \F)=0$ for all $j>0$ and all affinoid subdomains $U\subseteq X$.

\begin{lem}
	\label{belowacyclic}
	For each $n\geq 0$, let $\F_n^\bullet\in K(\mathrm{Shv}(X, LH(\h{\B}c_K)))$ be a bounded below chain complex consisting of locally acyclic sheaves. Then $\prod \F_n^\bullet$ is represented by $\prod^K \F_n^\bullet$.
\end{lem}
\begin{proof}
	For each $n$, let $\F_n^\bullet\to \mathcal{I}_n^\bullet$ be a quasi-isomorphism to a bounded below complex of injective sheaves, so that $\prod \F_n^\bullet$ is represented by $\prod^K \I_n^\bullet$.
	
	Let $C_n^\bullet$ denote the cone of the quasi-isomorphism $\F_n^\bullet\to \mathcal{I}_n^\bullet$ in $K(\mathrm{Shv}(X, LH(\h{\B}c_K)))$. Note that this is an exact bounded below complex, and for each $i$, $\C_n^i$ is locally acyclic, as the same is true for $\F_n^\bullet$ and $\mathcal{I}_n^\bullet$ (recall that injective sheaves are acyclic).
	
	The morphisms above induce a natural morphism $\prod^K \F_n^\bullet\to \prod^K \mathcal{I}_n^\bullet$ in the homotopy category $K(\mathrm{Shv}(X, LH(\h{\B}c_K)))$, whose cone is precisely $\prod^K \C_n^\bullet$. To show that $\prod^K \F_n^\bullet\to \prod^K \mathcal{I}_n^\bullet$ is a quasi-isomorphism, we thus only need to verify that $\prod^K \C_n^\bullet$ is exact.
	
	Since $\C_n^\bullet$ is exact and bounded below, one can show inductively that the kernel of each differential in the complex is also locally acyclic, and thus $\Gamma(U, \C_n^\bullet)$ is an exact complex for each $n$ and each affinoid subdomain $U$. As products in $LH(\h{\B}c_K)$ are exact, we deduce that $\prod_n \Gamma(U, \C_n^\bullet)=\Gamma(U, \prod^K \C_n^\bullet)$ is an exact complex for each affinoid $U$, so that $\prod^K \C_n^\bullet$ is an exact complex, as required.    
\end{proof}

Combining the Lemmas above, it is easy to guess what our strategy will be: as in Lemma \ref{tensorholimmodule}, we wish to consider products of projective resolutions, and by Lemma \ref{belowacyclic}, these are easy to understand if each projective resolutions is bounded below, i.e. finite. In this way, Auslander regularity (in particular, finite global dimension) becomes the crucial assumption which allows us to generalize Lemma \ref{tensorholimmodule} to the sheaf setting. 

We now consider the following setting: let $X=\Sp A$ be an affinoid $K$-variety with free tangent sheaf. Let $\A\subseteq A$ be an admissible affine formal model and let $\L\subseteq \T_X(X)$ be a free $(R, \A)$-Lie lattice. Fix a positive integer $m$, so that we can work on the subsite $X_m=X(\pi^m\L)$ of $\pi^m$-accessible subdomains.

We write $\L_m=\pi^m\L$. For each $n> m$, let $\L_n\subsetneq \L_{n-1}$ be a free $(R, \A)$-Lie lattice of $\T_X(X)$.

By construction, each $\L_n$-accessible affinoid subdomain is $\L_m$-accessible, so we obtain monoids $\mathscr{U}_n$ in $\mathrm{Shv}(X_m, LH(\h{\B}c_K))$ with
\begin{equation*}
	\mathscr{U}_n(X)=\h{U_{\A}(\L_n)}_K,
\end{equation*}
and
\begin{equation*}
	\mathscr{U}_n(U)=\h{U_{\B}(\B\otimes_{\A}\L_n)}_K
\end{equation*}
for any $U=\Sp B$ in $X_m$ with suitable admissible affine formal model $\B$.

We make the following assumptions:
\begin{enumerate}[(i)]
	\item the algebras $\mathscr{U}_n(X)$ are Auslander regular for each $n$, and there exists a $C$ such that $\mathrm{gl.dim.}\mathscr{U}_n(X)<C$ for each $n$.
	\item for any $U\in X_m$, the natural morphism $\mathscr{U}_n(U)\to \mathscr{U}_{n-1}(U)$ is flat.
	\item for any $U\in X_m$, there exists some Noetherian Banach $K$-algebra $A_U$ of countable type, with (almost) Noetherian unit ball, such that the natural morphism $\mathscr{U}_n(U)\to \mathscr{U}_{n-1}(U)$ is strictly completely continuous relative to $A_U$.
\end{enumerate}

We write $\mathscr{U}=\varprojlim \mathscr{U}_n$. Our assumptions imply that $\mathscr{U}(U)$ is a Fr\'echet--Stein $K$-algebra nuclear over $A_U$. Moreover $\mathscr{U}_n(U)$ is of countable type for each $U$ and each $n$.

In the case when $\L_n=\pi^n\L$, this simply recovers $\mathscr{U}_n=\D_n|_{X_m}$ and $\mathscr{U}=\w{\D}_X|_{X_m}$, but it is convenient to formulate our results in this generality, so that they can also be applied to the sheaves $\mathscr{U}_n=\D'_{m, n}$ in subsection 6.2, where $\mathscr{U}=\D'_{m, \infty}$ (see \cite[proof of Theorem 4.1.11]{Ardakov} and \cite[Theorem 4.9]{DcapOne} for the verification of (ii) and (iii) in that case).

Notions of coherent $\mathscr{U}_n$-modules, coadmissible $\mathscr{U}$-modules and $\C$-complexes in $\mathrm{D}(\mathrm{Mod}_{\mathrm{Shv}(X_m, LH(\h{\B}c_K))}(\mathscr{U}))$ can be developed in exactly the same way as for $\w{\D}_X$-modules.

We first note that coadmissible $\mathscr{U}$-modules admit a variant of the Bar resolution:

\begin{lem}
	\label{Barsheavesapp}
	Let $\N$ be a coadmissible $\mathscr{U}$-module, and let
	\begin{equation*}
		F^\bullet:=(\hdots\to \mathscr{U}(X)^{\widetilde{\otimes}r}\widetilde{\otimes}_K \N(X)\to \mathscr{U}(X)^{\widetilde{\otimes}r-1}\widetilde{\otimes}_K \N(X)\to \hdots \mathscr{U}(X)\widetilde{\otimes}_K \N(X)\to 0 )
	\end{equation*}
	 be the Bar resolution of $\N(X)$ as a $\mathscr{U}(X)$-module.
	\begin{enumerate}[(i)]
		\item The complex $\mathscr{U}\widetilde{\otimes}_{\mathscr{U}(X)} F^\bullet$ defines a resolution of $\N$.
		\item If $\M\in \mathrm{Mod}_{\mathrm{Shv}(X, LH(\h{\B}c_K))}(\mathscr{U}^{\mathrm{op}})$ such that $\M(U)\in \mathrm{Ind}(\mathrm{Ban}_K)$ for any affinoid subdomain $U\subseteq X$, then
		\begin{equation*}
			\M\widetilde{\otimes}_{\mathscr{U}}^{\mathbb{L}}\N\cong \M\widetilde{\otimes}_{\mathscr{U}} (\mathscr{U}\widetilde{\otimes}_{\mathscr{U}(X)} F^\bullet).
		\end{equation*}
	\end{enumerate}
\end{lem}
\begin{proof}
	Let $U$ be an affinoid subdomain of $X$. By the argument in Lemma \ref{flatprod}, $\mathscr{U}(U)\widetilde{\otimes}^\mathbb{L}_{\mathscr{U}(X)} \N(X)$ can be computed via the complex $\mathscr{U}(U)\widetilde{\otimes}_{\mathscr{U}(X)} F^\bullet$. But by \cite[Proposition 3.52]{SixOp} and \cite[Proposition 4.6]{DcapOne},
	\begin{equation*}
		\mathscr{U}(U)\widetilde{\otimes}^\mathbb{L}_{\mathscr{U}(X)} \N(X)\cong \mathscr{U}(U)\widetilde{\otimes}_{\mathscr{U}(X)} \N(X)\cong \N(U),
	\end{equation*}
	so $\mathscr{U}(U)\widetilde{\otimes}_{\mathscr{U}(X)} F^\bullet$ does indeed describe a resolution of $\N(U)$. 
	
	Note that 
	\begin{equation*}
		(\mathscr{U}\widetilde{\otimes}_{\mathscr{U}(X)} F^{-i})(U)=\mathscr{U}(U)\widetilde{\otimes}_K \mathscr{U}(X)^{\widetilde{\otimes}i} \widetilde{\otimes}_K \N(X),
	\end{equation*}
	i.e. no sheafification is necessary, compare \cite[Lemma 7.10]{SixOp}. In particular, $\mathscr{U}\widetilde{\otimes}_{\mathscr{U}(X)} F^\bullet$ is a resolution of $\N$, proving (i).
	
	For (ii), note that 
	\begin{equation*}
		\M\widetilde{\otimes}_{\mathscr{U}}^{\mathbb{L}} (\mathscr{U}\widetilde{\otimes}_{\mathscr{U}(X)} F^{-i})\cong \M\widetilde{\otimes}_{\mathscr{U}}^{\mathbb{L}}(\mathscr{U}\widetilde{\otimes}_K \mathscr{U}(X)^{\widetilde{\otimes}i}\widetilde{\otimes}_K \N(X)).
	\end{equation*}
	By \cite[Lemma 3.32, Proposition 4.21]{SixOp}, if $\F, \G\in \mathrm{Shv}(X, \mathrm{Ind}(\mathrm{Ban}_K))$, then
	\begin{equation*}
		\F\widetilde{\otimes}_K^{\mathbb{L}}\G\cong \F\widetilde{\otimes}_K \G\cong \F\overset{\rightarrow}{\otimes}_K \G,
	\end{equation*}
	so
	\begin{equation*}
		\M\widetilde{\otimes}_{\mathscr{U}}^{\mathbb{L}} (\mathscr{U}\widetilde{\otimes}_{\mathscr{U}(X)} F^{-i})\cong \M\widetilde{\otimes}_K \mathscr{U}(X)^{\widetilde{\otimes}i}\widetilde{\otimes}_K \N(X)
	\end{equation*}
	is concentrated in a single degree. Hence $\mathscr{U}\widetilde{\otimes}_{\mathscr{U}(X)} F^{-i}$ is $\M\widetilde{\otimes}_{\mathscr{U}}$-acyclic for each $i$, proving the result.
\end{proof}

Note that if $\N^\bullet\in \mathrm{D}_\C(\mathscr{U})$ is a $\C$-complex, then $\N_n^\bullet=\mathscr{U}_n\widetilde{\otimes}_{\mathscr{U}} \N^\bullet\in \mathrm{D}^b_{\mathrm{coh}}(\mathscr{U}_n)$ can be represented by a finite chain complex $\mathcal{P}_n^\bullet$ consisting of coherent $\mathscr{U}_n$-modules, thanks to the assumption that $\L_n$ is regular, so that $\mathscr{U}_n(X)$ is Auslander regular. In fact, we can assume that each $\mathcal{P}_n^i$ is the localisation of a finitely generated projective $\mathscr{U}_n(X)$-module. As coherent $\mathscr{U}_n$-modules are locally acyclic on $X_m$ (see \cite[Proposition 5.1]{DcapOne}), it follows from Lemma \ref{belowacyclic} above that $\prod_{n\geq m}\N_n^\bullet$ is represented by $\prod_{n\geq m}^K \mathcal{P}_n^\bullet$ on $X_m$. 

This allows to give an analogue of Lemma \ref{flatprod}:

\begin{lem}
	\label{acyclicproduct}
	Let $\M$ be a right coadmissible $\mathscr{U}$-module. For each $n\geq m$, let $\mathcal{P}_n$ denote the localisation of a finitely generated projective $\mathscr{U}_n(X)$-module. Then there is a natural isomorphism
	\begin{equation*}
		\M\widetilde{\otimes}^\mathbb{L}_{\mathscr{U}}\prod_{n\geq m} \mathcal{P}_n\cong \M\widetilde{\otimes}_{\mathscr{U}} \prod_{n\geq m} \mathcal{P}_n\cong \prod_{n\geq m} \M\widetilde{\otimes}_{\mathscr{U}}\mathcal{P}_n
	\end{equation*}
	in $\mathrm{D}(\mathrm{Shv}(X_m, LH(\h{\B}c_K)))$.
\end{lem}
\begin{proof}
	Note that $\prod \mathcal{P}_n$ is represented by $\prod^K \mathcal{P}_n$ and $\prod \M\widetilde{\otimes}_{\mathscr{U}}\mathcal{P}_n$ is represented by $\prod^K \M\widetilde{\otimes}_{\mathscr{U}} \mathcal{P}_n$ by the above.
	
	Let $F^\bullet$ be the Bar resolution of $\M(X)$ as a right $\mathscr{U}(X)$-module, i.e.
	\begin{equation*}
		F^{-i}=\M(X)\widetilde{\otimes}_K \mathscr{U}(X)^{\widetilde{\otimes}i+1}.
	\end{equation*}
	By the right-module version of Lemma \ref{Barsheavesapp}, 
	\begin{equation*}
		\F^\bullet:=F^\bullet\widetilde{\otimes}_{\mathscr{U}(X)} \mathscr{U}
	\end{equation*}
	is a $-\widetilde{\otimes}_{\mathscr{U}}\prod \mathcal{P}_n$-acyclic resolution of $\M$.
	
	Thus $\M\widetilde{\otimes}^\mathbb{L}_{\mathscr{U}} \prod \mathcal{P}_n$ can be represented by the complex
	\begin{equation*}
		\F^\bullet\widetilde{\otimes}_{\mathscr{U}}\prod \mathcal{P}_n.
	\end{equation*}
	
	Let $U$ be an affinoid subdomain in $X_m$. Then
	\begin{equation*}
		(\F^{-i}\widetilde{\otimes}_{\mathscr{U}}\prod \mathcal{P}_n)(U)=\M(X)\widetilde{\otimes}_K\mathscr{U}(X)^{\widetilde{\otimes}i}\widetilde{\otimes}_K \prod (\mathcal{P}_n(U)).
	\end{equation*}
	Note that no sheafification is necessary, since $U$ is quasi-compact and $\overset{\rightarrow}{\otimes}_K$ is exact.
	
	Now by Lemma \ref{pnprod} we have
	\begin{align*}
		(\F^{-i}\widetilde{\otimes}_{\mathscr{U}}\prod \mathcal{P}_n)(U)&\cong\M(X)\widetilde{\otimes}_K \mathscr{U}(X)^{\widetilde{\otimes}i}\widetilde{\otimes}_K \prod (\mathcal{P}_n(U))\\
		&\cong \prod \M(X)\widetilde{\otimes}_K \mathscr{U}(X)^{\widetilde{\otimes}i}\widetilde{\otimes}_K \mathcal{P}_n(U)\\
		&\cong (\prod {}^K \F^{-i}\widetilde{\otimes}_{\w{\D}_X}\mathcal{P}_n)(U)
	\end{align*}
	and thus
	\begin{equation*}
		\F^{-i}\widetilde{\otimes}_{\mathscr{U}} \prod_{n\geq m} \mathcal{P}_n\cong \prod_{n\geq m}{}^K \F^{-i}\widetilde{\otimes}_{\mathscr{U}}\mathcal{P}_n
	\end{equation*}
	via the natural morphism.
	
	Thus $\M\widetilde{\otimes}^\mathbb{L}_{\mathscr{U}} \prod \mathcal{P}_n$ is represented by the complex $\prod^K \F^\bullet\widetilde{\otimes}_{\mathscr{U}}\mathcal{P}_n$, and it suffices to show that
	\begin{equation*}
		\prod{}^K\F^\bullet\widetilde{\otimes}_{\mathscr{U}}\mathcal{P}_n\to \prod{}^K\M\widetilde{\otimes}_{\mathscr{U}} \mathcal{P}_n\to 0
	\end{equation*}
	is exact in $\mathrm{Shv}(X_m, LH(\h{\B}c_K))$.
	
	It suffices to check for exactness after applying $\Gamma(U, -)$ for an arbitrary affinoid subdomain $U\in X_m$. But by the above, this yields the complex
	\begin{equation*}
		\prod F^\bullet\widetilde{\otimes}_{\mathscr{U}(X)} \mathcal{P}_n(U)\to \prod \M(U)\widetilde{\otimes}_{\mathscr{U}(U)}\mathcal{P}_n(U)\to 0,
	\end{equation*}
	where $(\M\widetilde{\otimes}_{\mathscr{U}} \mathcal{P}_n)(U)=\M(U)\widetilde{\otimes}_{\mathscr{U}(U)} \mathcal{P}_n(U)$ follows from the case where $\mathcal{P}_n=\mathscr{U}_n$.
	
	Now $F^\bullet\widetilde{\otimes}_{\mathscr{U}(X)} \mathcal{P}_n(U)$ represents $\M(X)\widetilde{\otimes}^{\mathbb{L}}_{\mathscr{U}(X)} \mathcal{P}_n(U)$, and it follows from the coadmissibility of $\M$ that
	\begin{equation*}
		\M(U)\widetilde{\otimes}_{\mathscr{U}(U)} \mathcal{P}_n(U)\cong (\M(X)\widetilde{\otimes}_{\mathscr{U}(X)} \mathscr{U}(U))\widetilde{\otimes}_{\mathscr{U}(U)}\mathcal{P}_n(U)\cong \M(X)\widetilde{\otimes}_{\mathscr{U}(X)} \mathcal{P}_n(U).
	\end{equation*}
	Since \cite[Corollary 5.38, Lemma 5.32]{SixOp} implies that
	\begin{equation*}
		\M(X)\widetilde{\otimes}^\mathbb{L}_{\mathscr{U}(X)} \mathcal{P}_n(U)\cong \M(X)\widetilde{\otimes}^{\mathbb{L}}_{\mathscr{U}(X)} \mathscr{U}_n(X)\widetilde{\otimes}^{\mathbb{L}}_{\mathscr{U}_n(X)}\mathcal{P}_n(U)\cong \M(X)\widetilde{\otimes}_{\mathscr{U}(X)} \mathcal{P}_n(U),
	\end{equation*}
	the complex
	\begin{equation*}
		F^\bullet\widetilde{\otimes}_{\mathscr{U}(X)} \mathcal{P}_n(U)\to \M(X)\widetilde{\otimes}_{\mathscr{U}(X)}\mathcal{P}_n(U)\to 0
	\end{equation*}
	is exact, and as products are exact in $LH(\h{\B}c_K)$, the result follows.
\end{proof}

We are now in a position to prove Theorem \ref{tensorlimit} and Lemma \ref{tensorlimiota} from subsection 6.1 and 6.2, respectively.

\begin{lem}
	\label{tensorasholimapp}
	If $\M^\bullet\in \mathrm{D}_\C^-(\mathscr{U}^\mathrm{op})$, $\N^\bullet\in \mathrm{D}^-_\C(\mathscr{U})$, write $\M^\bullet_n=\M^\bullet \widetilde{\otimes}^\mathbb{L}_{\mathscr{U}} \mathscr{U}_n$, similarly for $\N^\bullet$. Then there are natural isomorphisms
	\begin{equation*}
		\M^\bullet \widetilde{\otimes}^\mathbb{L}_{\mathscr{U}} (\prod_{n\geq m} \N^\bullet_n) \cong \prod_{n\geq m} (\M^\bullet\widetilde{\otimes}^\mathbb{L}_{\mathscr{U}} \N_n^\bullet)\cong (\prod_{n\geq m} \M^\bullet_n)\widetilde{\otimes}^\mathbb{L}_{\w{\D}_X} \N^\bullet
	\end{equation*}
	in $\mathrm{D}(\mathrm{Shv}(X_m, LH(\h{\B}c_K)))$.
\end{lem}
\begin{proof}
	Suppose that $\M^\bullet=\M$ is concentrated in a single degree, i.e. is a coadmissible right $\mathscr{U}$-module. Let without loss of generality $\N^\bullet\in \mathrm{D}^{\leq 0}_\C(\mathscr{U})$, and note that due to Auslander regularity, $\N_n^\bullet=\mathscr{U}_n\widetilde{\otimes}^\mathbb{L}_{\mathscr{U}}\N^\bullet$ can be represented by a finite complex $\mathcal{P}_n^\bullet$, concentrated in non-positive degrees, where each term is the localisation of a finitely generated projective $\mathscr{U}_n(X)$-module. By Lemma \ref{belowacyclic} and Lemma \ref{acyclicproduct}, $\prod^K \mathcal{P}_n^\bullet$ is an $\M\widetilde{\otimes}-$acyclic resolution of $\prod \N_n^\bullet$, so that
	\begin{equation*}
		\M\widetilde{\otimes}^\mathbb{L}_{\mathscr{U}} \prod \N_n^\bullet
	\end{equation*}
	can then be represented by the complex $\M\widetilde{\otimes}_{\mathscr{U}} \prod^K \mathcal{P}_n^\bullet$, and Lemma \ref{acyclicproduct} also implies that this is isomorphic as a chain complex to $\prod^K \left(\M\widetilde{\otimes}_{\mathscr{U}} \mathcal{P}_n^\bullet\right)$.
	
	But by Lemma \ref{belowacyclic}, this complex represents $\prod (\M\widetilde{\otimes}^\mathbb{L}_{\mathscr{U}} \N_n^\bullet)$, since for each $i$, $\M\widetilde{\otimes}_{\mathscr{U}}\mathcal{P}_n^i$ is a direct summand of $\M\widetilde{\otimes}_{\mathscr{U}}\mathscr{U}_n^r$ for some $r$, which is locally acyclic by \cite[Proposition 5.1]{DcapOne}.
	
	This proves the first isomorphism in the Lemma in the case when $\M^\bullet$ is concentrated in a single degree.
	
	An induction on cohomological length now immediately yields the same result for the case when $\M^\bullet\in \mathrm{D}^b_\C(\mathscr{U}^\mathrm{op})$. If $\M^\bullet$ is only bounded above, suppose that $\N^\bullet$ is exact above degree $0$, then $\N_n^\bullet$ can be represented by a finite complex $\mathcal{P}_n^\bullet$ as above which is additionally $0$ in positive degrees. By Lemma \ref{belowacyclic}, $\prod \N_n^\bullet$ is represented by $\prod^K \mathcal{P}_n^\bullet$, which is also zero in positive degrees. 
	
	Note that in the same way, $\M_n^\bullet\in  \mathrm{D}^b_{\mathrm{coh}}(\mathscr{U}_n)$ can be represented by a bounded complex $\mathcal{Q}_n^\bullet$, where each term is a coherent $\mathscr{U}_n$-module. Moreover, as coherent $\mathscr{U}_n$-modules form an abelian category, we know that the complex $\tau^{\geq j}(\Q_n^\bullet)$, which represents $\tau^{\geq j}(\M_n^\bullet)$, also consists of coherent $\mathscr{U}_n$-modules.
	
	Then for any $j$,
	\begin{align*}
		\mathrm{H}^j(\M^\bullet\widetilde{\otimes}^\mathbb{L}_{\mathscr{U}}(\prod\N_n^\bullet))&\cong \mathrm{H}^j((\tau^{\geq j}\M^\bullet)\widetilde{\otimes}^\mathbb{L}_{\mathscr{U}}(\prod \N_n^\bullet))\\
		&\cong \mathrm{H}^j(\prod( \tau^{\geq j} \M^\bullet\widetilde{\otimes}^\mathbb{L}_{\mathscr{U}}\N_n^\bullet))
	\end{align*}
	by the argument above.
	
	Note that $\tau^{\geq j}\M^\bullet\widetilde{\otimes}^\mathbb{L}_{\mathscr{U}}\N_n^\bullet=\tau^{\geq j}(\M_n^\bullet)\widetilde{\otimes}^\mathbb{L}_{\mathscr{U}_n}\N_n^\bullet$ by \cite[Corollary 5.38]{SixOp}, and this complex is represented by the total complex $\mathrm{Tot}(\tau^{\geq j}\mathcal{Q}_n^\bullet\widetilde{\otimes}_{\mathscr{U}_n} \mathcal{P}_n^\bullet)$. Since each term in $\tau^{\geq j}\mathcal{Q}_n^\bullet$ is coherent and $\mathcal{P}_n^i$ is a direct summand of some free $\mathscr{U}_n$-module of finite rank, it follows that each term in $\mathrm{Tot}(\tau^{\geq j}\mathcal{Q}_n^\bullet\widetilde{\otimes}_{\D_n}\mathcal{P}_n^\bullet)$ is locally acyclic, so that
	\begin{equation*}
		\prod \tau^{\geq j}\M^\bullet\widetilde{\otimes}^\mathbb{L}_{\mathscr{U}}\N_n^\bullet\cong \prod{}^K \mathrm{Tot}(\tau^{\geq j}\mathcal{Q}_n^\bullet\widetilde{\otimes}_{\mathscr{U}_n} \mathcal{P}_n^\bullet).
	\end{equation*}
	Hence
	\begin{align*}
		\mathrm{H}^j(\M^\bullet\widetilde{\otimes}^\mathbb{L}_{\mathscr{U}} \prod \N_n^\bullet)&\cong \mathrm{H}^j(\prod{}^K\mathrm{Tot}(\tau^{\geq j}\mathcal{Q}_n^\bullet\widetilde{\otimes}_{\mathscr{U}_n}\mathcal{P}_n^\bullet))\\
		&\cong \mathrm{H}^j(\prod{}^K \mathrm{Tot}(\mathcal{Q}_n^\bullet\widetilde{\otimes}_{\mathscr{U}_n}\mathcal{P}_n^\bullet))\\
		&\cong \mathrm{H}^j(\prod (\M_n^\bullet\widetilde{\otimes}^\mathbb{L}_{\mathscr{U}_n}\N_n^\bullet))\\
		&\cong \mathrm{H}^j(\prod \M^\bullet\widetilde{\otimes}^\mathbb{L}_{\mathscr{U}}\N_n^\bullet),
	\end{align*}
	and the result follows.
	The same argument, swapping the roles of $\M^\bullet$ and $\N^\bullet$, proves the second isomorphism, noting that 
	\begin{equation*}
		\M^\bullet_n\widetilde{\otimes}^\mathbb{L}_{\mathscr{U}} \N^\bullet\cong \M^\bullet \widetilde{\otimes}^\mathbb{L}_{\mathscr{U}} \mathscr{U}_n\widetilde{\otimes}^\mathbb{L}_{\mathscr{U}}\N^\bullet \cong \M^\bullet \widetilde{\otimes}^\mathbb{L}_{\mathscr{U}} \N^\bullet_n.\qedhere
	\end{equation*}
\end{proof}

\begin{cor}
	\label{tensorlimitapp}
	If $\M^\bullet\in \mathrm{D}^-_\C(\mathscr{U}^\mathrm{op})$, $\N^\bullet\in \mathrm{D}^-_\C(\mathscr{U})$, then 
	\begin{equation*}
		\M^\bullet\widetilde{\otimes}^\mathbb{L}_{\mathscr{U}}\N^\bullet\cong \mathrm{holim} (\M_n^\bullet\widetilde{\otimes}^\mathbb{L}_{\mathscr{U}} \N^\bullet)\cong \mathrm{holim} (\M_n^\bullet\widetilde{\otimes}^\mathbb{L}_{\mathscr{U}_n} \N_n^\bullet)
	\end{equation*}
	in the sense that for each $m$, the natural morphisms yield a distinguished triangle
	\begin{equation*}
		\M^\bullet\widetilde{\otimes}^\mathbb{L}_{\mathscr{U}}\N^\bullet\to \prod_{n\geq m} \M_n^\bullet\widetilde{\otimes}^\mathbb{L}_{\mathscr{U}_n}\N_n^\bullet\to \prod_{n\geq m} \M_n^\bullet\widetilde{\otimes}^\mathbb{L}_{\mathscr{U}_n} \N_n^\bullet
	\end{equation*}
	in $\mathrm{D}(\mathrm{Shv}(X_m, LH(\h{\B}c_K)))$.
	
	The same holds for $\M^\bullet \in \mathrm{D}_{\C}^b(\mathscr{U}^{\mathrm{op}})$, $\N^\bullet\in \mathrm{D}_{\C}(\mathscr{U})$, and vice versa.
\end{cor}

\begin{proof}
	By Lemma \ref{Ccomplexesasholim} and the fact that derived functors are triangulated, we have a distinguished triangle
	\begin{equation*}
		\M^\bullet\widetilde{\otimes}^\mathbb{L}_{\mathscr{U}} \N^\bullet\to \M^\bullet\widetilde{\otimes}^\mathbb{L}_{\mathscr{U}} \prod_{n\geq m} \N_n^\bullet\to \M^\bullet\widetilde{\otimes}^\mathbb{L}_{\mathscr{U}} \prod_{n\geq m} \N_n^\bullet
	\end{equation*}
	in $\mathrm{D}(\mathrm{Shv}(X_m, LH(\h{\B}c_K)))$. Now apply Lemma \ref{tensorasholimapp}.
	
	If $\N^\bullet$ is bounded, then the uniform bound on the global dimension of $\mathscr{U}_n(X)$ means that $\prod \N_n^\bullet$ can be represented by a finite complex, where each term is of the form $\prod \mathcal{P}_n$, $\mathcal{P}_n$ the localisation of a finitely generated projective $\mathscr{U}_n(X)$-module. We can then perform the same argument as before even for unbounded $\M^\bullet\in \mathrm{D}_{\C}(\mathscr{U}^{\mathrm{op}})$.
\end{proof}

\subsection{Relative de Rham cohomology on the finite level}

Let $f: Z=X\times Y\to Y$ be the projection morphism as an in subsection 6.3. We now prove Lemma \ref{nbasechangelem}.

We first discuss suitable resolutions for $f^*\w{\D}_Y$ and $f^*\D_n$.

Note that if $U\subseteq X$, $V\subseteq Y$ are affinoid subdomains, then
\begin{equation*}
	f^*\w{\D}_Y(U\times V)=(\O_Z\widetilde{\otimes}_{f^{-1}\O_Y} f^{-1}\w{\D}_Y)(U\times V)=\O_X(U)\widetilde{\otimes}_K \w{\D}_Y(V)
\end{equation*}
and
\begin{equation*}
	f^*\D_n(U\times V)=(\O_Z\widetilde{\otimes}_{f{-1}\O_Y}f^{-1}\D_n)(U\times V)=\O_X(U)\widetilde{\otimes}_K \D_n(V).
\end{equation*}

Using the Spencer resolution from \cite[Theorem 6.12]{SixOp}, there exists a resolution of $f^*\w{\D}_Y$ as a coadmissible $\w{\D}_Z$-module by locally free modules of finite rank. Explicitly, there is a finite complex $\mathcal{S}^\bullet$ quasi-isomorphic to $f^*\w{\D}_Y$ with
\begin{equation*}
	\mathcal{S}^{-i}(U\times V)=(\w{\D}_X(U)\widetilde{\otimes}_{\O_X(U)} \wedge^i \T_X(U))\widetilde{\otimes}_K \w{\D}_Y(V)
\end{equation*}
for $i\geq 0$, and $\mathcal{S}^i=0$ for $i>0$.

It is straightforward to see that this is even a resolution of $f^*\w{\D}_Y$ as a $(\w{\D}_Z, f^{-1}\w{\D}_Y)$-bimodule.

In an analogous way, $f^*\D_n$ can be represented by a bounded complex of $(\w{\D}_Z, f^{-1}\D_n)$-bimodules $\mathcal{S}'^\bullet$ which is of the form
\begin{equation*}
	\mathcal{S}'^{-i}(U\times V)=(\w{\D}_X(U)\widetilde{\otimes}_{\O_X(U)}\wedge^i\T_X(U))\widetilde{\otimes}_K \D_n(V).
\end{equation*}

\begin{lem}
	\label{Spenceracyclic}
	Let $\M$ be a coadmissible right $\w{\D}_Z$-module. For each $i$, the left $\w{\D}_Z$-module $\mathcal{S}'^{-i}$ is $\M\widetilde{\otimes}_{\w{\D}_Z}$-acyclic.
\end{lem}
\begin{proof}
	Let $U=\Sp B\subseteq X$ be an affinoid subdomain with free tangent sheaf. Let $\B\subseteq B$ be an admissible affine formal model and let $\L'$ be an $(R, \B)$-Lie lattice in $\T_U(U)$. Then the closure $\C$ of $\B\otimes_R \A$ in $B\widetilde{\otimes}_K A$ is an admissible affine formal model for $U\times Y$ and
	\begin{equation*}
		\L_{m, n}:=(\C\otimes_{\B} \pi^m\L')\oplus(\C\otimes_{\A}\pi^n\L)
	\end{equation*} 
	is an $(R, \C)$-Lie lattice in $\T_Z(U\times Y)$ for any $m\geq 0$, as the derivations induced by $\L'$ commute with those induced by $\L$.
	
	Let $\D_{m, n}$ denote the monoid in $\mathrm{Shv}(U_m\times Y_n, LH(\h{\B}c_K))$ defined by this choice of lattice, so that
	\begin{equation*}
		\D_{m, n}(U\times Y)=\h{U_{\C}(\L_{m, n})}_K,
	\end{equation*}
	and let $\D_{\infty, n}=\varprojlim \D_{m, n}$, a monoid in $\mathrm{Shv}(U\times Y_n, LH(\h{\B}c_K))$ such that
	\begin{equation*}
		\D_{\infty, n}(U\times Y)=\varprojlim \D_{m, n}(U\times Y)
	\end{equation*}
	is a Fr\'echet--Stein $K$-algebra.
	
	Moreover,
	\begin{equation*}
		\D_{\infty, n}(U\times Y)\cong \w{\D}_X(U)\widetilde{\otimes}_K \D_n(Y) 
	\end{equation*}
	by \cite[Corollary 5.23]{SixOp}.
	
	The natural morphisms $\w{\D}_{U\times Y}\to \D_{m, n}$ induce a morphism $\w{\D}_{U\times Y}\to \D_{\infty, n}$.
	
	Note that as a left $\w{\D}_{U\times Y}$-module, $\mathcal{S}'^{-i}|_{U\times Y}$ is isomorphic to $\D_{\infty, n}^{\binom{d}{i}}$, where $d$ denotes the dimension of $X$.
	
	It thus suffices to show that
	\begin{equation*}
		\M|_{U\times Y}\widetilde{\otimes}^{\mathbb{L}}_{\w{\D}_{U\times Y}} \D_{\infty, n}\cong \M|_{U\times Y}\widetilde{\otimes}_{\w{\D}_{U\times Y}} \D_{\infty, n}.
	\end{equation*}
	But this follows directly from \cite[Proposition 5.37]{SixOp}.
\end{proof}

\begin{lem}
	\label{cohomvan}
	Let $\M$ be a coadmissible right $\w{\D}_Z$-module. Let $U\subseteq X$ be an affinoid subdomain with free tangent sheaf. For any $i$ and any $j>0$, we have
	\begin{equation*}
		\mathrm{H}^j(U\times Y, \M\widetilde{\otimes}_{\w{\D}_Z}\mathcal{S}^{-i})=0
	\end{equation*}
	and
	\begin{equation*}
		\mathrm{H}^j(U\times Y, \M\widetilde{\otimes}_{\w{\D}_Z} \mathcal{S}'^{-i})=0.
	\end{equation*}
\end{lem}
\begin{proof}
	By construction, 
	\begin{equation*}
		\M\widetilde{\otimes}_{\w{\D}_Z}\mathcal{S}^{-i}|_{U\times Y}\cong \M^{\binom{d}{i}}|_{U\times Y}
	\end{equation*}
	in $\mathrm{Shv}(U\times Y, LH(\h{\B}c_K))$. Since coadmissible modules have vanishing higher cohomology on affinoids, it follows that
	\begin{equation*}
		\mathrm{H}^j(U\times Y, \M\widetilde{\otimes}_{\w{\D}_Z} \mathcal{S}^{-i})=0
	\end{equation*}
	for all $j>0$.
	
	Similarly, we have seen above that
	\begin{equation*}
		\M\widetilde{\otimes}_{\w{\D}_Z}\mathcal{S}'^{-i}|_{U\times Z}\cong \M|_{U\times Y}\widetilde{\otimes}_{\w{\D}_{U\times Y}}\D_{\infty, n}^{\binom{d}{i}}
	\end{equation*}
	in $\mathrm{Shv}(U\times Y, LH(\h{\B}c_K))$. As
	\begin{equation*}
		\M|_{U\times Y}\widetilde{\otimes}_{\w{\D}_{U\times Y}} \D_{\infty, n}
	\end{equation*}
	is a coadmissible $\D_{\infty, n}$-module (compare \cite[Proposition 5.33]{SixOp}), it has again vanishing higher cohomology on affinoids.
\end{proof}

\begin{lem}
	\label{Cechbasechangeapp}
	Let $\M^\bullet\in \mathrm{D}^b_\C(\w{\D}_Z^\mathrm{op})$. Then the natural morphism
	\begin{equation*}
		\mathrm{R}\Gamma(X\times Y, \M^\bullet\widetilde{\otimes}^\mathbb{L}_{\w{\D}_Z}f^*\w{\D}_Y)\widetilde{\otimes}^\mathbb{L}_{\w{\D}_Y(Y)}\D_n(Y)\to \mathrm{R}\Gamma(X\times Y, \M^\bullet\widetilde{\otimes}^\mathbb{L}_{\w{\D}_Z}f^*\D_n)
	\end{equation*}
	is an isomorphism in $\mathrm{D}(\mathrm{Mod}_{Y_n, LH(\h{\B}c_K)}(\D_n(Y)^\mathrm{op}))$.
\end{lem}
\begin{proof}
	We first consider the case where $\M^\bullet=\M$ is a right coadmissible $\w{\D}_Z$-module. We compute both sides in terms of a Cech-de Rham double complex.
	
	By Lemma \ref{Spenceracyclic}, we have
	\begin{equation*}
		\M\widetilde{\otimes}_{\w{\D}_Z}^{\mathbb{L}}f^*\w{\D}_Y\cong \M\widetilde{\otimes}_{\w{\D}_Z}\mathcal{S}^\bullet
	\end{equation*}
	and
	\begin{equation*}
		\M\widetilde{\otimes}_{\w{\D}_Z}^{\mathbb{L}}f^*\D_n\cong \M\widetilde{\otimes}_{\w{\D}_Z}\mathcal{S}'^\bullet.
	\end{equation*}
	
	Let $\mathfrak{U}$ be a countable affinoid covering of $X$ such that $\T_X|_U$ is a free $\O_U$-module for each $U\in \mathfrak{U}$ (see \cite[p.212]{Boschlectures} for an argument how quasi-paracompactness implies a countable admissible covering for any connected component). By Lemma \ref{cohomvan}, we can invoke \cite[tag 0FLH]{stacksproj} to deduce that 
	\begin{equation*}
		\mathrm{R}\Gamma(X\times Y, \M\widetilde{\otimes}^\mathbb{L}_{\w{\D}_Z}f^*\w{\D}_Y)\cong \mathrm{Tot}(\check{C}^\bullet(\mathfrak{U}, \M\widetilde{\otimes}_{\w{\D}_Z} \mathcal{S}^\bullet))
	\end{equation*}
	and
	\begin{equation*}
		\mathrm{R}\Gamma(X\times Y, \M\widetilde{\otimes}_{\w{\D}_Z}^{\mathbb{L}}f^*\D_n)\cong \mathrm{Tot}(\check{C}^\bullet(\mathfrak{U}, \M\widetilde{\otimes}_{\w{\D}_Z}\mathcal{S}'^\bullet)).
	\end{equation*}
	
	We now wish to show that
	\begin{equation*}
		\mathrm{Tot}(\check{C}^\bullet(\mathfrak{U}, \M\widetilde{\otimes}_{\w{\D}_Z}\mathcal{S}^\bullet))\widetilde{\otimes}^\mathbb{L}_{\w{\D}_Y(Y)} \D_n(Y)\cong \mathrm{Tot}(\check{C}^\bullet(\mathfrak{U}, \M\widetilde{\otimes}_{\w{\D}_Z} \mathcal{S}'^\bullet)).
	\end{equation*}
	First note that if $U\subseteq X$ is a finite intersection of elements of $\mathfrak{U}$, then $U$ is an affinoid, as $X$ is assumed to be separated, and then
	\begin{equation*}
		(\M(U\times Y)\widetilde{\otimes}^{(\mathbb{L})}_{\w{\D}_Z(U\times Y)} \mathcal{S}^{-i}(U\times Y))\widetilde{\otimes}^\mathbb{L}_{\w{\D}_Y(Y)} \D_n(Y)
	\end{equation*}
	is isomorphic to 
	\begin{equation*}
		\M(U\times Y)\widetilde{\otimes}^\mathbb{L}_{\w{\D}_Z(U\times Y)}(\w{\D}_X(U)\widetilde{\otimes}_{\O_X(U)} \wedge^i\T_X(U)\widetilde{\otimes}^{(\mathbb{L})}_K \D_n(Y)),
	\end{equation*}
	which is isomorphic to
	\begin{equation*}
		\M(U\times Y)\widetilde{\otimes}^\mathbb{L}_{\w{\D}_Z(U\times Y)}  \D_{\infty, n}(U\times Y)^{\binom{d}{i}}\\
		\cong \M(U\times Y)\widetilde{\otimes}_{\w{\D}_Z(U\times Y)}^{\mathbb{L}} \mathcal{S}'^{-i}(U\times Y),
	\end{equation*}
	which we have already seen to be concentrated in degree $0$. Thus by Lemma \ref{flatprod}, $\check{C}^j(\mathfrak{U}, \M\widetilde{\otimes}_{\w{\D}_Z}\mathcal{S}^k)$ is $-\widetilde{\otimes}_{\w{\D}_Y(Y)}\D_n(Y)$-acyclic for any $j$ and $k$, so that 
	\begin{equation*}
		\mathrm{Tot}(\check{C}^\bullet(\mathfrak{U}, \M\widetilde{\otimes}_{\w{\D}_Z} \mathcal{S}^\bullet))\widetilde{\otimes}^\mathbb{L}_{\w{\D}_Y(Y)} \D_n(Y)\cong \mathrm{Tot}(\check{C}^\bullet(\mathfrak{U}, \M\widetilde{\otimes}_{\w{\D}_Z} \mathcal{S}^\bullet))\widetilde{\otimes}_{\w{\D}_Y(Y)} \D_n(Y).
	\end{equation*} 
	
	But Lemma \ref{flatprod} together with the calculation above implies then that 
	\begin{equation*}
		\check{C}^j(\mathfrak{U}, \M\widetilde{\otimes}_{\w{\D}_Z}\mathcal{S}^k)\widetilde{\otimes}_{\w{\D}_Y(Y)} \D_n(Y)\cong \check{C}^j(\mathfrak{U}, \M\widetilde{\otimes}_{\w{\D}_Z} \mathcal{S}'^k),
	\end{equation*}
	and the result follows.
	
	The isomorphism immediately generalizes to the case where $\M^\bullet$ is a bounded $\C$-complex.
\end{proof}

\end{document}